\documentclass[12pt]{article}
\usepackage{amsmath,amscd,amssymb,amsthm,mathrsfs,amsfonts,layout,indentfirst,graphicx,caption,mathabx,  tikz, stmaryrd,appendix}
\usepackage{palatino}  
\usepackage{slashed} 
\usepackage{mathrsfs} 

\usepackage{lipsum}
\let\OLDthebibliography\thebibliography
\renewcommand\thebibliography[1]{
	\OLDthebibliography{#1}
	\setlength{\parskip}{0pt}
	\setlength{\itemsep}{2pt} 
}

\allowdisplaybreaks  
\usepackage{latexsym}
\usepackage{chngcntr}
\usepackage[colorlinks,linkcolor=blue,anchorcolor=blue,linktocpage]{hyperref}
\hypersetup{citecolor=[rgb]{0,0.5,0}}

\usepackage{fullpage}
\counterwithin{figure}{section}

\theoremstyle{definition}
\newtheorem{df}{Definition}[section]

\newtheorem{rem}[df]{Remark}

\theoremstyle{plain}
\newtheorem{thm}[df]{Theorem}
\newtheorem{pp}[df]{Proposition}
\newtheorem{co}[df]{Corollary}
\newtheorem{lm}[df]{Lemma}
\newtheorem{cond}{Condition}

\newcommand{\tr}{\mathrm{t}} 
\newcommand{\End}{\mathrm{End}} 
\newcommand{\id}{\mathrm{id}}
\newcommand{\Hom}{\mathrm{Hom}}
\newcommand{\Conf}{\mathrm{Conf}}

\newcommand{\ev}{\mathrm{ev}}
\newcommand{\coev}{\mathrm{coev}}

\newcommand{\Rep}{\mathrm{Rep}}
\newcommand{\diag}{\mathrm{diag}}

\newcommand{\uni}{\mathrm{u}}
\newcommand{\di}{\slashed d}
\newcommand{\Diff}{\mathrm{Diff}}
\newcommand{\PSU}{\mathrm{PSU}}
\newcommand{\Vir}{\mathrm{Vir}}
\newcommand{\Span}{\mathrm{Span}}
\newcommand{\pri}{\mathrm{p}}
\newcommand{\E}{E^1(V)}
\newcommand{\mc}{\mathcal}
\newcommand{\ER}{E^1(V)_{\mathbb R}}

\numberwithin{equation}{section}

\title{Unitarity of the modular tensor categories associated to unitary vertex operator algebras, II}
\author{{\sc Bin Gui}\\
	{\small Department of Mathematics, Vanderbilt University}\\
	{\small binguimath@gmail.com}
}
\date{}
\begin{document}\sloppy 
	\pagenumbering{arabic}
	\setcounter{section}{0}
	\setcounter{equation}{6}

	\maketitle

\begin{abstract}
This is the second part in a two-part series of papers constructing a unitary structure for the modular tensor category (MTC) associated to a unitary rational vertex operator algebra (VOA). We define, for a unitary rational vertex operator algebra $V$, a non-degenerate sesquilinear form $\Lambda$ on each vector space of intertwining operators. We give two sets of criteria for the positivity of $\Lambda$, both concerning the energy bounds condition of vertex operators and intertwining operators. These criteria can be applied to many familiar examples, including unitary Virasoro VOAs, unitary affine VOAs of type $A$, $D$, and more. Having shown that $\Lambda$ is an inner product, we prove that $\Lambda$ induces a unitary structure on the modular tensor category of $V$.
\end{abstract}

	\tableofcontents
\section*{Introduction}

In this article, we continue the study initiated in  \cite{Gui19a} of the unitarity structure on the modular tensor category (MTC) associated to a unitary ``rational" vertex operator algebra (VOA) $V$. The precise meaning of the word ``rational" will be given in the notation list.  We will follow the notations and conventions in \cite{Gui19a}. For simplicity, we assume that all $V$-modules are unitarizable. To prove the unitarity of the modular tensor category  of $V$, we  need to define, for any unitary $V$-modules $W_i,W_j$,  a canonical unitary structure  on the unitarizable $V$-module $W_i\boxtimes W_j$.

We first recall the definition of $W_i\boxtimes W_j$. For each equivalence class $[W_k]$ of irreducible unitary $V$-module, we choose a representing element $W_k$, and let $\mathcal E$ be the set of these $W_k$'s. Then the cardinality of $\mathcal E$ is finite by the rationality of $V$. The tensor product of $W_i,W_j$ is defined to be
\begin{align*}
W_{ij}\equiv W_i\boxtimes W_j=\bigoplus_{k\in\mathcal E}\mathcal V{k \choose i~j}^*\otimes W_k.
\end{align*}
where $\mathcal V{k\choose i~j}^*$ is the dual vector space of $\mathcal V{k\choose i~j}$, the finiteness of the dimension of which follows also from the rationality of $V$. The action of $V$ on $W_i\boxtimes W_j$ is $\bigoplus_k \id\otimes Y_k$, where $Y_k$ is the vertex operator describing the action of $V$ on $W_k$. Since the unitary structure on each $W_k$ is already chosen, to determine a suitable unitary on $W_i\boxtimes W_j$, it suffices to find such an inner product $\Lambda$ on $\mathcal V{k\choose i~j}^*$ for any $k\in\mathcal E$.

\subsubsection*{Definition of $\Lambda$} 

We first define $\Lambda=\Lambda(\cdot|\cdot)$ as a sesquilinear form on $\mathcal V{k\choose i~j}^*$, antilinear on the second variable. Recall from part I that for any $\mathcal Y_\alpha\in\mathcal V{k\choose i~j}$, its adjoint intertwining operator $\mathcal Y_{\alpha^*}\in \mathcal V{j\choose \overline i~k}$ is defined to satisfy that for any $w^{(i)}\in W_i$, 
\begin{align*}
\mathcal Y_{\alpha^*}(\overline{w^{(i)}},x)=\mathcal Y_\alpha(e^{xL_1}(e^{-i\pi}x^{-2})^{L_0}w^{(i)},x^{-1})^\dagger.
\end{align*}
Here $\dagger$ is the formal adjoint operation, which means that for any $w^{(j)}\in W_j,w^{(k)}\in W_k$, 
\begin{align*}
\langle \mathcal Y_{\alpha^*}(\overline{w^{(i)}},x)w^{(k)}|w^{(j)}  \rangle=\langle w^{(k)}|\mathcal Y_\alpha(e^{xL_1}(e^{-i\pi}x^{-2})^{L_0}w^{(i)},x^{-1})w^{(j)} \rangle.
\end{align*}
The creation operator $\mathcal Y^i_{i0}\in\mathcal V{i\choose i~0}$ of $W_i$ is defined so that for any $v\in V$ and $w^{(i)}\in W_i$, 
\begin{align*}
\mathcal Y^i_{i0}(w^{(i)},x)v=e^{xL_{-1}}Y_i(v,-x)w^{(i)}.
\end{align*}
The annihilation operator $\mathcal Y^0_{\overline ii}$ is defined to be the adjoint intertwining operator of $\mathcal Y^i_{i0}$, which is of type $0\choose \overline i~i$.

We now choose a basis $\{\mathcal Y_\alpha:\alpha\in\Theta^k_{ij} \}$ of the vector space $\mathcal V{k\choose i~j}$, and let $\{\widecheck{\mathcal Y}^\alpha:\alpha\in\Theta^k_{ij} \}$ be its dual basis in $\mathcal V{k\choose i~j}^*$. By fusion relations, there exits, for each $k\in\mathcal E$, a complex matrix $\Lambda=\{\Lambda^{\alpha\beta} \}_{\alpha,\beta\in\Theta^k_{ij}}$, such that for any $z_1,z_2\in\mathbb C^\times$ satisfying $0<|z_2-z_1|<|z_1|<|z_2|$ and $\arg z_1=\arg z_2=\arg (z_2-z_1)$, we have

\begin{align}
Y_j\big(\mathcal Y^0_{\overline ii}(\overline {w^{(i)}_2},z_2-z_1)w^{(i)}_1,z_1\big)=\sum_{k\in\mathcal E}\sum_{\alpha,\beta\in\Theta^k_{ij}}\Lambda^{\alpha\beta}\mathcal Y_{\beta^*}(\overline{w^{(i)}_2},z_2)\mathcal Y_{\alpha}(w^{(i)}_1,z_1).\label{eq365}
\end{align}
The fusion relation \eqref{eq365} is called \emph{transport formula}, and the matrix $\Lambda$ is called \emph{transport matrix}.\footnote{The primitive form of $\Lambda$ appeared in \cite{Wassermann}. See the discussion in remark \ref{lb30}.} We then define a sesquilinear form $\Lambda(\cdot|\cdot)$ on $\mathcal V{k\choose i~j}^*$ to satisfy that
\begin{align}
\Lambda( \widecheck{\mathcal Y}^{\alpha}|\widecheck{\mathcal Y}^{\beta}  )=\Lambda^{\alpha\beta}
\end{align}
for any $\alpha,\beta\in\Theta^k_{ij}$. It is easy to see that this definition does not depend on the basis $\Theta^k_{ij}$ chosen.

The sesquilinear forms $\Lambda$ on the dual vector spaces of intertwining operators induce a sesquilinear form on $W_i\boxtimes W_j$, also denoted by $\Lambda$.  If one can prove that $\Lambda$ is positive, then by the rigidity of the MTC of $V$, one can easily show that $\Lambda$ is also non-degenerate. Therefore $\Lambda$ becomes an inner product. It will then be  not hard to show, as we shall see in chapter 7, that the MTC of $V$ is unitary. However, the difficulty lies exactly in the proof of  the positivity of $\Lambda$.

\subsubsection*{Positivity of $\Lambda$}

To prove the positivity of $\Lambda$,  we will use several analytic conditions on VOAs and their intertwining operators, so that certain results and techniques in conformal nets can be applied here. In this paper, we assume that the unitary rational VOA $V$ satisfies  condition \ref{CondA} or \ref{CondB}. The precise statements of these two conditions are in section \ref{Condition ABC}. Here we make a brief and somewhat simplified description. Both conditions require that $V$ is strongly local \cite{CKLW}, so that we can construct a conformal net $\mathcal M_V$  using smeared vertex operators from $V$. Both require that there exists a ``generating set"  of intertwining operators, which means, more precisely, that there exists a set $\mathcal F=\{W_i:i\in\mathcal F \}$ of irreducible $V$-modules such that any irreducible $V$-module is a submodule of a tensor product of these modules (i.e., the set $\mathcal F$ generate the monoidal category of $V$), and that for any $i\in\mathcal F$ and any irreducible $V$-module $W_j,W_k$, the type $k\choose i~j$ intertwining operators are energy-bounded.\footnote{By proposition 3.4, if an intertwining operator is energy bounded, then so is its adjoint intertwining operator. Therefore, it suffices to require that $\mathcal F\cup\overline{\mathcal F}$, instead of $\mathcal F$, generates the monoidal category of $V$, where $\overline{\mathcal F}=\{W_{\overline i}:i\in\mathcal F \}$. Moreover, we are also interested in the general case where $\mathcal F\cup\overline{\mathcal F}$ generates, not the whole tensor category of $V$, but a smaller tensor subcategory. In this case, this tensor subcategory might not be modular, but only a ribbon fusion category.  In section \ref{Condition ABC}, the statement of conditions \ref{CondA} and \ref{CondB} will take care of this general case. } The difference between these two conditions is that condition A requires that there exists a generating set of quasi-primary vectors in $V$ whose vertex operators satisfy $1$-st order energy bounds, whereas in condition B, the  $1$-st order energy bounds condition is required for intertwining operators. The reason we need $1$-st order energy bounds condition on either vertex operators or intertwining operators is to guarantee the \emph{strong intertwining property}: causally disjoint smeared vertex operators and smeared intertwining operators commute, not only when acting on  a common invariant core of them, but also strongly in the sense that the von Neumann algebras they generate commute. (See proposition \ref{lb106}.)

We now sketch the proof of the positivity of $\Lambda$ on $W_i\boxtimes W_j$. We first assume that $i,j\in\mathcal F$. So any intertwining operator whose charge space is $W_i$ or $W_j$ is energy-bounded. In this special case, our proof is modeled on Wassermann's argument in \cite{Wassermann}. We first recall some  algebraic and analytic properties of smeared intertwining operators proved in part I. Choose disjoint open intervals $I,J\in S^1\setminus\{-1\}$.\\
(a) \emph{Braiding}: If $\mathcal Y_\alpha$ and $\mathcal Y_{\alpha'}$ are (energy bounded) intertwining operators with charge space $W_i$,  $\mathcal Y_\beta$ and $\mathcal Y_{\beta'}$ are (energy bounded) intertwining operators with charge space $W_j$, and for any $w^{(i)}\in W_i,w^{(j)}\in W_j,z_1\in I,z_2\in J$ we have the braid relation
\begin{align}
\mathcal Y_\alpha(w^{(i)},z_1)\mathcal Y_\beta(w^{(j)},z_2)=\mathcal Y_{\beta'}(w^{(j)},z_2)\mathcal Y_{\alpha'}(w^{(i)},z_1),
\end{align}
Then the smeared intertwining operators also have braiding
\begin{align}
\mathcal Y_\alpha(w^{(i)},f)\mathcal Y_\beta(w^{(j)},g)=\mathcal Y_{\beta'}(w^{(j)},g)\mathcal Y_{\alpha'}(w^{(i)},f)
\end{align}
for any $f\in C^\infty_c(I),g\in C^\infty_c(J)$.\\
(b) \emph{Adjoint relation}: If $w^{(i)}$ is quasi-primary with conformal dimension $\Delta_{w^{(i)}}$, then
\begin{align}
\mathcal Y_\alpha(w^{(i)},f)^\dagger= {e^{-i\pi\Delta_{w^{(i)}}}} \mathcal Y_{\alpha^*}(\overline{w^{(i)}},\overline {e_{(2-2\Delta_{w^{(i)}})}f}),
\end{align}
where for each $r\in\mathbb R$, $e_r$ is a function on $S^1\setminus\{-1\}$ satisfying  $e_r(e^{i\theta})=e^{ir\theta}$ ($-\pi<\theta<\pi$).\\
(c) \emph{Strong intertwining property}: If $\mathcal Y_\alpha\in\mathcal V{k\choose i~l}$, then for any $f\in C^\infty_c(I)$ and $y\in\mathcal M_V(J)$, 
\begin{align}
\pi_k(y)\mathcal Y_\alpha(w^{(i)},f)\subset \mathcal Y_\alpha(w^{(i)},f)\pi_l(y).
\end{align}
Here $\pi_k$ and $\pi_l$ are the representations of $\mathcal M_V$ on $\mathcal H_k,\mathcal H_l$ integrated from the $V$-modules $W_k,W_l$ respectively, the existence of which will be discussed in section \ref{lb300} (see also \cite{CWX}).\\
(d) \emph{Rotation covariance}: If  $w^{(i)}\in W_i$ is a homogeneous vector  with conformal weight $\Delta_{w^{(i)}}$, and for $t\in\mathbb R$, we define $\mathfrak r(t)f\in C^\infty(S^1)$  to satisfy $(\mathfrak r(t)f)(e^{i\theta})=f(e^{i(\theta-t)})$, then
\begin{gather}
e^{it{{\overline {L_0}}}}\mathcal Y_\alpha(w^{(i)},f)e^{-it{\overline{ L_0}}}=\mathcal Y_\alpha\big(w^{(i)},e^{i(\Delta_{w^{(i)}}-1)t}\mathfrak r(t)f\big).
\end{gather}

Now the positivity of $\Lambda$ on $W_i\boxtimes W_j$ can be argued as follows.  Write $W_{ij}=W_i\boxtimes W_j$, and define a type ${ij\choose i~j}={W_i\boxtimes W_j\choose W_i~W_j}$ intertwining operator $\mathcal Y_{i\boxtimes j}$, such that for any $w^{(i)}\in W_i,w^{(j)}\in W_j$,
\begin{align}\label{eq400}
\mathcal Y_{i\boxtimes j}(w^{(i)},x)w^{(j)}=\sum_{k\in\mathcal E}\sum_{\alpha\in\Theta^k_{ij}}\widecheck{\mathcal Y}^{\alpha}\otimes\mathcal Y_{\alpha}(w^{(i)},x)w^{(j)}.
\end{align}
This definition is independent of the basis $\Theta^k_{ij}$ chosen. Using rotation covariance and lemma \ref{lb15}, we can show that the vectors of the form
\begin{align}
\xi=\sum_{s=1}^N\mathcal Y_{i\boxtimes j}(w^{(i)}_s,f_s)\cdot\mathcal Y^j_{j0}(w^{(j)}_s,g_s)\Omega\label{eq368}
\end{align}
form a dense subspace of the norm closure $\mathcal H_{ij}$ of $W_{ij}=W_i\boxtimes W_j$, where $N=1,2,\dots,f_1,\dots,f_N\in C^\infty_c(I),g_1,\dots,g_N\in C^\infty_c(J)$, $w^{(i)}_1,\dots,w^{(i)}_N\in W_i,w^{(j)}_1,\dots,w^{(j)}_N\in W_j$ are quasi-primary, and $\Omega$ is the vacuum vector in the vacuum module $W_0=V$. Therefore, to prove the positivity of the sesquilinear form $\Lambda$ on $W_i\boxtimes W_j$, it suffices to prove that $\Lambda(\xi|\xi)\geq0$ for all such $\xi$. Note that the transport formula \eqref{eq365} can be written in the form of braiding:
	\begin{align}
&\mathcal Y^j_{j0}(w^{(j)},z_0)\mathcal Y^0_{\overline ii}(\overline{w^{(i)}_2},z_2)\mathcal Y^i_{i0}(w^{(i)}_1,z_1)\nonumber\\
=&\bigg(\sum_{k\in\mathcal E}\sum_{\alpha,\beta\in\Theta^*_{ij}}\Lambda^{\alpha\beta}\mathcal Y_{\beta^*}(\overline{w^{(i)}_2},z_2)\mathcal Y_{\alpha}(w^{(i)}_1,z_1)\bigg)\mathcal Y^j_{j0}(w^{(j)},z_0).
\end{align}
for any $z_1,z_2\in I,z_0\in J,w^{(i)}_1,w^{(i)}_2\in W_i,w^{(j)}\in W_j$. Then, using the smeared version of this braid relation, together with the adjoint relation, one is  able to compute that
\begin{align}
\Lambda(\xi|\xi)=\sum_{1\leq s,t\leq N}\langle \mathcal Y^j_{j0}(w^{(i)}_t,f_t)^\dagger\mathcal Y^j_{j0}(w^{(i)}_s,f_s)\mathcal Y^i_{i0}(w^{(j)}_t,g_t)^\dagger\mathcal Y^i_{i0}(w^{(j)}_s,g_s)\Omega|\Omega \rangle.\label{eq366}
\end{align}
By the strong intertwining property, the right hand side of equation \eqref{eq366} can be approximated by
\begin{align}
\sum_{1\leq s,t\leq N}\langle \mathfrak A^*_t\mathfrak A_s\mathfrak B^*_t\mathfrak B_s\Omega|\Omega \rangle,\label{eq367}
\end{align}
where for each $s,t$,  $\mathfrak A_s\in\Hom_{\mathcal M_V(I^c)}(\mathcal H_0,\mathcal H_i)$ and $\mathfrak B_t\in\Hom_{\mathcal M_V(J^c)}(\mathcal H_0,\mathcal H_j)$ are bounded operators, where $I^c$ and $J^c$ are the interiors of the complements of $I$ and $J$ in $S^1$ respectively, and $\mathcal H_0,\mathcal H_i,\mathcal H_j$ are the $\mathcal M_V$-modules integrated from $W_0,W_i,W_j$ respectively. But \eqref{eq367} equals
\begin{align}
\bigg\lVert \sum_{1\leq s\leq N}\mathfrak A_s\Omega\boxtimes \mathfrak B_s\Omega \bigg\lVert^2,
\end{align}
where $\boxtimes$ is the \emph{Connes fusion product} \cite{Con80}. So $\Lambda(\xi|\xi)$ must be non-negative, and hence $\Lambda$ is positive on $W_i\boxtimes W_j$.

\subsubsection*{Generalized (smeared) intertwining operators}

Now consider the general case when $W_i$ and $W_j$ are not necessarily in $\mathcal F$. Then the energy bounds condition of the intertwining operators with charge spaces $W_i$ or $W_j$ is unknown. But since $\mathcal F$ generates the monoidal category of $V$, there exist $W_{i_1},\dots,W_{i_m}\in\mathcal F$ such that $W_i$ is equivalent to a submodule of $W_{i_m}\boxtimes\cdots\boxtimes W_{i_1}$. Equivalently, there exist  irreducible $V$-modules $W_{r_2},\dots,W_{r_{m-1}}$ and non-zero (energy-bounded) intertwining operators $\mathcal Y_{\sigma_2}\in\mathcal V{r_2\choose i_2~i_1 },\mathcal Y_{\sigma_3}\in\mathcal V{r_3\choose i_3~r_2},\dots\mathcal Y_{\sigma_m}\in\mathcal V{i\choose i_m~r_{m-1}}$. Choose mutually disjoint open intervals $I_1,I_2,\dots,I_m\subset I$. Then we define a \emph{generalized intertwining operator} $\mathcal Y_{\sigma_n\cdots\sigma_2,\alpha}$ (acting on the source space $W_j$ of $\mc Y_\alpha$), such that for any $z_1\in I_1,z_2\in I_2,\dots,z_m\in I_m,w^{(i_1)}\in W_{i_1},w^{(i_2)}\in W_{i_2},\dots,w^{(i_m)}\in W_{i_m}$,
\begin{align}
&\mathcal Y_{\sigma_n\cdots\sigma_2,\alpha}(w^{(i_n)},z_n;\dots;w^{(i_2)},z_2;w^{(i_1)},z_1)\nonumber\\
=& \mathcal Y_\alpha\big(\mathcal Y_{\sigma_m}(w^{(i_n)},z_n-z_1)\cdots\mathcal Y_{\sigma_2}(w^{(i_2)},z_2-z_1)w^{(i_1)},z_1\big).
\end{align}
Now, for each $f_1\in C^\infty_c(I_1),\dots,f_m\in C^\infty_c(I_m)$, we define a \emph{generalized smeared interwining operator}
\begin{align*}
&\mathcal Y_{\sigma_m\cdots\sigma_2,\alpha}(w^{(i_m)},f_m;\dots;w^{(i_1)},f_1)\nonumber\\
=&\int_{-\pi}^{\pi}\cdots \int_{-\pi}^{\pi} \mathcal Y_{\sigma_m\cdots\sigma_2,\alpha}(w^{(i_m)},e^{i\theta_m};\dots;w^{(i_1)},e^{i\theta_1})\cdot f_1(e^{i\theta_1})\cdots f_m(e^{i\theta_m})\di\theta_1\cdots\di\theta_m,
\end{align*}
where $\di\theta=e^{i\theta}d\theta/2\pi$. We can also define, for any $\mathcal Y_{\beta}$ with charge space $W_j$, mutually disjoint $J_1,\dots,J_n\subset J$, and $g_1\in C^\infty_c(J_1),\dots,g_n\in C^\infty_c(J_n)$, a generalized smeared intertwining operator $\mathcal Y_{\rho_n\cdots\rho_2,\beta}(w^{(j_n)},g_n;\dots;w^{(j_1)},g_1)$ in a similar way. Then we can prove the positivity of $\Lambda$ on $W_i\boxtimes W_j$ in a similar way as above, once we've established the braid relations, the adjoint relation, the strong intertwining property, and the rotation covariance of generalized smeared intertwining operators. The last two properties follow easily from those of the smeared intertwining operators. The first two, especially the adjoint relation, is much harder to prove. All these properties will be treated in this paper.

\subsubsection*{Outline of this paper}

In chapter 4 we discuss some of the relations between a strongly-local unitary rational VOA $V$ and its conformal net $\mathcal M_V$. A unitary $V$-module $W_i$ is called strongly integrable, if it can be integrated to an $\mathcal M_V$-module $\mathcal H_i$. In section 4.1, we study the relation between the abelian category of strongly integrable unitary $V$-modules and the one of $\mathcal M_V$-modules. In section 4.2  we give a useful criterion for the strong integrability of unitary $V$-modules  based on the energy bounds condition of intertwining operators. 

Chapter 5 is devoted to the study of generalized intertwining operators and their smeared versions. In particular, we prove (in section 5.3) the braid relations, the adjoint relation, the strong intertwining property, and the rotation covariance of  generalized smeared intertwining operators. Since the proof of the first two properties are harder, and since the difficulty is mostly on the unsmeared side, in sections 5.1 and 5.2 we prove the braid relations and the adjoint relation of  generalized (unsmeared) intertwining operators.

Chapter 6 is the climax of our  series of papers. Recall that for any $w^{(i)}\in W_i$ and $z\in\mathbb C\setminus\{0\}$, $\mathcal Y_{i\boxtimes j}(w^{(i)},z)W_j$ is dense in the algebraic completion $\widehat W_{ij}$ of $W_{ij}=W_i\boxtimes W_j$ by proposition A.3 in part I, where  $\mathcal Y_{i\boxtimes j}$ is defined by equation \eqref{eq400}. In section 6.1 we first prove a similar density result for generalized intertwining operators. Its smeared version is also proved with the help of rotation covariance. In section 6.2 we define the sesquilinear form $\Lambda$. The positive-definiteness of $\Lambda$ is proved in section 6.3 using all the  results we have achieved so far. 

In chapter 7 we show that $\Lambda$ defines a unitary structure on the MTC of $V$. Our result is applied to unitary Virasoro VOAs and many unitary affine VOAs in sections 8.1 and 8.2. The sesquilinear form $\Lambda$ is closely related to the non-degenerate bilinear form considered in \cite{HK07}. We will explain this relation in section 8.3.

\subsubsection*{Acknowledgment}
This research is supported by NSF grant DMS-1362138. I am grateful to my advisor, professor Vaughan Jones, for his consistent support throughout this project.

\subsubsection*{Conventions and Notations}
We follow the notations and the conventions  in part I \cite{Gui19a} (see conventions 1.12, 2.1, 2.19, and definition 1.13).  In particular, if $s\in\mathbb R$, we always assume $\arg e^{is}=s$. If $z\in\mathbb C^\times$ and $\arg z$ is chosen, then we let $\arg{\overline z}=-\arg z$ and $\arg z^s=s\arg z$. If $z_1,z_2\in\mathbb C^\times$ and $\arg z_1,\arg z_2$ are chosen, then we set $\arg(z_1z_2)=\arg z_1+\arg z_2$. We also understand $z_1/z_2$ as $z_1z_2^{-1}$. Therefore $\arg(z_1/z_2)=\arg(z_1z_2^{-1})=\arg z_1+\arg z_2^{-1}=\arg z_1-\arg z_2$. 

Arguments of explicit  positive \emph{real} numbers (e.g. $1,\sqrt 2,\pi$) are assumed to be $0$ unless otherwise stated. (This is also assumed but not explicitly mentioned in \cite{Gui19a}.)  In this article as well as \cite{Gui19a},  symbols like $r,s,t$ (sometimes with subscripts) are used as \textbf{real variables}. When they take positive values we also assume their arguments  to be  $0$. $e^{2i\pi}$ is regarded  not as a positive real number but as a positive \emph{complex} number. As a consequence, its argument is not $0$ but  $2\pi$. Symbols like $z,\zeta$ are regarded as \textbf{complex variables}. Therefore, even if they can take positive values, their arguments are still not necessarily $0$. 

Now assume that $U$ is an open subset of $\mathbb C$, $f:U\rightarrow \mathbb C^\times$ is continuous, $z_1,z_2\in U$, and the interval $E$ connecting $z_1,z_2$ is inside $U$. Choose arguments $\arg f(z_1),\arg f(z_2)$ of $f(z_1),f(z_2)$. Following definition 1.13, we say that $\arg f(z_2)$ is \textbf{close to  $\arg f(z_1)$ as $z_2\rightarrow z_1$}, if there exists a (unique) continuous function $A:[0,1]\rightarrow \mathbb R$ satisfying that: (a) $A(0)=\arg z_1,A(1)=\arg z_2$. (b) $A(t)$ is an argument of $f(tz_1+(1-t)z_2)$ for any $t\in[0,1]$.

We always assume $V$ to be  a VOA of CFT type satisfying the following  condition:
\begin{flalign*}
&(\alpha)\text{ $V$ is isomorphic to $V'$.}&\\
&(\beta)\text{ Every $\mathbb N$-gradable weak $V$-module is completely reducible.}&\\
&(\gamma)\text{ $V$ is $C_2$-cofinite.}&
\end{flalign*}
The precise meanings of these conditions, which are not quite important to our theory, can be found in \cite{H MI}. The importance of these conditions is to guarantee the existence of a MTC of $V$ due to \cite{H Rigidity}.  

The following notation list, which is an expansion of the one in part I, is used throughout this paper.\\

\noindent
$A^\tr$: the transpose of the linear operator $A$.\\
$A^\dagger$: the formal adjoint of the linear operator $A$.\\
$A^*$: the ajoint of the possibly unbounded linear operator $A$.\\
$\overline A$: the closure of the pre-closed linear operator $A$.\\
$C_i$: the antiunitary map $W_i\rightarrow W_{\overline i}$.\\
$\mathbb C^\times=\{z\in\mathbb C:z\neq0 \}$.\\
$\Conf_n(\mathbb C^\times)$: the $n$-th configuration space of $\mathbb C^\times$.\\
$\mathscr D(A)$: the domain of the possibly unbounded operator $A$.\\
$\di\theta=  \frac {e^{i\theta}}{2\pi}d\theta$.\\
$e_r(e^{i\theta})=e^{ir\theta}\quad(-\pi<\theta<\pi)$.\\
$E^1(W_i),E^1(V)$: see section \ref{Condition ABC}.\\
$\ER$: the real subspace of the vectors in $\E$ that are fixed by the CPT operator $\theta$.\\
$\mathcal E$: a complete list of mutually inequivalent irreducible $V$-modules.\\
$\mathcal E^\uni$: the set of unitary $V$-modules in $\mathcal E$.\\
$\mathcal F$: a non-empty set of non-zero irreducible unitary $V$-modules.\\
$\mathcal F^\boxtimes$: see section \ref{Condition ABC}.\\
$\Hom_V(W_i,W_j)$: the vector space of $V$-module homomorphisms from $W_i$ to $W_j$.\\
$\Hom_{\mathcal M_V}(\mathcal H_i,\mathcal H_j)$: the vector space of bounded $\mathcal M_V$-module homomorphism from $\mathcal H_i$ to $\mathcal H_j$.\\
$\Hom_{\mathcal M_V(I)}(\mathcal H_i,\mathcal H_j)$: the vector space of bounded operators $\mathcal H_i\rightarrow H_j$ intertwining the action of elements in $\mathcal M_V(I)$.\\
$\mathcal H_i$: the norm completion of the vector space $W_i$. If the $V$-module $W_i$ is a unitary, energy-bounded, and strongly integrable, then $\mathcal H_i$ is the  $\mathcal M_V$-module associated with $W_i$.\\
$\mathcal H^r_i$: the vectors of $\mathcal H_i$ that are inside $\mathscr D((1+\overline{L_0})^r)$.\\
$\mathcal H^\infty_i=\bigcap_{r\geq0}\mathcal H^r_i$.\\
$i$: either the index of a $V$-module, or $\sqrt{-1}$.\\
$I^c=S^1\setminus\overline I$.\\
$I_1\subset\joinrel\subset I_2$: $I_1,I_2\in\mathcal J$ and $\overline {I_1}\subset I_2$.\\
$\id_i=\id_{W_i}$: the identity operator of $W_i$.\\
$\mathcal J$: the set of (non-empty, non-dense) open intervals of $S^1$.\\
$\mathcal J(U)$: the set of open intervals of $S^1$ contained in the open set $U$.\\
$\mathcal M_V$: the conformal net constructed from $V$.\\
$\mathcal M_V(I)_\infty$: the set of smooth operators in $\mathcal M_V(I)$.\\
$\mathscr O_n(I)$: see the beginning of chapter \ref{lb112}.\\
$\mathfrak O_n(I)$: see section \ref{Condition ABC}.\\
$P_s$: the projection operator of $W_i$ onto $W_i(s)$.\\
$\mathfrak r(t):S^1\rightarrow S^1$: $\mathfrak r(t)(e^{i\theta})=e^{i(\theta+t)}$.\\
$\mathfrak r(t):C^\infty(S^1)\rightarrow C^\infty(S^1)$: $\mathfrak r(t)h=h\circ\mathfrak r(-t)$.\\
$\Rep(V)$: the modular tensor category of the representations of $V$.\\
$\Rep^\uni(V)$: the category of the unitary representations of $V$.\\
$\Rep^\uni_{\mathcal G}(V)$: When $\mathcal G$ is additively closed, it is the subcategory of $\Rep^\uni(V)$ whose objects are unitary $V$-modules in $\mathcal G$. When $\mathcal G$ is multiplicatively closed, then it is furthermore equipped with the structure of a ribbon tensor category.\\
$S^1=\{z\in\mathbb C:|z|=1 \}$.\\
$\mathcal S$: the collection of strongly integrable energy-bounded unitary $V$-modules.\\
$\mathcal V{k\choose i~j}$: the vector space of type $k\choose i~j$ intertwining operators.\\
$W_0=V$, the vacuum module of $V$.\\
$W_i$: a $V$-module.\\
$\widehat W_i$: the algebraic completion of $W_i$.\\
$W_{\overline i}\equiv W'_i$: the contragredient module of $W_i$.\\
$W_{ij}\equiv W_i\boxtimes W_j$: the tensor product of $W_i,W_j$.\\
$w^{(i)}$: a vector in $W_i$.\\
$\overline {w^{(i)}}=C_iw^{(i)}.$\\
$x$:  a formal variable, or an element inside $\mathcal M_V(I)$.\\
$Y_i$: the vertex operator of $W_i$.\\
$\mathcal Y_\alpha$: an intertwining operator of $V$.\\
$\mathcal Y_{\overline{\alpha}}\equiv\overline{\mathcal Y_\alpha}$: the conjugate intertwining operator of $\mathcal Y_\alpha$.\\
$\mathcal Y_{\alpha^*}\equiv\mathcal Y_{\alpha}^\dagger$: the adjoint intertwining operator of $\mathcal Y_\alpha$.\\
$\mathcal Y_{B_\pm\alpha}\equiv B_{\pm}\mathcal Y_\alpha$: the braided intertwining operators of $\mathcal Y_\alpha$.\\
$\mathcal Y_{C\alpha}\equiv C\mathcal Y_\alpha$: the contragredient intertwining operator of $\mathcal Y_\alpha$.\\
$\mathcal Y^i_{i0}=B_\pm Y_i$, the creation operator of $W_i$.\\
$\mathcal Y^0_{\overline ii}=C^{-1}\mathcal Y^{\overline i}_{\overline i0}=(\mathcal Y^i_{i0})^\dagger$, the annihilation operator of $W_i$.\\
$\mathcal Y_{\sigma_n\cdots\sigma_2,\alpha}$: a generalized intertwining operator of $V$.\\
$\Delta_i$: the conformal weight of  $W_i$.\\
$\Delta_w$: the conformal weight (the energy) of the homogeneous vector $w$.\\
$\Theta^k_{ij}$: a set of linear basis of $\mathcal V{k\choose i~j}$.\\
$\Theta^k_{i*}=\coprod_{j\in\mathcal E}\Theta^k_{ij},\Theta^k_{*j}=\coprod_{i\in\mathcal E}\Theta^k_{ij},\Theta^*_{ij}=\coprod_{k\in\mathcal E}\Theta^k_{ij}.$\\
$\Theta^*_{i*}=\coprod_{s,t\in\mathcal E\cap\mathcal F^\boxtimes}\Theta^t_{is},\Theta^*_{*j}=\coprod_{s,t\in\mathcal E\cap\mathcal F^\boxtimes}\Theta^t_{sj}.$\\
$\theta$: the PCT operator of $V$, or a real variable.\\
$\vartheta_i$: the twist of $W_i$.\\
$\Lambda$: the sesquilinear form defined by transport matrices on $\mathcal V{k\choose i~j}^*$ or on $W_i\boxtimes W_j$.\\
$\nu$: the conformal vector of $V$.\\
$\sigma_{i,j}$: the braid operator  $\sigma_{i,j}:W_i\boxtimes W_j\rightarrow W_j\boxtimes W_i$.\\
$\Upsilon^0_{i\overline i}=C\mathcal Y^i_{i0}$.\\
$\Omega$: the vacuum vector of $V$.

	\setcounter{section}{3}

\section{From unitary VOAs to conformal nets}

In this chapter, we assume that $V$ is unitary and energy-bounded. A net $\mathcal M_V$ of von Neumann algebras on the circle can be defined using smeared vertex operators of $V$. If $\mathcal M_V$ is a conformal net, then $V$ is called strongly local. A theorem in \cite{CKLW} shows that when $V$ is generated by a set of quasi-primary vectors whose field operators satisfy linear energy bounds, then $V$ is strongly local. This is discussed in section 4.1. 

Let $W_i$ be an energy-bounded unitary $V$-module. If this representation of $V$ can be integrated to a representation of the conformal net $\mathcal M_V$, we say that $W_i$ is strongly integrable. In section 4.1, we show that the abelian category of energy-bounded strongly-integrable unitary $V$-modules is equivalent to the  category of the corresponding integrated $\mathcal M_V$-modules. A similar topic is treated in \cite{CWX}.

There are two major ways to prove the strong integrability of a unitary $V$-modules $W_i$.  First, if the action of $V$ on $W_i$ is restricted from the inclusion of $V$ in a larger energy-bounded strongly-local unitary VOA, then $W_i$ is  strongly local. This result is proved in \cite{CWX}, and will not be used in our paper. In section 4.2, we give a different criterion using the energy bounds condition of intertwining operators.

\subsection{Unitary VOAs, conformal nets, and their representations}

We first review the definition of conformal nets. Standard references  are \cite{CKLW,Car04,GF93,GL96,KL04}. Conformal nets are based on the theory of  von Neumann algebras. For an outline of this theory, we recommend \cite{Con80} chapter 5. More details can be found in \cite{Jon03,Takesaki I,Takesaki II,KR83,KR15}.

Let $\Diff(S^1)$ be the group of orientation-preserving diffeomorphisms of $S^1$. Convergence in $\Diff(S^1)$ means uniform convergence of all derivatives. Let $\mathcal H$ be a Hilbert space, and let $\mathcal U(\mathcal H)$ let the group of unitary operators on $\mathcal H$, equipped with the strong (operator) topology. $P\mathcal U(\mathcal H)$ is the quotient topology group of $\mathcal U(\mathcal H)$, defined by identifying $x$ with $\lambda x$ when $x\in \mathcal U(\mathcal H), \lambda\in S^1$. A \textbf{strongly continuous projective representation} of $\Diff(S^1)$ on $\mathcal H$ is, by definition, a continuous homomorphism from $\Diff(S^1)$ into $P\mathcal U(\mathcal H)$.

$\Diff(S^1)$ contains the subgroup $\PSU(1,1)$ of M\"obius transformations of $S^1$. Elements in $\PSU(1,1)$ are of the form
\begin{align}
z\mapsto \frac{\lambda z+\mu}{\overline\mu z+\overline \lambda}  \quad(z\in S^1),
\end{align}
where $\lambda,\mu\in\mathbb C,|\lambda|^2-|\mu|^2=1$. $\PSU(1,1)$ contains the subgroup $S^1=\{\mathfrak r(t):t\in\mathbb R \}$ of rotations of $S^1$.

A \textbf{conformal net} $\mathcal M$ associates to each $I\in\mathcal J$ a von Neumann algebra $\mathcal M(I)$ acting on a fixed Hilbert space $\mathcal H_0$, such that the following conditions hold:\\
(a) (Isotony) If $I_1\subset I_2\in\mathcal J$, then $\mathcal M(I_1)$ is a von Neumann subalgebra of $\mathcal M(I_2)$.\\
(b) (Locality) If $I_1,I_2\in\mathcal J$ are disjoint, then $\mathcal M(I_1)$ and $\mathcal M(I_2)$ commute.\\
(c) (Conformal covariance) We have a strongly continuous projective unitary representation $U$ of $\Diff(S^1)$ on $\mathcal H_0$, such that for any $g\in \Diff(S^1),I\in\mathcal J$,
\begin{align*}
U(g)\mathcal M(I)U(g)^*=\mathcal M(gI).
\end{align*}
 Moreover, if $g$ fixes the points in $I$, then for any $x\in\mathcal M(I)$,
\begin{align*}
U(g)xU(g)^*=x.
\end{align*}
(d) (Positivity of energy) The generator of the restriction of $U$ to $S^1$ is positive.\\
(e) There exists a  unique (up to scalar) unit vector $\Omega\in\mathcal H_0$ (the vacuum vector), such that $U(g)\Omega\in\mathbb C\Omega$ for any $g\in\PSU(1,1)$. Moreover, $\Omega$ is  cyclic under the action of $\bigvee_{I\in\mathcal J}\mathcal M(I)$ (the von Neumann algebra generated by all $\mathcal M(I)$).

The following properties are satisfied by a conformal net, and will be used in our theory:\\
(1) (Additivity) If $\{I_a:a\in\mathcal A \}$ is a collection of open intervals in $\mathcal J$, $I\in\mathcal J$, and $I=\bigcup_{a\in\mathcal A}I_a$, then $\mathcal M(I)=\bigvee_{a\in\mathcal A}\mathcal M(I_a)$.\\
(2) (Haag duality) $\mathcal M(I)'=\mathcal M(I^c)$, where $\mathcal M'=\End_{\mathcal M}(\mathcal H_0)$ is the commutant of $\mathcal M$.\\
(3) $\mathcal M(I)$ is a type III factor. (Indeed, it is of type III$_1$.)\\
Properties (2) and (3) are natural consequences of Bisognano-Wichmann theorem, cf. \cite{BGL93,GF93}.

Following \cite{CKLW}, we now show how to construct a conformal net $\mathcal M_V$ from $V$.  Let the Hilbert space $\mathcal H_0$ be the norm completion of $V$. For any $I\in\mathcal J$ we define $\mathcal M_V(I)$ to be the von Neumann algebra on $\mathcal H_0$  generated by  closed operators of the form $\overline {Y(v,f)}$, where $v\in V$  and $f\in C^\infty_c(I)$. Thus we've obtained a net of von Neumann algebras $I\in\mathcal J\mapsto \mathcal M_V(I)$ and denote it by $\mathcal M_V$. The vacuum vector $\Omega$ in $\mathcal H_0$ is the same as that of $V$. The projective representation $U$ of $\Diff(S^1)$ is obtained by integrating the action of the real part of the Virasoro algebra  on $V$. The representation of $\PSU(1,1)$ is determined by the action of $L_{\pm1},L_0$ on $V$. \emph{All the axioms of conformal nets, except possibly locality, are satisfied for $\mathcal M_V$.}

Locality of $\mathcal M_V$, however, is much harder to prove. To be sure, for any disjoint $I,J\in\mathcal J$, and any  $u,v\in V$, we can use proposition 2.13, corollary 3.13, and proposition 3.9 to show that
\begin{align}
Y(u,f)Y(v,g)=Y(v,g)Y(u,f),\label{eq305}\\
Y(u,f)^\dagger Y(v,g)=Y(v,g)Y(u,f)^\dagger,
\end{align}
where both sides act on $\mathcal H^\infty_0$. The commutativity of closed operators on a common invariant core, however, does not imply the strong commutativity of these two operators, as indicated by the example of Nelson (cf. \cite{Nel}). So far, the best result we have for the locality of $\mathcal M_V$ is the following:
\begin{thm}\label{lb61}
Suppose that $V$ is generated  by a set $E$ of quasi-primary vectors, and that for any $v\in E$, $Y(v,x)$ satisfies linear energy bounds. Then the net $\mathcal M_V$ satisfies the  locality condition, and  is therefore a conformal net. Moreover, if we let $E_{\mathbb R}=\{v+\theta v,i(v-\theta v):v\in E \}$, then for any $I\in\mathcal J$, $\mathcal M_V(I)$ is generated by the closed operators $\overline{Y(u,f)}$, where $u\in E_{\mathbb R}$, and $f\in C^\infty_c(I)$ satisfies $e^{i\pi\Delta_u/2}e_{1-\Delta_u}f=\overline {e^{i\pi\Delta_u/2}e_{1-\Delta_u}f}$.
\end{thm}

\begin{proof}
Clearly $E_{\mathbb R}$ generates $V$. From the proof of \cite{CKLW}	theorem 8.1, it suffices to prove, for any disjoint $I,J\in\mathcal J$,  $u,v\in E_{\mathbb R}$, and  $f\in C^\infty_c(I),g\in C^\infty_c(J)$ satisfying $e^{i\pi\Delta_u/2}e_{1-\Delta_u}f=\overline {e^{i\pi\Delta_u/2}e_{1-\Delta_u}f},e^{i\pi\Delta_v/2}e_{1-\Delta_v}g=\overline {e^{i\pi\Delta_v/2}e_{1-\Delta_v}g}$, that $\overline {Y(u,f)}$ and $\overline{Y(v,g)}$ commute strongly. By proposition 3.9-(b), $Y(u,f)$ and $Y(v,g)$ are symmetric operators. Hence by equation (3.38), proposition 3.9-(a), equation \eqref{eq305}, Lemma B.8, and theorem B.9, $\overline {Y(u,f)}$ and $\overline{Y(v,g)}$ are self-adjoint operators, and they commute strongly with each other.
\end{proof}
We say that a unitary energy-bounded strongly local VOA $V$ is \textbf{strongly local}, if $\mathcal M_V$ satisfies the locality condition.\\

Suppose that $V$ is strongly local. We now discuss representations of the conformal net $\mathcal M_V$. Let $\mathcal H_i$ be a Hilbert space (currently not yet related to $W_i$). Suppose that for any $I\in\mathcal J$, we have a (normal unital *-) representation $\pi_{i,I}:\mathcal M_V(I)\rightarrow B(\mathcal H_i)$, such that for any $I_1,I_2\in\mathcal J$ satisfying $I_1\subset I_2$, and any $x\in\mathcal M_V(I_1)$, we have $\pi_{i,I_1}(x)=\pi_{i,I_2}(x)$. Then $(\mathcal H_i,\pi_i)$ (or simply $\mathcal H_i$) is called a (locally normal) \textbf{represention} of the  $\mathcal M_V$ (or an \textbf{$\mathcal M_V$-module}). We shall write $\pi_{i,I}(x)$ as $\pi_i(x)$ when we have specified that $x\in\mc M_V(I)$. If moreover $\xi^{(i)}\in\mathcal H_i$, we also write $x\xi^{(i)}$ for $\pi_i(x)\xi^{(i)}=\pi_{i,I}(x)\xi^{(i)}$.

The $\mathcal M_V$-modules  we are interested in are those arising from unitary $V$-modules. Let $W_i$ be an energy-bounded unitary  $V$-module, and let $\mathcal H_i$ be the norm completion of the inner product space $W_i$. Assume that we have a  representation $\pi_i$ of $\mathcal M_V$ on $\mathcal H_i$. Then we say that $(\mathcal H_i,\pi_i)$ is \textbf{associated with the $V$-module $(W_i,Y_i)$}, if for any $I\in\mathcal J$, $v\in V$, and $f\in C^\infty_c(I)$, we have
\begin{align}
\pi_{i,I}\big(\overline{Y(v,f)}\big)=\overline{Y_i(v,f)}.\label{eq228}
\end{align}
(See section B.1 for the definition of $\pi_{i,I}$ acting on unbounded closed operators affiliated with $\mathcal M_V(I)$.) A $\mathcal M_V$-module associated with $W_i$, if exists, must be unique. We say that an energy-bounded  unitary $V$-module $W_i$ is \textbf{strongly integrable} if there exists a $\mathcal M_V$-module $(\mathcal H_i,\pi_i)$ associated with $W_i$. Let $\mathcal S$ be the collection of strongly integrable energy-bounded unitary $V$-modules. Obviously $V\in\mathcal S$. It is easy to show that $S$ is additively complete.

We now introduce a very useful density property. For any $I\in\mathcal J$, we define $\mathcal M_V(I)_\infty$ to be the set of \textbf{smooth operators} in $\mathcal M_V(I)$, i.e., the set of all  $x\in\mathcal M_V(I)$ satisfying that for any unitary $V$-module $W_i$ inside $\mathcal S$, 
\begin{align}
x\mathcal H_i^\infty\subset\mathcal H_i^\infty,\quad x^*\mathcal H_i^\infty\subset\mathcal H_i^\infty.\label{eq301}
\end{align}
\begin{pp}\label{lb104}
If $V$ is unitary, energy-bounded, and strongly local, then $\mathcal M_V(I)_\infty$ is a strongly dense self-adjoint subalgebra of $\mathcal M_V(I)$.
\end{pp}

\begin{proof}
By additivity or by the construction of $\mathcal M_V$, we have $\mathcal M_V(I)=\bigvee_{J\subset\joinrel\subset I}\mathcal M_V(J)$. ($J\subset\joinrel\subset I$ means that $J\in\mathcal J$ and $\overline J\subset I$.) For each $J\subset\joinrel\subset I$ and $x\in\mathcal M_V(J)$, we choose $\epsilon>0$ such that $\mathfrak r(t)J\subset I$ whenever $t\in(-\epsilon,\epsilon)$. For each $h\in C_c^\infty(-\epsilon,\epsilon)$ such that $\int_{-\epsilon}^{\epsilon}h(t)dt=1$, define $$x_h=\int_{-\epsilon}^{\epsilon}e^{it\overline{L_0}}xe^{-it\overline{L_0}}h(t)dt.$$ Then by (3.39), $x_h\in\mathcal M_V(I)$. For each  $W_i$ inside $\mathcal S$, equations (3.39) and \eqref{eq228} imply that 
\begin{align}
\pi_i(e^{it\overline{L_0}}xe^{-it\overline{L_0}})=e^{it\overline{L_0}}\pi_i(x)e^{-it\overline{L_0}}.\label{eq319}
\end{align}
So we have
$$\pi_i(x_h)=\int_{-\epsilon}^{\epsilon}e^{it\overline{L_0}}\pi_i(x)e^{-it\overline{L_0}}h(t)dt,$$
which implies that
\begin{align}
e^{it\overline{L_0}}\pi_i(x_h)\xi^{(i)}=\pi_i(x_{h_t})e^{it\overline{L_0}}\xi^{(i)},
\end{align} 
where $h_t(s)=h(s-t)$.
From this equation, we see that the derivative of $e^{it\overline{L_0}}\xi^{(i)}\in\mathcal H^\infty_i$ at $t=0$ exists and equals
	\begin{gather}
	-\pi_i(x_{h'})\xi^{(i)}+i\pi_i(x_h) \overline{L_0}\xi^{(i)}.\label{eq229}
	\end{gather}
	This implies that $\pi_i(x_h)\xi^{(i)}\in \mathcal H^1_i$ and $i\overline{L_0}\pi_i(x_h)\xi^{(i)}$ equals \eqref{eq229}. Using the same argument, we see that for each $n\in\mathbb Z_{\geq0}$, the following Leibniz rule holds:
	\begin{gather*}
	\pi_i(x_h)\xi^{(i)}\in \mathscr D(\overline {L_0}^n)=\mathcal H^n_i,\\
	\overline {L_0}^n\pi_i(x_h)\xi^{(i)}=\sum_{m=0}^n {n\choose m} i^m\pi_i(x_{h^{(m)}})\cdot\overline{L_0}^{n-m}\xi^{(i)},
	\end{gather*}
	where $h^{(m)}$ is the $m$-th derivative of $h$. This proves that $\pi_i(x_h)\mathcal H^\infty_i\subset \mathcal H^\infty_i$.
	
	Since $(x_h)^*=(x^*)_{\overline h}$, we also have $x_h^*\mathcal H^\infty_i\subset\mathcal H^\infty_i$. So $x_h\in\mathcal M_V(I)_\infty$. Clearly $x_h\rightarrow x$ strongly as $h$ converges to the $\delta$-function at $0$. We thus conclude that any $x\in\mathcal M_V(J)$ can be strongly approximated by elements in $\mathcal M_V(I)_\infty$. Hence the proof is finished.
\end{proof}

We study the relation between the representation categories of $\mathcal M_V$ and $V$. Assume, as before, that $V$ is unitary, energy-bounded, and strongly local. We define an additive category $\Rep_{\mathcal S}(\mathcal M_V)$ as follows: The objects are $\mathcal M_V$-modules of the form $\mathcal H_i$, where $W_i$ is an element inside $\mathcal S$. If $W_i,W_j$ are inside $\mathcal S$, then the vector space of morphisms $\Hom_{\mathcal M_V}(\mathcal H_i,\mathcal H_j)$ consists of bounded linear operators $R:\mathcal H_i\rightarrow \mathcal H_j$, such that for any $I\in\mathcal J,x\in\mathcal M_V(I)$, the relation $R\pi_i(x)=\pi_j(x)R$ holds.

Define a functor $\mathfrak F:\Rep^\uni_{\mathcal S}(V)\rightarrow \Rep_{\mathcal S}(\mathcal M_V)$ in the following way: If $W_i$ is a unitary $V$-module in $\mathcal S$, then we let $\mathfrak F(W_i)$ be the $\mathcal M_V$-module $\mathcal H_i$. If $W_i,W_j$ are in $\mathcal S$ and  $R\in\Hom_V(W_i,W_j)$, then by lemma 2.20, $R$ is bounded, and hence can be extended to a bounded linear map $R:\mathcal H_i\rightarrow \mathcal H_j$. It is clear that $R$ is an element in $\Hom_{\mathcal M_V}(\mathcal H_i,\mathcal H_j)$. We let $\mathfrak F(R)$ be this $\mathcal M_V$-module homomorphism. Clearly $\mathfrak F:\Hom_V(W_i,W_j)\rightarrow \Hom_{\mathcal M_V}(\mathcal H_i,\mathcal H_j)$ is linear. We show that $\mathfrak F$ is an isomorphism.

\begin{thm}\label{lb63}\footnote{This theorem is also proved in \cite{CWX}. We would like to thank Sebastiano Carpi for letting us know this fact.}
Let $V$ be unitary, energy-bounded, and strongly local. For any $W_i,W_j$ in $\mathcal S$, the linear map $\mathfrak F:\Hom_V(W_i,W_j)\rightarrow \Hom_{\mathcal M_V}(\mathcal H_i,\mathcal H_j)$ is an isomorphism. Therefore, $\mathfrak F:\Rep^\uni_{\mathcal S}(V)\rightarrow \Rep_{\mathcal S}(\mathcal M_V)$ is an equivalence of additive categories.
\end{thm}

\begin{proof}
The linear map $\mathfrak F:\Hom_V(W_i,W_j)\rightarrow \Hom_{\mathcal M_V}(\mathcal H_i,\mathcal H_j)$ is clearly injective. We only need to prove that $\mathfrak F$ is  surjective.
Choose $R\in\Hom_{\mathcal M_V}(\mathcal H_i,\mathcal H_j)$. Define an orthogonal direct sum module $W_k=W_i\oplus^\perp W_j$. Then  $\mathcal H_k$ is  the orthogonal direct sum $\mathcal M_V$-module of $\mathcal H_i,\mathcal H_j$. Regard $R$ as an element in $\End_{\mathcal M_V}(\mathcal H_k)$, which is the original operator when acting on $\mathcal H_i$, and is $0$ when acting on $\mathcal H_j$.  Then for any $I\in\mathcal J,x\in\mathcal M_V(I)$, $R$ commutes with $\pi_k(x),\pi_k(x^*)$. Therefore, for any homogeneous $v\in V$ and $f\in C^\infty_c(I)$, $R$ commutes strongly with $\pi_k\big(\overline{Y(v,f)}\big)=\overline{Y_k(v,f)}$.

We first show that $RW_i\subset W_j$. Choose $I_1,I_2\in\mathcal J$ and $f_1\in C^\infty_c(I_1,\mathbb R),f_2\in C^\infty_c(I_2,\mathbb R)$ such that $f_1+f_2=1$. Regard $L_0$ as an unbounded operator on $\mathcal H_k$ with domain $W_k$. Then  $L_0$ is the restriction of the smeared vertex operator $Y_k(\nu,e_1)$ to $W_k$. (Recall that by our notation of $e_r$, $e_1(e^{i\theta})=e^{i\theta}$.) Therefore,
\begin{align*}
L_0\subset Y_k(\nu,e_1f_1)+Y_k(\nu,e_1f_2),
\end{align*}
and hence
\begin{align*}
\overline{L_0}\subset\overline{Y_k(\nu,e_1f_1)+Y_k(\nu,e_1f_2)}\subset\overline {\overline{Y_k(\nu,e_1f_1)}+\overline{Y_k(\nu,e_1f_2)}}.
\end{align*}
Recall that $\nu$ is quasi-primary and $\Delta_\nu=2$. Therefore, by equation (3.25),  $\overline{Y_k(\nu,e_1f_1)}$ and $\overline{Y_k(\nu,e_1f_2)}$ are symmetric operators. It follows that $A=\overline {\overline{Y_k(\nu,e_1f_1)}+\overline{Y_k(\nu,e_1f_2)}}$ is symmetric. Note that $\overline{L_0}$ is self adjoint. Thus we have $$\overline{L_0}\subset A\subset A^*\subset \overline {L_0}^*=\overline{L_0},$$ which implies that $$\overline{L_0}=\overline {\overline{Y_k(\nu,e_1f_1)}+\overline{Y_k(\nu,e_1f_2)}}.$$ 
Therefore, since $R$ commutes strongly with $\overline{Y_k(\nu,e_1f_1)}$ and $\overline{Y_k(\nu,e_1f_2)}$, $R$ also commutes strongly with $\overline {L_0}$. In particular, $R$ preserves  every eigenspace of $\overline{L_0}$ in $\mathcal H_k$. This implies that $RW_i(s)\subset W_j(s)$ for any $s\in\mathbb R$, and hence that $RW_i\subset W_j$.

Now, for any  $n\in\mathbb Z,w^{(i)}\in W_i$, and  $v\in V$, we have $$Y_k(v,n)w^{(i)}=Y_k(v,e_n)w^{(i)}=\overline{Y_k(v,e_nf_1)}w^{(i)}+\overline{Y_k(v,e_nf_2)}w^{(i)}.$$
Since $R$ commutes strongly with $\overline{Y_i(v,e_nf_1)},\overline{Y_i(v,e_nf_2)}$, we have $RY_k(v,e_n)w^{(i)}=Y_k(v,e_n)Rw^{(i)}$, which implies that $RY_i(v,n)w^{(i)}=Y_j(v,n)Rw^{(i)}$. Therefore, $R\in\Hom_V(W_i,W_j)$.
\end{proof}

\begin{co}\label{lb103}
If $W_i$ is a unitary $V$-module in $\mathcal S$, and $\mathcal H_1$ is a (norm-)closed $\mathcal M_V$-invariant subspace of $\mathcal H_i$, then there exists a $V$-invariant subspace $W_1$ of $W_i$, such that $\mathcal H_1$ is the norm closure of $W_1$. 
\end{co}
\begin{proof}
Let $e_1$ be the orthogonal projection of $\mathcal H_i$ onto $\mathcal H_1$. Then $e_1\in\End_{\mathcal M_V}(\mathcal H_i)$. By theorem \ref{lb63}, $e_1$ restricts to an element in $\End_V(W_i)$. So $W_1=e_1W_i$ is a $V$-invariant subspace of $W_i$, and $e_1L_0=L_0e_1$ when both sides act on $W_i$. Therefore $e_1$ commutes strongly with $\overline{L_0}$. Let $P_s$ be  the projection of $\mathcal H_i$ onto $W_i(s)$. Then $P_s$ is a spectral projection of $\overline {L_0}$. Hence $eP_s=P_se$ for any $s\geq0$. 

Choose any $\xi\in \mathcal H_1$. Then $\xi=\sum_{s\geq0}P_s\xi$. Since for any $s\geq0$ we have $P_s\xi=P_se_1\xi=e_1P_s\xi\in e_1W_i= W_1$, we see that $\xi$ can be approximated by vectors in $W_1$. This proves that $\mathcal H_1$ is the norm closure of $W_1$.
\end{proof}

\subsection{A criterion for strong integrability}\label{lb300}

Assume that $V$ is unitary, energy bounded, and strongly local. In this section, we give a criterion for the strong integrability of  energy-bounded unitary $V$-modules.

\begin{pp}\label{lb118}
	Let $W_i$ be a non-trivial energy-bounded unitary $V$-module. Then $W_i$ is strongly integrable if and only if for any $I\in\mathcal J$, there exists a unitary operator $U_I:\mathcal H_0\rightarrow \mathcal H_i$, such that  any $v\in V$ and $f\in C_c^\infty(I)$ satisfy
	\begin{align}
	\overline{Y_i(v,f)}=U_I\overline{Y(v,f)}U_I^*.\label{eq355}
	\end{align}
\end{pp}

\begin{proof}
	``If part": For any $I\in\mathcal J(I)$, we define a representation $\pi_{i,I}$ of $\mathcal M_V(I)$ on $\mathcal H_i$ to be 
	\begin{align}
	\pi_{i,I}(x)=U_IxU_I^*\qquad(x\in\mathcal M_V(I)).
	\end{align}
	If $J\in\mathcal J(I)$ and $I\subset J$, then by equation \eqref{eq355}, $U_J^*U_I$ commutes strongly with every $\overline{Y(v,f)}$ where $v\in V$ and $f\in C^\infty_c(I)$. So $U_J^*U_I$ commutes with $\mathcal M_V(I)$, which implies that $\pi_{i,I}$ is the restriction of $\pi_{i,J}$ on $\mathcal M_V(I)$. So $\pi_i$ is a representation of the conformal net $\mathcal M_V$. It is obvious that $\pi_i$ is associated with $W_i$. So $W_i$ is strongly integrable.
	
	``Only if part": Suppose that $W_i$ is strongly integrable. We let $(\mathcal H_i,\pi_i)$ be the $\mathcal M_V$-module associated with $W_i$. For each $I\in\mathcal M_I$, $\pi_{i,I}$ is a non-trivial representation of $\mathcal M_V(I)$ on $\mathcal H_i$. Since the Hilbert spaces $\mathcal H_0,\mathcal H_i$ are separable, and $\mathcal M_V(I)$ is a type III factor, $\pi_{i,I}$ is (unitary) equivalent to the representation $\pi_{0,I}$ of $\mathcal M_V(I)$ on $\mathcal H_0$. So there exits a unitary $U_I:\mathcal H_i\rightarrow \mathcal H_i$ such that equation \eqref{eq355} always holds.
\end{proof}

\begin{rem}
	Equation \eqref{eq355} is equivalent to one of the following equivalent relations:
	\begin{gather}
	U_I\overline{Y(v,f)}\subset \overline{Y_i(v,f)}U_I,\\
	U_I^*\overline{Y_i(v,f)}\subset \overline{Y(v,f)}U_I^*.
	\end{gather}
\end{rem}

\begin{pp}\label{lb119}
	Let $W_j,W_k$ be non-trivial energy-bounded unitary $V$-modules. Assume that $W_j$ is strongly integrable. If for any $I\in\mathcal J$ there exits a collection $\{T_a:a\in\mathcal A \}$ of bounded linear operators from $\mathcal H_j$ to $\mathcal H_k$, such that $\bigvee_{a\in\mathcal A}T_a\mathcal H_j$ is dense in $\mathcal H_k$, and that for any $a\in\mathcal A,v\in V,f\in C^\infty_c(I)$, we have
	\begin{gather}
	T_a\overline{Y_j(v,f)}\subset \overline{Y_k(v,f)}T_a,\label{eq356}\\
	T_a^*\overline{Y_k(v,f)}\subset \overline{Y_j(v,f)}T_a^*,\label{eq357}
	\end{gather}
	then $W_k$ is strongly integrable.
\end{pp}

We remark that when $T_\alpha$ is not unitary, equations \eqref{eq356} and \eqref{eq357} do not imply each other.
\begin{proof}
	Let $W_l=W_j\oplus^\perp W_k$ be the direct sum module of $W_j$ and $W_k$, and extend each $T_a$ to a bounded linear operator on $\mathcal H_l$, such that $T_a$ equals zero on the subspace $\mathcal H_k$. Choose any $I\in\mathcal J$. Since $\overline{Y_l(v,f)}=\diag(\overline{Y_j(v,f)},\overline{Y_k(v,f)})$,  equations \eqref{eq356} and \eqref{eq357} are equivalent to that $T_a$ commutes strongly with $\overline{Y_l(v,f)}$ for any $v\in V,f\in C^\infty_c(I)$. We construct a unitary operator $U_I:\mathcal H_j\rightarrow\mathcal H_k$ such that
	\begin{align}
	\overline{Y_k(v,f)}=U_I\overline{Y_j(v,f)}U_I^*\label{eq358}
	\end{align}
	for any $v\in V,f\in C^\infty_c(I)$. Then the strong integrability of $W_k$ will follow immediately from proposition \ref{lb118} and the strong integrability of $W_j$.
	
	Let $\{U_b:b\in\mathcal B \}$ be a maximal collection of non-zero partial isometries from $\mathcal H_j$ to $\mathcal H_k$ satisfying the following conditions:\\
	(a) For any $b\in\mathcal B,v\in V,f\in C^\infty_c(I)$, $U_b$ commutes strongly with $\overline{Y_l(v,f)}$.\\
	(b) The projections $\{e_b=U_bU_b^*:b\in\mathcal B \}$ are pairwise orthogonal.\\
	Note that similar to $T_a$, each $U_b$ is  extended to a partial isometry on $\mathcal H_l$, being zero when acting on $\mathcal H_k$. 
	
	Let $e=\sum_{b\in\mathcal B}e_b$. We prove that $e=\id_{\mathcal H_k}$. Let $e'=\id_{\mathcal H_k}-e$. If $e'\neq0$, then by the density of $\bigvee_{a\in\mathcal A}T_a\mathcal H_j$ in $\mathcal H_k$, there exists $a\in\mathcal A$ such that $e'T_a\neq0$. Take the left polar decomposition $e'T_a=U_aH_a$ of $e'T_a$, where $U_a$ is the partial isometry part. Then $U_aU_a^*$ is the projection of $\mathcal H_l$ onto the range of $e'T_a$, which is nonzero and orthogonal to each $e_b$. For each $v\in V,f\in C^\infty_c(I)$, since  $e'$ and $T_a$ commute strongly with $\overline{Y_l(v,f)}$, $U_a$ also commutes strongly with $\overline{Y_l(v,f)}$. Therefore, $\{U_b:b\in\mathcal B \}\cup\{U_a \}$ is a collection of partial isometries from $\mathcal H_j$ to $\mathcal H_k$ satisfying conditions (a) and (b), and $\{U_b:b\in\mathcal B \}$ is its proper sub-collection. This contradicts the fact that $\{U_b:b\in\mathcal B \}$ is maximal. So $e'=0$, and hence $e=\id_{\mathcal H_k}$.
	
	For each $b\in\mathcal B$ we let $p_b=U_b^*U_b$, which is a non-zero projection on $\mathcal H_j$. We now restrict ourselves to operators on $\mathcal H_j$. Then $p_b$ commutes strongly with each $\overline{Y_j(v,f)}$, which, by the strong integrability of $W_j$, is equivalent to that $p_b\in \pi_{j,I}(\mathcal M_V(I))'$. Note that $\mathcal B$ must be countable. We choose a countable collection $\{q_b:b\in\mathcal B \}$ of non-zero orthogonal projections on $\mathcal H_j$ satisfying that $\sum_{b\in\mathcal B}q_b=\id_{\mathcal H_j}$, and that each $q_b\in\pi_{j,I}(\mathcal M_V(I))'$. Since $\pi_{j,I}(\mathcal M_V(I))'$ is a type III factor, for each $b$ there exists a partial isometry $\widetilde U_b\in\pi_{j,I}(\mathcal M_V(I))'$ satisfying $\widetilde U_b\widetilde U_b^*=p_b,\widetilde U_b^*\widetilde U_b=q_b$. 
	
	We turn our attention back to operators on $\mathcal H_l$. Since $\widetilde U_b\in\pi_{j,I}(\mathcal M_V(I))'$, $\widetilde U_b$ commutes strongly with each  $\overline{Y_l(v,f)}$. Let $U_I=\sum_{b\in\mathcal B}U_b\widetilde U_b$. Then $U_I$ is a unitary operator from $\mathcal H_j$ to $\mathcal H_k$ satisfying relation \eqref{eq358} for any $v\in V,f\in C^\infty_c(I)$. Thus our proof is finished.
\end{proof}

We now prove the strong integrability of an energy-bounded unitary $V$-module using the linear energy-boundedness of intertwining operators.

\begin{thm}\label{lb120}
	Let $W_i,W_j,W_k$ be non-zero unitary irreducible $V$-modules. Assume that $W_j$ and $W_k$ are energy-bounded, that $W_j$ is strongly integrable, and that there exist a non-zero quasi-primary vector $w^{(i)}_0\in W_i$ and a non-zero intertwining operator $\mathcal Y_{\alpha}\in\mathcal V{k\choose i~j}$, such that $\mathcal Y_{\alpha}(w^{(i)}_0,x)$ satisfies linear energy bounds. Then $W_k$ is strongly integrable.
\end{thm}

\begin{proof}
	Step 1. Fix any $J\in\mathcal J(S^1\setminus\{-1\})$, and let $\mathcal W_J$ be the subspace of $\mathcal H_k$ spanned by the vectors $\mathcal Y_\alpha(w^{(i)}_0,g)w^{(j)}$ where $g\in C^\infty_c(J)$ and $w^{(j)}\in W_j$. We show that $\mathcal W_J$ is a dense subspace of $\mathcal H_k$.
	
Our proof is similar to that of Reeh-Schlieder theorem (cf. \cite{RS61}).	Choose $\xi^{(k)}\in \mathcal W_J^\perp$.  Note that for each $\eta^{(k)}\in\mathcal H^{k}$, the multivalued function
	\begin{align}
	z\mapsto z^{\overline{L_0}}\eta^{(k)}=\sum_{s\geq0}z^sP_s\eta^{(k)}
	\end{align}
	is continuous on $\overline D^\times(1)=\{\zeta\in\mathbb C:0<|\zeta|\leq1 \}$ and holomorphic on its interior $D^\times(1)$. So we have a multivalued holomorphic function of $z$:
	\begin{align}
	\langle z^{\overline{L_0}}\mathcal Y_\alpha(w^{(i)}_0,g)w^{(j)}|\xi^{(k)} \rangle,\label{eq362}
	\end{align}
	which is continuous on $\overline D^\times(1)$ and holomorphic on $D^\times(1)$. Choose $\varepsilon>0$ such that   the support of $g^t=\exp(it(\Delta_{w^{(i)}_0}-1))\mathfrak r(t)g$ is inside $J$ for any $t\in(-\varepsilon,\varepsilon)$. Then, by proposition 3.15, we have
	\begin{align}
	\langle e^{it\overline{L_0}}\mathcal Y_\alpha(w^{(i)}_0,g)w^{(j)}|\xi^{(k)} \rangle=\langle \mathcal Y_\alpha(w^{(i)}_0,g^t)e^{it\overline{L_0}}w^{(j)}|\xi^{(k)} \rangle,\label{eq363}
	\end{align}
	which must be zero when $t\in(-\delta,\delta)$. 
	
	By Schwarz reflection principle, the value of function \eqref{eq362} is zero for any $z\in\overline D^\times(r)$. In particular, it is zero for any $z\in S^1$. This shows that \eqref{eq363} is zero for any $t\in\mathbb R$. Here, when we define the smeared intertwining operator, we allow the arguments to exceed the region $(-\pi,\pi)$ under the action of $\mathfrak r(t)$. So the right hand side of equation \eqref{eq363} becomes
	\begin{align}
	\sum_{s\in\mathbb R}\int_{t-\pi}^{t+\pi} \langle \mathcal Y_\alpha(w^{(i)}_0,e^{i\theta})e^{it\overline{L_0}}w^{(j)}|P_s\xi^{(k)} \rangle\cdot\exp(it(\Delta_{w^{(i)}_0}-1))g(e^{i(\theta-t)})\di\theta,\label{eq360}
	\end{align}
which is $0$ for any $t\in\mathbb R$. (Recall our notation that $\di\theta=e^{i\theta}d\theta/2\pi$.)	
	Since $W_i,W_j,W_k$ are irreducible, we let $\Delta_i,\Delta_j,\Delta_k$ be their conformal dimensions, and set $\Delta_\alpha=\Delta_i+\Delta_j-\Delta_k$. Then by equation (1.25),
	\begin{align}
	\mathcal Y_\alpha(w^{(i)}_0,z)z^{\Delta_\alpha}=\sum_{n\in\mathbb Z}\mathcal Y_\alpha(w^{(i)}_0,\Delta_\alpha-1-n)z^{n}
	\end{align}
	is a single valued holomorphic function for $z\in\mathbb C^\times $. So the fact that \eqref{eq360} always equals $0$ implies that
	\begin{align}
	\sum_{s\in\mathbb R}\int_{-\pi}^\pi \langle \mathcal Y_\alpha(w^{(i)}_0,e^{i\theta})w^{(j)}|P_s\xi^{(k)} \rangle e^{i\Delta_\alpha\theta}\cdot h(e^{i\theta})d\theta=0\label{eq361}
	\end{align}
	for any $w^{(j)}\in W_j,I\in \mathcal J$ and $h\in C^\infty_c(I)$.  By partition of unity on $S^1$, we see that equation \eqref{eq361} holds for any $h\in C^\infty(S^1)$.

	For any $m\in\mathbb Z$, we choose $h(e^{i\theta})=e^{-im\theta}$. Then the left hand side of equation \eqref{eq361} becomes
	\begin{align}
	&\sum_{s\in\mathbb R}\int_{-\pi}^\pi \langle \mathcal Y_\alpha(w^{(i)}_0,e^{i\theta})w^{(j)}|P_s\xi^{(k)} \rangle e^{i\Delta_\alpha\theta}\cdot e^{-im\theta}d\theta\nonumber\\
	=&\sum_{s\in\mathbb R}\int_{-\pi}^\pi\sum_{n\in\mathbb Z} \langle \mathcal Y_\alpha(w^{(i)}_0,\Delta_\alpha-1-n)w^{(j)}|P_s\xi^{(k)} \rangle\cdot e^{i(n-m)\theta}d\theta\nonumber\\
	=&\sum_{s\in\mathbb R}\sum_{n\in\mathbb Z} \int_{-\pi}^\pi\langle \mathcal Y_\alpha(w^{(i)}_0,\Delta_\alpha-1-n)w^{(j)}|P_s\xi^{(k)} \rangle\cdot e^{i(n-m)\theta}d\theta\nonumber\\
	=&2\pi\sum_{s\in\mathbb R}\langle \mathcal Y_\alpha(w^{(i)}_0,\Delta_\alpha-1-m)w^{(j)}|P_s\xi^{(k)} \rangle\nonumber\\
	=&2\pi\langle \mathcal Y_\alpha(w^{(i)}_0,\Delta_\alpha-1-m)w^{(j)}|\xi^{(k)} \rangle,
	\end{align}
	which by equation \eqref{eq361} must be zero. By corollary 2.15 and the proof of corollary A.4, vectors of the form $\mathcal Y_\alpha(w^{(i)}_0,s)w^{(j)}$ (where $s\in\mathbb R,w^{(j)}\in W_j$) span $W_k$, which is a dense subspace of $\mathcal H_j$. So $\xi^{(k)}=0$.\\
	
	Step 2.  Choose any $I\in\mathcal J$, and let $J\in\mathcal J(I^c\setminus\{-1\})$. Take $W_l=W_j\oplus^\perp W_k$. Then for each $v\in V,f\in C^\infty_c(I)$ we have $\overline{Y_l(v,f)}=\diag(\overline{Y_j(v,f)},\overline{Y_k(v,f)})$. For each $g\in C^\infty_c(J)$, we extend $\mathcal Y_\alpha(w^{(i)}_0,g)$ to an operator on $\mathcal H_l^\infty$ whose restriction to $\mathcal H_k^\infty$ is zero. We also regard $A=\mathcal Y_\alpha(w^{(i)}_0,g)$ as an unbounded operator on $\mathcal H_l$ with domain $\mathcal H^\infty_l$. Let $\mathcal N(I)$ be the von Neumann algebra on $\mathcal H_j$ generated by the operators $\overline{Y_l(v,f)}$ where $v\in V,f\in C^\infty_c(I)$, and let $\mathcal N(I)_\infty$ be the set of all $x\in\mathcal N(I)$ satisfying  $x\mathcal H^\infty_l\subset \mathcal H^\infty_l,x^*\mathcal H^\infty_l\subset \mathcal H^\infty_l$. Then as in the proof of proposition \ref{lb104}, $\mathcal N(I)_\infty$ is a strongly dense self-adjoint subalgebra of $\mathcal N(I)$. Let $H=(A+A^\dagger)/2$ and $K=(A-A^\dagger)/(2i)$ be symmetric unbounded operators on $\mathcal H_l$ with domain $\mathcal H^\infty_l$. Then by  proposition 2.13, corollary 3.13,  remark 3.14, and equation (3.26), for any $v\in V$ and $f\in C^\infty_c(I)$, $Y_l(v,f)$ commutes with $H$ and $K$ when acting on $\mathcal H^\infty_l$. By lemma B.8 and relations (3.38), (3.26), $\overline H$ and $\overline K$ are self adjoint, and by theorem B.9, $\overline{Y_l(v,f)}$ commutes strongly with $\overline H$ and $\overline K$. Hence any $x\in\mathcal N(I)$ commutes strongly with $\overline H$ and $\overline K$. In particular, if $x\in\mathcal N(I)_\infty$, we have $xH=Hx,xK=Kx$ when both sides of the equations act on $\mathcal H^\infty_l$. So $x(H+iK)=(H+iK)x$ when acting on $\mathcal H^\infty_l$. Therefore, $x\overline{\mathcal Y_\alpha(w^{(i)}_0,g)}\subset \overline{\mathcal Y_\alpha(w^{(i)}_0,g)}x$ for any $x\in\mathcal N(I)_\infty$, which implies that $\mathcal N(I)$ commutes strongly with $\overline{\mathcal Y_\alpha(w^{(i)}_0,g)}$. Thus $\overline{Y_l(v,f)}$ commutes strongly with $\overline{\mathcal Y_\alpha(w^{(i)}_0,g)}$.
	
	Let $\overline{\mathcal Y_\alpha(w^{(i)}_0,g)}=T_gH_g$ be the left polar decomposition of $\overline{\mathcal Y_\alpha(w^{(i)}_0,g)}$, where $T_g$ is the partial isometry. Then $T_g$ commutes strongly with each $\overline{Y_l(v,f)}$. By step 1, $\{T_g:g\in C^\infty_c(J) \}$ form a collection of bounded operators from $\mathcal H_j$ to $\mathcal H_k$ satisfying the conditions in proposition \ref{lb119}. Therefore, by that proposition, $W_k$ is strongly integrable.
\end{proof}

\section{Generalized intertwining operators}\label{lb112}

Generalized intertwining operators are nothing but genus $0$ correlation functions written in a particular way. Suppose that $\mathcal Y_{\sigma_2},\dots,\mathcal Y_{\sigma_n}$ is a chain of intertwining operators with charge spaces $W_{i_2},\dots,W_{i_n}$ respectively, such that the source space of $\mathcal Y_{\sigma_2}$ is $W_{i_1}$, and the target space of $\mathcal Y_{\sigma_n}$ is $W_i$. Choose $\mathcal Y_{\alpha}\in\mathcal Y{k\choose i~j}$. Choose $(z_1,\dots,z_n)\in\Conf_n(\mathbb C^\times)$, and choose arguments
 $\arg z_1,\arg (z_2-z_1),\dots,\arg(z_n-z_1)$. A \textbf{generalized intertwining operator} $\mathcal Y_{\sigma_n\cdots\sigma_n,\alpha}$ is defined near $(z_1,\dots,z_n)$ in the following two situations.

The first case is when $(z_1,\dots,z_n)$ satisfies $0<|z_2-z_1|<\cdots<|z_n-z_1|<|z_1|$. We define a $(W_j\otimes W_{i_1}\otimes\cdots\otimes W_{i_n}\otimes W_{\overline k} )^*$-valued holomorphic function $\mathcal Y_{\sigma_n\cdots\sigma_n,\alpha}$ near $(z_1,\dots,z_n)$ to satisfy that for any  $w^{(j)}\in W_j,w^{(i_1)}\in W_{i_1},\dots,w^{(i_n)}\in W_{i_n},w^{(\overline k)}\in W_{\overline k}$, 
\begin{align}
& \langle \mathcal Y_{\sigma_n\cdots\sigma_2,\alpha}(w^{(i_n)},z_n;\dots;w^{(i_2)},z_2;w^{(i_1)},z_1)w^{(j)},w^{(\overline k)}\rangle\nonumber\\
=&\big\langle \mathcal Y_\alpha\big(\mathcal Y_{\sigma_n}(w^{(i_n)},z_n-z_1)\cdots\mathcal Y_{\sigma_2}(w^{(i_2)},z_2-z_1)w^{(i_1)},z_1\big)w^{(j)},w^{(\overline k)}\big\rangle.\label{eq370}
\end{align}
The $V$-modules $W_{i_1},\dots,W_{i_n}$ are called the \textbf{charge spaces} of $ \mathcal Y_{\sigma_n\cdots\sigma_2,\alpha}$. $W_j$ is called the \textbf{source space} of $ \mathcal Y_{\sigma_n\cdots\sigma_2,\alpha}$, and $W_k$ is called the \textbf{target space} of $ \mathcal Y_{\sigma_n\cdots\sigma_2,\alpha}$.  The vector space of generalized intertwining operators with charge spaces $W_{i_1},\dots,W_{i_n}$, source space $W_j$, and target space $W_k$ is also denoted by $\mathcal V{k\choose i_n~\dots~i_1~j}$. 

In the second case, we choose $I\in\mathcal J$, and choose an arbitrary continuous argument function $\arg_I$ on $I$. We define $\mathscr O_n(I)$ to be the set of all $(z_1,\dots,z_n)\in\Conf_n(\mathbb C^\times)\cap I^n$ satisfying that for any $2\leq l<m\leq n$, either $\arg_I(z_lz_1^{-1})\arg_I(z_mz_1^{-1})<0$, or $|\arg_I(z_lz_1^{-1})|<|\arg_I(z_mz_1^{-1})|$. Our definition is clearly independent of the choice of $\arg_I$, and $\mathscr O_n(I)$ is a finite disconnected union of simply-connected sets. 

We want to define our generalized intertwining operators near any $(z_1,\dots,z_n)\in\mathscr O_n(I)$. To do this, we rotate $z_1,\dots,z_n$ along $I$ without meeting each other, until  these points satisfy $0<|z_2-z_1|<\cdots<|z_n-z_1|<|z_1|=1$. The arguments of $z_1,z_2-z_1,\dots,z_n-z_1$ are changed continuously. We first define  $\mathcal Y_{\sigma_n\cdots\sigma_n,\alpha}$ near the new point $(z_1,\dots,z_n)$ using equation \eqref{eq370}. Then we reverse this process of rotating $z_1,\dots,z_n$, and change $\mathcal Y_{\sigma_n\cdots\sigma_n,\alpha}$ continuously so as to be defined near the original point.

We now define the product of two generalized intertwining operators defined near $S^1$. Products of more than two generalized intertwining operators are defined in a similar way. Choose disjoint $I,J\in\mathcal J$,  choose  $(z_1,\dots,z_m)\in \mathscr O_m(I),(\zeta_1,\dots,\zeta_n)\in \mathscr O_n(J)$, and choose arguments $\arg z_1,\arg(z_2-z_1),\dots,\arg(z_m-z_1),\arg\zeta_1,\arg(\zeta_2-\zeta_1),\dots,\arg(\zeta_n-\zeta_1)$. Choose generalized intertwining operators $\mathcal Y_{\sigma_m\cdots\sigma_1,\alpha}\in\mathcal V{k\choose i_m\cdots i_1 i_0},\mathcal Y_{\rho_n\cdots\rho_1,\beta}\in\mathcal V{i_0\choose j_n\cdots j_1 j_0}$. If we choose $\arg z_2,\dots,\arg z_m,\arg\zeta_2,\dots,\arg\zeta_n$, then we can find uniquely chains of intertwining operators $\mathcal  Y_{\alpha_1},\dots,\mathcal Y_{\alpha_m}$ with charge spaces $W_{i_1},\dots,W_{i_m}$ respectively, and $\mathcal Y_{\beta_1},\dots,\mathcal Y_{\beta_n}$ with charge spaces $W_{j_1},\dots,W_{j_n}$ respectively, such that the source space of $\mathcal Y_{\beta_1}$ is $W_{j_0}$, that the source space of $\mathcal Y_{\alpha_1}$ and the target space of $\mathcal Y_{\beta_n}$ are $W_{i_0}$,  that the target space of $\mathcal Y_{\alpha_m}$ is $W_k$, and that for any $w^{(j_1)}\in W_{j_1},\dots,w^{(j_n)}\in W_{j_n},w^{(i_1)}\in W_{i_1},\dots,w^{(i_m)}\in W_{i_m}$, we have the fusion relations
\begin{gather}
\mathcal Y_{\sigma_m\cdots\sigma_2,\alpha}(w^{(i_m)},z_m;\dots;w^{(i_1)},z_1)
=\mathcal Y_{\alpha_m}(w^{(i_m)},z_m)\cdots\mathcal Y_{\alpha_1}(w^{(i_1)},z_1),\label{eq293}\\
\mathcal Y_{\rho_n\cdots\rho_2,\beta}(w^{(j_n)},\zeta_n;\dots;w^{(j_1)},\zeta_1)
=\mathcal Y_{\beta_n}(w^{(j_n)},\zeta_n)\cdots\mathcal Y_{\beta_1}(w^{(j_1)},\zeta_1).\label{eq294}
\end{gather}
We then define a $(W_{j_0}\otimes W_{j_1}\otimes\cdots\otimes W_{j_n}\otimes W_{i_1}\otimes\cdots\otimes W_{i_m}\otimes W_{\overline k})^*$-valued holomorphic function $\mathcal Y_{\sigma_m\cdots\sigma_1,\alpha}\mathcal Y_{\rho_n\cdots\rho_1,\beta}$ near $(\zeta_1,\dots,\zeta_n,z_1,\dots,z_m)$ to satisfy that for any $w^{(j_0)}\in W_{j_0},w^{(j_1)}\in W_{j_1},\dots,w^{(j_n)}\in W_{j_n},w^{(i_1)}\in W_{i_1},\dots,w^{(i_m)}\in W_{i_m},w^{(\overline k)}\in W_k$,
\begin{align}
&\langle \mathcal Y_{\sigma_m\cdots\sigma_2,\alpha}(w^{(i_m)},z_m;\dots;w^{(i_1)},z_1)\mathcal Y_{\rho_n\cdots\rho_2,\beta}(w^{(j_n)},\zeta_n;\dots;w^{(j_1)},\zeta_1)w^{(j_0)},w^{(\overline k)} \rangle\nonumber\\
=&\langle\mathcal Y_{\alpha_m}(w^{(i_m)},z_m)\cdots\mathcal Y_{\alpha_1}(w^{(i_1)},z_1)\mathcal Y_{\beta_n}(w^{(j_n)},\zeta_n)\cdots\mathcal Y_{\beta_1}(w^{(j_1)},\zeta_1)w^{(j_0)},w^{(\overline k)} \rangle.\label{eq292}
\end{align}

\begin{rem}\label{lbb1}
It is clear that our definition does not depend on the choice of $\arg z_2,\dots,\arg z_m,\arg\zeta_2,\dots,\arg\zeta_n$. Moreover, if we choose $\varsigma\in S_m,\varpi\in S_n$, and \emph{real variables} $\lambda_1,\dots,\lambda_n,r_1,\dots,r_m$ defined near $1$ and satisfying $0<\lambda_{\varpi(1)}<\cdots,\lambda_{\varpi(n)}<r_{\varsigma(1)}<\cdots<r_{\varsigma(m)}$, then the following series
\begin{align}
\sum_{s\in\mathbb R}\langle \mathcal Y_{\sigma_m\cdots\sigma_2,\alpha}(w^{(i_m)},r_mz_m;\dots;w^{(i_1)},r_1z_1)P_s\mathcal Y_{\rho_n\cdots\rho_2,\beta}(w^{(j_n)},\lambda_n\zeta_n;\dots;w^{(j_1)},\lambda_1\zeta_1)w^{(j_0)},w^{(\overline k)} \rangle\label{eq371}
\end{align}
of $s$ converges absolutely, and by proposition 2.11, as $r_1,\dots,r_m,\lambda_1,\dots,\lambda_n\rightarrow 1$, the limit of \eqref{eq371} exists and equals the left hand side of equation \eqref{eq292}.
\end{rem}

\subsection{Braiding of generalized intertwining operators}

\begin{thm}\label{lb16}
	
Choose disjoint $I,J\in\mathcal J$. Choose $(z_1,\dots,z_m)\in \mathscr O_m(I),(\zeta_1,\dots,\zeta_n)\in \mathscr O_n(J)$. Choose arguments  $\arg z_1,\arg \zeta_1,\arg(z_2-z_1),\dots,\arg(z_m-z_1),\arg(\zeta_n-\zeta_1),\dots,\arg(\zeta_n-\zeta_1)$.	Let $W_i,W_j,W_{i_1},W_{i_2},\dots,W_{i_m},W_{j_1},W_{j_2},\dots,W_{j_n}$ be $V$-modules. Assume that for any $w^{(i)}\in W_i,w^{(j)}\in W_j$, the braid relation
	\begin{align}
	\mathcal Y_\alpha(w^{(i)},z_1)\mathcal Y_\beta(w^{(j)},\zeta_1)=\mathcal Y_{\beta'}(w^{(j)},\zeta_1)\mathcal Y_{\alpha'}(w^{(i)},z_1)
	\end{align}
	holds. Then for any intertwining operators $\mathcal Y_{\sigma_2},\dots,\mathcal Y_{\sigma_m},\mathcal Y_{\rho_2},\dots,\mathcal Y_{\rho_n}$, any $w^{(i_1)}\in W_{i_1},\dots,w^{(i_m)}\in W_{i_m},w^{(j_1)}\in W_{j_1},\dots,w^{(j_n)}\in W_{j_n}$, we have the generalized braid relation
	\begin{align}
&\mathcal Y_{\sigma_m\cdots\sigma_2,\alpha}(w^{(i_m)},z_m;\dots;w^{(i_1)},z_1)	\mathcal Y_{\rho_n\cdots\rho_2,\beta}(w^{(j_n)},\zeta_n;\dots;w^{(j_1)},\zeta_1)\nonumber\\
=&\mathcal Y_{\rho_n\cdots\rho_2,\beta'}(w^{(j_n)},\zeta_n;\dots;w^{(j_1)},\zeta_1)\mathcal Y_{\sigma_m\cdots\sigma_2,\alpha'}(w^{(i_m)},z_n;\dots;w^{(i_1)},z_1).\label{eq41}
	\end{align}
(Note that here, as before, we follow convention 2.19 to simplify our statement.)	
\end{thm}

\begin{proof}
	By analytic continuation, it suffices to assume that $|z_1-\zeta_1|$ is small enough with respect to $1$, and  $|z_2-z_1|,\dots,|z_m-z_1|,|\zeta_2-\zeta_1|,\dots,|\zeta_n-\zeta_1|$ are small enough with respect to $|z_1-\zeta_1|$, such that for any \emph{real variables} $r,\lambda>0$ satisfying $\frac 23<\frac r \lambda <\frac 32$, the following inequalities are satisfied:
\begin{gather}
|\zeta_n-\zeta_1|+|z_m-z_1|<1/4,\\
0<|\lambda\zeta_2-\lambda \zeta_1|<|\lambda\zeta_3-\lambda\zeta_1|<\cdots<|\lambda\zeta_n-\lambda\zeta_1|<|rz_1-\lambda\zeta_1|-|rz_m-rz_1|,\\
0<|rz_2-rz_1|<|rz_3-rz_1|<\cdots<|rz_m-rz_1|<|rz_1-\lambda\zeta_1|<\lambda-|rz_m-rz_1|.
\end{gather}

Choose $\arg(z_1-\zeta_1)$. Since $|z_1-\zeta_1|<1$, there exist intertwining operators $\mathcal Y_{\gamma}$ and $\mathcal Y_{\delta}$ such that for any $w^{(i)}\in W_i,w^{(j)}\in W_j$, we have
	\begin{flalign}
\mathcal Y_\alpha(w^{(i)},z_1)\mathcal Y_\beta(w^{(j)},\zeta_1)=\mathcal Y_{\delta}(\mathcal Y_{\gamma}(w^{(i)},z_1-\zeta_1)w^{(j)},\zeta_1)=\mathcal Y_{\beta'}(w^{(j)},\zeta_1)\mathcal Y_{\alpha'}(w^{(i)},z_1).
	\end{flalign}
Choose \emph{real variables} $r_m>\cdots>r_1>\lambda_n>\dots>\lambda_1>0$ satisfying $2/3<r_1/\lambda_1<3/2$. (Recall that the arguments of real variables are chosen to be $0$.) When $r_2/r_1,\dots,r_m/r_1,\lambda_2/\lambda_1,\dots,\lambda_n/\lambda_1$ are close to  $1$, by corollary 2.7, the right hand side of the equation
	\begin{flalign}
&\mathcal Y_\alpha\big(\mathcal Y_{\sigma_m}(w^{(i_m)},r_mz_m-r_1z_1)\cdots\mathcal Y_{\sigma_2}(w^{(i_2)},r_2z_2-r_1z_1)w^{(i_1)},r_1z_1\big)\nonumber\\
&\cdot\mathcal Y_{\beta}\big(\mathcal Y_{\rho_n}(w^{(j_n)},\lambda_n\zeta_n-\lambda_1\zeta_1)\cdots\mathcal Y_{\rho_2}(w^{(j_2)},\lambda_2\zeta_2-\lambda_1\zeta_1)w^{(j_1)},\lambda_1\zeta_1\big)\nonumber\\
=&\mathcal Y_{\delta}\bigg(\mathcal Y_{\gamma}\Big(\mathcal Y_{\sigma_m}(w^{(i_m)},r_mz_m-r_1z_1)\cdots\mathcal Y_{\sigma_2}(w^{(i_2)},r_2z_2-r_1z_1)w^{(i_1)},r_1z_1-\lambda_1\zeta_1\Big)\nonumber\\
&\quad~\cdot\mathcal Y_{\rho_n}(w^{(j_n)},\lambda_n\zeta_n-\lambda_1\zeta_1)\cdots\mathcal Y_{\rho_2}(w^{(j_2)},\lambda_2\zeta_2-\lambda_1\zeta_1)w^{(j_1)},\lambda_1\zeta_1\bigg)\label{eq281}
\end{flalign}
converges absolutely and locally uniformly. If moreover $r_1/\lambda_1=4/3$, then by theorem 2.6, the left hand side of equation \eqref{eq281} also converges absolutely and locally uniformly, and hence equation \eqref{eq281} holds.

Now we let $r_1,\dots,r_m,\lambda_1,\dots,\lambda_n\rightarrow 1$, then  the left hand side of equation \eqref{eq281} converges to the left hand side of equation \eqref{eq41}, and the right hand side of 
	\eqref{eq281} converges to
	\begin{align}
	&\mathcal Y_{\delta}\bigg(\mathcal Y_{\gamma}\Big(\mathcal Y_{\sigma_m}(w^{(i_m)},z_m-z_1)\cdots\mathcal Y_{\sigma_2}(w^{(i_2)},z_2-z_1)w^{(i_1)},z_1-\zeta_1\Big)\nonumber\\
	&\quad~\cdot\mathcal Y_{\rho_n}(w^{(j_n)},\zeta_n-\zeta_1)\cdots\mathcal Y_{\rho_2}(w^{(j_2)},\zeta_2-\zeta_1)w^{(j_1)},\zeta_1\bigg).\label{eq42}
	\end{align}
Therefore, the left hand side of equation	\eqref{eq41} equals \eqref{eq42}. The same  argument shows that the right hand side of \eqref{eq41} also equals \eqref{eq42}. This finishes our proof.
\end{proof}

Note that it is easy to generalize proposition 2.11 to generalized intertwining operators.

\subsection{The adjoint relation for generalized intertwining operators}

This section is devoted to the proof of  the adjoint relation for generalized intertwining operators \eqref{eq314}. We first recall that if $\mathcal Y_\alpha$ is a unitary intertwining operator of a unitary $V$, $z\in S^1$ with chosen argument, and $w^{(i)}\in W_i$ is quasi-primary, then	by relation (1.34), 
\begin{align}
\mathcal Y_\alpha(w^{(i)},z)^\dagger=e^{-i\pi\Delta_{w^{(i)}}}z^{2\Delta_{w^{(i)}}}\mathcal Y_{\alpha^*}(\overline{w^{(i)}},z).\label{eq291}
\end{align}
We want to obtain a similar relation for generalized intertwining operators. To achieve this goal, we first  need an auxiliary fusion relation. Recall that for any $V$-module $W_i$, we have the creation operator $\mathcal Y^i_{i0}=B_\pm Y_i$ of $W_i$, and the annihilation operator $\mathcal Y^0_{i\overline i}=C^{-1}\mathcal Y^i_{i0}$ of $W_{\overline i}$. We set $\Upsilon^0_{i\overline i}=C\mathcal Y^i_{i0}$. Then similar to equation (1.40), for any $w^{(i)}_1\in W_i,w^{(\overline i)}_2\in W_{\overline i}$ we have
\begin{align}
\langle \Upsilon^0_{ i \bar i}(w^{(i)}_1,x)w^{(\overline i)}_2,\Omega\rangle=\langle e^{x^{-1}L_{1}}w^{(\overline i)}_2,(e^{-i\pi}x^{-2})^{L_0}e^{-x^{-1}L_1}w^{(i)}_1\rangle.\label{eq284}
\end{align}
\begin{pp}[Fusion with annihilation operators]\label{lb2}
	Let $z_1,z_2\in\mathbb C^\times$ satisfy $0<|z_1|,|z_1-z_2|<|z_2|$. Choose $\arg z_2$, let $\arg z_1$ be close to $\arg z_2$ as $z_1\rightarrow z_2$, and let $\arg(z_2-z_1)$ be close to $\arg z_2$ as $z_1\rightarrow 0$. Then  for any  $\mathcal Y_\alpha\in\mathcal V{k\choose i~j}, w^{(i)}\in W_i$ and $w^{(j)}\in W_j$, we have the  fusion relation
	\begin{align}
	\Upsilon^0_{k\overline k}\big(\mathcal Y_\alpha(w^{(i)},e^{i\pi}(z_2-z_1))w^{(j)},z_2\big)=\Upsilon^0_{j\overline j}(w^{(j)},z_2)\mathcal Y_{C\alpha}(w^{(i)},z_1).\label{eq14}
	\end{align}
\end{pp}

\begin{proof}
	Let us assume that $z_1,z_2\in\mathbb R_{>0}$ and $0<z_2-z_1<z_1<z_2$. If the proposition is proved for this special case, then by analytic continuation, it also holds in general.

	Therefore, we assume that $\arg z_1=\arg z_2=\arg(z_2-z_1)=0$. Let $\arg(z_1^{-1}-z_2^{-1})$ be close to $\arg(z_1^{-1})=-\arg z_1$ as $z_2^{-1}\rightarrow 0$. (The reason for choosing this argument is to use lemma 2.16-(1). Recall also convention 1.12 in \cite{Gui19a}.) Then it is obvious that $\arg(z_1^{-1}-z_2^{-1})=0=\arg(\frac{z_2-z_1}{z_1z_2})$. 
	
	We now use equation \eqref{eq284} and the definition of $C\alpha$ to compute that
	\begin{align}
	&\langle \Upsilon^0_{j\overline j}(w^{(j)},z_2)\mathcal Y_{C\alpha}(w^{(i)},z_1)w^{(\overline k)},\Omega\rangle\nonumber\\
	=&\sum_{s\in\mathbb R}\langle\Upsilon^0_{j\overline j}(w^{(j)},z_2)P_s\mathcal Y_{C\alpha}(w^{(i)},z_1)w^{(\overline k)},\Omega\rangle\nonumber\\
	=&\sum_{s\in\mathbb R}\langle e^{z_2^{-1}L_1}P_s\mathcal Y_{C\alpha}(w^{(i)},z_1)w^{(\overline k)},(e^{-i\pi}z_2^{-2})^{L_0}e^{-z_2^{-1}L_1}w^{(j)}\rangle\nonumber\\
	=&\sum_{s\in\mathbb R}\big\langle w^{(\overline k)},\mathcal Y_\alpha\big(e^{z_1L_1}(e^{-i\pi}z_1^{-2})^{L_0}w^{(i)},z_1^{-1}\big) P_se^{z_2^{-1}L_{-1}}(e^{-i\pi}z_2^{-2})^{L_0}e^{-z_2^{-1}L_1}w^{(j)}\big\rangle,\label{eq8}
	\end{align}
	which, according to lemma 2.16-(1), converges absolutely and equals
	\begin{align}
	\Big\langle w^{(\overline k)},e^{z_2^{-1}L_{-1}}\mathcal Y_\alpha\Big(e^{z_1L_1}(e^{-i\pi}z_1^{-2})^{L_0}w^{(i)},\frac{z_2-z_1}{z_1z_2}\Big) (e^{-i\pi}z_2^{-2})^{L_0}e^{-z_2^{-1}L_1}w^{(j)}\Big\rangle.
	\end{align}
	By (1.26) and (1.30), the above formula equals
	\begin{align}
	&\Big\langle w^{(\overline k)},e^{z_2^{-1}L_{-1}}(e^{-i\pi}z_2^{-2})^{L_0}\nonumber\\
	&~~\cdot\mathcal Y_\alpha\Big((e^{i\pi}z_2^2)^{L_0}e^{z_1L_1}(e^{-i\pi}z_1^{-2})^{L_0}w^{(i)},e^{i\pi}(z_2-z_1)\frac {z_2}{z_1}\Big) e^{-z_2^{-1}L_1}w^{(j)}\Big\rangle\nonumber\\
	=&\Big\langle w^{(\overline k)},e^{z_2^{-1}L_{-1}}(e^{-i\pi}z_2^{-2})^{L_0}\nonumber\\
	&~~\cdot\mathcal Y_\alpha\Big(e^{-z_1z_2^{-2}L_1}\Big(\frac{z_2}{z_1}\Big)^{2L_0}w^{(i)},e^{i\pi}(z_2-z_1)\frac {z_2}{z_1}\Big) e^{-z_2^{-1}L_1}w^{(j)}\Big\rangle.\label{eq285}
	\end{align}
	
	On the other hand, we have
\begin{align}
&\big\langle\Upsilon^0_{k\overline k}\big(\mathcal Y_\alpha(w^{(i)},e^{i\pi}(z_2-z_1))w^{(j)},z_2\big)w^{(\overline k)},\Omega\big\rangle\nonumber\\
=&\sum_{s\in\mathbb R}\big\langle\Upsilon^0_{k\overline k}\big(P_s\mathcal Y_\alpha(w^{(i)},e^{i\pi}(z_2-z_1))w^{(j)},z_2\big)w^{(\overline k)},\Omega\big\rangle\nonumber\\
=&\sum_{s\in\mathbb R}\langle w^{(\overline k)},e^{z_2^{-1}L_{-1}}(e^{-i\pi}z_2^{-2})^{L_0}e^{-z_2^{-1}L_1}P_s\mathcal Y_\alpha(w^{(i)},e^{i\pi}(z_2-z_1))w^{(j)}\rangle.\label{eq5}
\end{align}
Note that $|-z_2^{-1}|<|e^{i\pi}(z_2-z_1)|^{-1}$. Let $\arg(1-e^{i\pi}(z_2-z_1)\cdot (-z_2^{-1}))$ be close to $\arg(1-e^{i\pi}(z_2-z_1)\cdot 0)=0$ as 
	$-z_2^{-1}\rightarrow 0$. Then clearly $\arg(1-e^{i\pi}(z_2-z_1)\cdot (-z_2^{-1}))=0=\arg(\frac {z_1}{z_2})$. We can use lemma 2.16-(2) to compute that \eqref{eq5} equals \eqref{eq285}. This proves equation \eqref{eq14} when both sides act on $\Omega$. By the proof of corollary 2.15, equation \eqref{eq14} holds when acting on any vector inside $V$. 
\end{proof}

\begin{rem}
By proposition 2.9 and the above property, we have the fusion relation
\begin{align}
\Upsilon^0_{k\overline k}\big(\mathcal Y_{B_+\alpha}(w^{(j)},z_2-z_1)w^{(i)},z_1\big)=\Upsilon^0_{j\overline j}(w^{(j)},z_2)\mathcal Y_{C\alpha}(w^{(i)},z_1)
\end{align}
 when  $0<|z_2-z_1|<|z_1|<|z_2|$, $\arg z_1$ is close to $\arg z_2$ as $z_1\rightarrow z_2$, and $\arg (z_2-z_1)$ is close to $\arg z_2$ as $z_1\rightarrow 0$. Similarly, we can also show that
\begin{align}
\mathcal Y^0_{k\overline k}\big(\mathcal Y_{B_-\alpha}(w^{(j)},z_2-z_1)w^{(i)},z_1\big)=\mathcal Y^0_{j\overline j}(w^{(j)},z_2)\mathcal Y_{C^{-1}\alpha}(w^{(i)},z_1).
\end{align}
\end{rem}

\begin{thm}[Fusion of contragredient intertwining operators]\label{lb3}
	Let $z_1,\dots,z_n,z_1',\dots,z_n'\in\mathbb C^\times$  satisfy the following conditions:\\
	(1) $0<|z_1|<|z_2|<\dots<|z_n|$ and $0<|z_2-z_1|<\dots<|z_n-z_1|<|z_1|$;\\
	(1') $|z_1'|>|z_2'|>\dots>|z_n'|>0$ and $0<|z_2'-z_1'|<\dots<|z_n'-z_1'|<|z_1'|$.\\
	Choose arguments $\arg z_1,\arg z'_1$. For each $2\leq m\leq n$, we choose arguments $\arg(z_m-z_1),\arg(z'_m-z'_1)$.  Let $\arg z_m$ be close to $\arg z_1$ as $z_m\rightarrow z_1$, and let $\arg z'_m$ be close to $\arg z'_1$ as $z'_m\rightarrow z'_1$.

Let $W_{i_1},\dots,W_{i_n}$, and $W_i$ be $V$-modules, and let $\mathcal Y_{\sigma_2},\dots,\mathcal Y_{\sigma_n}$ be a chain of intertwining operators of $V$ satisfying the following conditions:\\
	(a) for each $2\leq m\leq n$,  the charge space of $\mathcal Y_{\sigma_m}$ is $W_{i_m}$;\\
	(b) the source space of $\mathcal Y_{\sigma_2}$ is $W_{i_1}$;\\
	(c) the target space of $\mathcal Y_{\sigma_n}$ is $W_{i}$.\\
	Then  there exists a chain of  intertwining operators $\mathcal Y_{\sigma_2'},\dots,\mathcal Y_{\sigma_n'}$, whose types are the same as those of $\mathcal Y_{\sigma_2},\dots,\mathcal Y_{\sigma_n}$ respectively, such that for any  $\mathcal Y_\alpha\in\mathcal V{k\choose i~j}$ , if $\mathcal Y_{\alpha_1},\mathcal Y_{\alpha_2},\dots,\mathcal Y_{\alpha_n}$ is a chain of intertwining operators of $V$ satisfying the following conditions:\\
	(i)	 for each $1\leq m\leq n$, the charge space of $\mathcal Y_{\alpha_m}$ is $W_{i_m}$;\\
	(ii) the source space of $\mathcal Y_{\alpha_1}$ is $W_j$;\\
	(iii)  the target space of $\mathcal Y_{\alpha_n}$ is $W_k$;\\
	(iv) for any $w^{(i_1)}\in W_{i_1},\dots,w^{(i_n)}\in W_{i_n}$, we have the  fusion relation
	\begin{align}
	&\mathcal Y_\alpha\big(\mathcal Y_{\sigma_n}(w^{(i_n)},z_n-z_1)\cdots\mathcal Y_{\sigma_2}(w^{(i_2)},z_2-z_1)w^{(i_1)},z_1\big)\nonumber\\
	=&\mathcal Y_{\alpha_n}(w^{(i_n)},z_n)\cdots\mathcal Y_{\alpha_1}(w^{(i_1)},z_1),\label{eq15}
	\end{align}
	then the following fusion relation also holds:
	\begin{align}
	&\mathcal Y_{C\alpha}\big( {\mathcal Y}_{\sigma'_n}(w^{(i_n)},z_n'-z_1')\cdots {\mathcal Y}_{\sigma'_2}(w^{(i_2)},z_2'-z_1')w^{(i_1)},z_1'\big)\nonumber\\
	=&\mathcal Y_{C\alpha_1}(w^{(i_1)},z_1')\cdots\mathcal Y_{C\alpha_n}(w^{(i_n)},z_n').\label{eq18}
	\end{align}
\end{thm}

\begin{proof}
	Let $W_{j_1},\dots,W_{j_{n-1}}$  be the target spaces of $\mathcal Y_{\alpha_1},\dots,\mathcal Y_{\alpha_{n-1}}$ respectively. Choose $\zeta_0',\zeta'_1,\cdots,\zeta'_n\in\mathbb R_{<0}$ satisfying  $\zeta'_0<\zeta'_1<\cdots<\zeta'_n<0$ and $|\zeta'_0-\zeta'_1|>|\zeta'_1-\zeta'_n|$. Let $\zeta_1=\zeta_1'-\zeta_0',\dots,\zeta_n=\zeta_n'-\zeta_0'$.
	Let $\arg\zeta'_0=\arg\zeta'_1=\cdots=\arg\zeta'_n=-\pi,\arg\zeta_1=\arg(\zeta'_1-\zeta'_0)=0,\dots,\arg\zeta_n=\arg(\zeta'_n-\zeta'_0)=0$. Note that for any $2\leq m\leq n$, $\zeta_m-\zeta_1=\zeta'_m-\zeta'_1$. We let $\arg(\zeta_m-\zeta_1)=\arg(\zeta'_m-\zeta'_1)=0$. 
	
	We now rotate and stretch these points, so that for each $1\leq m\leq n$, $\zeta_m$ is moved to $\widetilde z_m=z_m$, $\zeta'_m$ is moved to $\widetilde z'_m=z'_m$, $\arg \zeta_m$ becomes $\arg \widetilde z_m=\arg z_m$, and $\arg \zeta'_m$ becomes $\arg \widetilde z'_m=\arg z'_m$. We assume that during this process, conditions  (1) and (1') are always satisfied. (Note that such process might not exist if the choice of $\arg z_2,\arg z_3,\dots$ and $\arg z'_2,\arg z'_3,\dots$ are arbitrary with respect to $\arg z_1$ and $\arg z'_1$. ) Denote this process by $(\mathbf{P})$. Then under this process, for each $2\leq m\leq n$, $\arg(\zeta_m-\zeta_1)$ is changed to an argument $\arg(\widetilde z_m-\widetilde z_1)$ of $\widetilde z_m-\widetilde z_1$, and $\arg(\zeta'_m-\zeta'_1)$ is changed to an argument $\arg(\widetilde z'_m-\widetilde z'_1)$ of $\widetilde z'_m-\widetilde z'_1$ accordingly. Since $\arg(\widetilde z_m-\widetilde z_1)\in \arg(z_m- z_1)+2i\pi\mathbb Z$ and $\arg(\widetilde z'_m-\widetilde z'_1)\in \arg(z'_m- z'_1)+2i\pi\mathbb Z$, there exist intertwining operators $\mathcal Y_{\widetilde \sigma_m},\mathcal Y_{\sigma'_m}$ of the same type as that of $\mathcal Y_{\sigma_m}$, such that for any $w^{(i_m)}\in W_{i_m}$, 
	\begin{gather*}
	\mathcal Y_{\widetilde \sigma_m}(w^{(i_m)},\widetilde z_m-\widetilde z_1)=\mathcal Y_{\sigma_m}(w^{(i_m)},z_m-z_1),\\
	\mathcal Y_{\sigma'_m}(w^{(i_m)},z'_m-z'_1)=\mathcal Y_{\widetilde \sigma_m}(w^{(i_m)},\widetilde z'_m-\widetilde z'_1).
	\end{gather*}
	Then equation \eqref{eq15} implies that
	\begin{align}
	&\mathcal Y_\alpha\big(\mathcal Y_{\widetilde\sigma_n}(w^{(i_n)},\widetilde z_n-\widetilde z_1)\cdots\mathcal Y_{\widetilde\sigma_2}(w^{(i_2)},\widetilde z_2-\widetilde z_1)w^{(i_1)},\widetilde z_1\big)\nonumber\\
	=&\mathcal Y_{\alpha_n}(w^{(i_n)},\widetilde z_n)\cdots\mathcal Y_{\alpha_1}(w^{(i_1)},\widetilde z_1).
	\end{align}
	By reversing  process $(\mathbf P)$, the above equation is analytically continued to the equation
	\begin{align}
	&\mathcal Y_\alpha\big(\mathcal Y_{\widetilde\sigma_n}(w^{(i_n)},\zeta_n'-\zeta_1')\cdots\mathcal Y_{\widetilde\sigma_2}(w^{(i_2)},\zeta_2'-\zeta_1')w^{(i_1)},\zeta_1'-\zeta_0'\big)\nonumber\\
	=&\mathcal Y_{\alpha_n}(w^{(i_n)},\zeta_n'-\zeta_0')\cdots\mathcal Y_{\alpha_1}(w^{(i_1)},\zeta_1'-\zeta_0')\label{eq16}.
	\end{align}
	
	For any $1\leq m\leq n$, we let $\arg(\zeta'_0-\zeta'_m)$ be close to $\arg \zeta'_0=-\pi$ as $\zeta'_m\rightarrow 0$. Then $\arg(\zeta'_0-\zeta'_m)=-\pi$, and hence $\zeta'_m-\zeta'_0=e^{i\pi}(\zeta_0-\zeta_m)$. Choose arbitrary $w^{(j)}\in W_j$.  Then by lemma \ref{lb2}, we have
	\begin{align}
	&\Upsilon^0_{j\overline j}(w^{(j)},\zeta_0')\mathcal Y_{C\alpha_1}(w^{(i_1)},\zeta_1')\cdots\mathcal Y_{C\alpha_n}(w^{(i_n)},\zeta_n')\nonumber\\
	=&\Upsilon^0_{j_1\overline{j_1}}\big(\mathcal Y_{\alpha_1}(w^{(i_1)},\zeta_1'-\zeta_0')w^{(j)},\zeta_0'\big)\mathcal Y_{C\alpha_2}(w^{(i_2)},\zeta_2')\nonumber\\
	&\cdot\mathcal Y_{C\alpha_3}(w^{(i_3)},\zeta_3')\cdots \mathcal Y_{C\alpha_n}(w^{(i_n)},\zeta_n')\nonumber\\
	=&\Upsilon^0_{j_2\overline{j_2}}\big(\mathcal Y_{\alpha_2}(w^{(i_2)},\zeta_2'-\zeta_0')\mathcal Y_{\alpha_1}(w^{(i_1)},\zeta_1'-\zeta_0')w^{(j)},\zeta_0'\big)\nonumber\\
	&\cdot\mathcal Y_{C\alpha_3}(w^{(i_3)},\zeta_3')\cdots \mathcal Y_{C\alpha_n}(w^{(i_n)},\zeta_n')\nonumber\\
	&\qquad\qquad\qquad\qquad\qquad\vdots\nonumber\\
	=&\Upsilon^0_{k\overline k}\big(\mathcal Y_{\alpha_n}(w^{(i_n)},\zeta_n'-\zeta_0')\cdots \mathcal Y_{\alpha_1}(w^{(i_1)},\zeta_1'-\zeta_0')w^{(j)},\zeta_0'\big),\label{eq287}
	\end{align}
	where, by theorem 2.6, the expression in each step converges absolutely. By \eqref{eq16}, expression \eqref{eq287} equals
	\begin{align}
	\Upsilon^0_{k\overline k}\Big(\mathcal Y_\alpha\big(\mathcal Y_{\widetilde\sigma_n}(w^{(i_n)},\zeta_n'-\zeta_1')\cdots \mathcal Y_{\widetilde\sigma_2}(w^{(i_2)},\zeta_2'-\zeta_1')w^{(i_1)},\zeta_1'-\zeta_0'\big)w^{(j)},\zeta_0'\Big),\label{eq17}
	\end{align}
	the absolute convergence of which is guaranteed by corollary 2.7. Again by proposition \ref{lb2}, equation \eqref{eq17} equals
	\begin{align}
	\Upsilon^0_{j\overline j}(w^{(j)},\zeta_0')\mathcal Y_{C\alpha}\big(\mathcal Y_{\widetilde\sigma_n}(w^{(i_n)},\zeta_n'-\zeta_1')\cdots \mathcal Y_{\widetilde\sigma_2}(w^{(i_2)},\zeta_2'-\zeta_1')w^{(i_1)},\zeta_1'\big),\label{eq288}
	\end{align}
	the absolute convergence of which follows from theorem 2.6. Therefore, the left hand side of equation \eqref{eq287} equals \eqref{eq288}. By proposition 2.3, we obtain
	\begin{align}
	&\mathcal Y_{C\alpha}\big(\mathcal Y_{\widetilde\sigma_n}(w^{(i_n)},\zeta_n'-\zeta_1')\cdots \mathcal Y_{\widetilde\sigma_2}(w^{(i_2)},\zeta_2'-\zeta_1')w^{(i_1)},\zeta_1'\big)\nonumber\\
	=&\mathcal Y_{C\alpha_1}(w^{(i_1)},\zeta_1')\cdots\mathcal Y_{C\alpha_n}(w^{(i_n)},\zeta_n'). \label{eq289}
	\end{align}
	
	Now we do process $(\mathbf P)$. Then \eqref{eq289} is analytically continued to the equation
	\begin{flalign}
	&\mathcal Y_{C\alpha}\big( {\mathcal Y}_{\widetilde\sigma_n}(w^{(i_n)},\widetilde z_n'-\widetilde z_1')\cdots {\mathcal Y}_{\widetilde\sigma_2}(w^{(i_2)},\widetilde z_2'-\widetilde z_1')w^{(i_1)},\widetilde z_1'\big)\nonumber\\
	=&\mathcal Y_{C\alpha_1}(w^{(i_1)},\widetilde z_1')\cdots\mathcal Y_{C\alpha_n}(w^{(i_n)},\widetilde z_n'),
	\end{flalign}
	which implies \eqref{eq18}. Thus the proof is completed.
\end{proof}

\begin{rem}\label{lb87}
Choose (not necessarily disjoint) $I,J\in\mathcal J$, and choose  $(z_1,\dots,z_n)\in\mathscr O_n(I),(z_1',\dots,z'_n)\in \mathscr O_n(J)$. Choose continuous argument functions $\arg_I,\arg_J$ on $I,J$ respectively, and let $\arg z_1=\arg_I(z_1),\dots,\arg z_n=\arg_I(z_n),\arg z'_1=\arg_J(z'_1),\dots,\arg z'_n=\arg_J(z'_n)$.  For each $2\leq m\leq n$ we choose arguments $\arg(z_m-z_1)$ and $\arg(z'_m-z'_1)$. Then by theorem \ref{lb3}, for any chain of intertwining operators $\mathcal Y_{\sigma_2},\dots, \mathcal Y_{\sigma_n}$ satisfying conditions (a), (b), and (c) of theorem \ref{lb3}, there exists a chain of intertwining operators $\mathcal Y_{\sigma_2'},\dots,\mathcal Y_{\sigma_n'}$ whose types are the same as those of $\mathcal Y_{\sigma_2},\dots, \mathcal Y_{\sigma_n}$ respectively, such that
\begin{align}
\mathcal Y_{\sigma_n\cdots\mathcal\sigma_2,\alpha}(w^{(i_n)},z_n;\dots,w^{(i_1)},z_1)=\mathcal Y_{\alpha_n}(w^{(i_n)},z_n)\cdots\mathcal Y_{\alpha_1}(w^{(i_1)},z_1)\label{eqb1}
\end{align}
always implies
\begin{align}
\mathcal Y_{\sigma'_n\cdots\sigma'_2,C\alpha}(w^{(i_n)},z'_n;\dots,w^{(i_1)},z'_1)=\mathcal Y_{C\alpha_1}(w^{(i_1)},z_1')\cdots\mathcal Y_{C\alpha_n}(w^{(i_n)},z_n').\label{eqb2}
\end{align}
\end{rem}

\begin{co}[Adjoint of generalized intertwining operators]
Let $V$ be unitary. Let $I\in \mathcal J$, choose $(z_1,\dots,z_n)\in\mathscr O_n(I)$, and choose arguments $\arg z_1,\arg(z_2-z_1),\dots,\arg(z_n-z_1)$. Let $W_{i_1},\dots,W_{i_n}$, and $W_i$ be unitary $V$-modules, and let $\mathcal Y_{\sigma_2},\dots,\mathcal Y_{\sigma_n}$ be a chain of unitary intertwining operators of $V$ satisfying the following conditions:\\
	(a) for each $2\leq m\leq n$,  the charge space of $\mathcal Y_{\sigma_m}$ is $W_{i_m}$;\\
	(b) the source space of $\mathcal Y_{\sigma_2}$ is $W_{i_1}$;\\
	(c) the target space of $\mathcal Y_{\sigma_n}$ is $W_{i}$.\\
Then for each $2\leq m\leq n$, there exists a unitary intertwining operator $\mathcal Y_{\widetilde\sigma_m}$ whose type is the same as that of $\mathcal Y_{\overline{\sigma_m}}$, such that for   any unitary $\mathcal Y_\alpha\in\mathcal V{k\choose i~j}$, and any nonzero quasi-primary vectors $w^{(i_1)}\in W_{i_1},\dots,w^{(i_n)}\in W_{i_n}$, we have
	\begin{align}
	&\mathcal Y_{\sigma_n\cdots\sigma_2,\alpha}\big(w^{(i_n)},z_n;\dots;w^{(i_1)},z_1\big)^\dagger\nonumber\\
	=&e^{-i\pi\big(\Delta_{w^{(i_1)}}+\cdots+\Delta_{w^{(i_n)}}\big)}z_1^{2\Delta_{w^{(i_1)}}}\cdots z_n^{2\Delta_{w^{(i_n)}}}\cdot\mathcal Y_{\widetilde\sigma_n\cdots\widetilde \sigma_2,\alpha^*}\big(\overline{w^{(i_n)}},z_n;\dots;\overline{w^{(i_1)}},z_1\big),\label{eq314}
	\end{align}
where the formal adjoint is defined for evaluations of the  operators between the vectors inside $W_j$ and $W_k$.
\end{co}

\begin{proof}
Let $\arg_I$ be the continuous argument function on $I$ satisfying $\arg_I(z_1)=\arg z_1$. We let $\arg z_2=\arg_I(z_2),\dots,\arg z_n=\arg_I(z_n)$. Recall that by convention 1.12, we have $\arg\overline{z_1}=-\arg z_1,\arg\overline{z_2}=-\arg z_2,\dots,\arg\overline{z_n}=-\arg z_n$.  Let $\arg (\overline{z_2}-\overline {z_1})=-\arg(z_2-z_1),\dots,\arg (\overline{z_n}-\overline {z_1})=-\arg(z_n-z_1)$. By remark \ref{lb87}, we can find a chain of unitary intertwining operators $\mathcal Y_{\sigma_2'},\dots,\mathcal Y_{\sigma_n'}$ whose types are the same as those of $\mathcal Y_{\sigma_2},\dots,\mathcal Y_{\sigma_n}$ respectively, such that for any chain of intertwining operators $\mathcal Y_{\alpha_1},\dots,\mathcal Y_{\alpha_n}$ and any  unitary $\mathcal Y_\alpha$, if equation \eqref{eqb1} holds for any $w^{(i_1)}\in W_1,\dots,w^{(i_n)}\in W_{i_n}$, then
	\begin{align}
\mathcal Y_{\sigma'_n\cdots\sigma'_2,C\alpha}(w^{(i_n)},\overline{z_n};\dots,w^{(i_1)},\overline{z_1})=\mathcal Y_{C\alpha_1}(w^{(i_1)},\overline{z_1})\cdots\mathcal Y_{C\alpha_n}(w^{(i_n)},\overline{z_n}).\label{eq312}
	\end{align}
Now assume that $w^{(i_1)},\dots,w^{(i_n)}$ are quasi-primary. By equation (1.27), for any $1\leq m\leq n$, we have
\begin{align}
\mathcal Y_{C\alpha_m}(w^{(i_m)},\overline{z_m})=e^{-i\pi\Delta_{w^{(i_m)}}}z_m^{2\Delta_{w^{(i_m)}}}\mathcal Y_{\alpha_m}(w^{(i_m)},z_m)^{\tr}.
\end{align}
Therefore, by equation \eqref{eqb1}, we see that \eqref{eq312} equals
\begin{align}
&e^{-i\pi\big(\Delta_{w^{(i_1)}}+\cdots+\Delta_{w^{(i_n)}}\big)}z_1^{2\Delta_{w^{(i_1)}}}\cdots z_n^{2\Delta_{w^{(i_n)}}}\big(\mathcal Y_{\alpha_n}(w^{(i_n)},z_n)\cdots\mathcal Y_{\alpha_1}(w^{(i_1)},z_1)\big)^\tr\nonumber\\
=&e^{-i\pi\big(\Delta_{w^{(i_1)}}+\cdots+\Delta_{w^{(i_n)}}\big)}z_1^{2\Delta_{w^{(i_1)}}}\cdots z_n^{2\Delta_{w^{(i_n)}}}\mathcal Y_{\sigma_n\cdots\sigma_2,\alpha}\big(w^{(i_n)},z_n;\dots;w^{(i_1)},z_1\big)^\tr.\label{eq313}
\end{align}

Recall that $\alpha^*=\overline{C\alpha}$. It is obvious that equation
\begin{align}
C_j^{-1}\mathcal Y_{\sigma'_n\cdots\sigma'_2,C\alpha}(w^{(i_n)},\overline{z_n};\dots,w^{(i_1)},\overline{z_1}) C_k=\mathcal Y_{\overline{\sigma'_n}\cdots\overline{\sigma'_2},\alpha^*}(\overline{w^{(i_n)}},z_n;\dots;\overline{w^{(i_1)}},z_1)
\end{align}
holds when $z_1,\dots,z_n$ also satisfy $0<|z_2-z_1|<\cdots<|z_n-z_1|<|z_1|$. By analytic continuation, it holds  for general $(z_1,\dots,z_n)\in\mathscr O_n(I)$. Therefore, if we apply $C_j^{-1}(\cdot)C_k$ to the left hand side of equation \eqref{eq312} and the right hand side of equation \eqref{eq313}, we obtain
\begin{align}
&\mathcal Y_{\overline{\sigma'_n}\cdots\overline{\sigma'_n},\alpha^*}(\overline{w^{(i_n)}},z_n;\dots;\overline{w^{(i_1)}},z_1)\nonumber\\
=&e^{i\pi\big(\Delta_{w^{(i_1)}}+\cdots+\Delta_{w^{(i_n)}}\big)}z_1^{-2\Delta_{w^{(i_1)}}}\cdots z_n^{-2\Delta_{w^{(i_n)}}}\mathcal Y_{\sigma_n\cdots\sigma_2,\alpha}\big(w^{(i_n)},z_n;\dots;w^{(i_1)},z_1\big)^\dagger.
\end{align}
So if we let $\mathcal Y_{\widetilde\sigma_2}=\mathcal Y_{\overline{\sigma_2'}},\dots,\mathcal Y_{\widetilde\sigma_2}=\mathcal Y_{\overline{\sigma_2'}}$, then equation \eqref{eq314} is proved.
\end{proof}

\subsection{Generalized smeared intertwining operators}\label{Condition ABC}
In this section, we assume that $V$ is unitary, energy-bounded, and strongly local. Let $\mathcal F$ be a non-empty set of non-zero irreducible unitary $V$-modules, and let  $\overline{\mathcal F}=\{W_{\overline i}:i\in\mathcal F \}$. Let $\mathcal F^\boxtimes$ be the collection of unitary $V$-modules $W_i$, where $W_i$ is equivalent to a finite  direct sum of  submodules of  tensor products of some $V$-modules in $\mathcal F\cup\overline{\mathcal F}$. So $\mathcal F^\boxtimes$ is additively closed, and any irreducible element in $\mathcal F^\boxtimes$ is equivalent to a submodule of $W_{i_1}\boxtimes\cdots\boxtimes W_{i_n}$, where $i_1,\dots,i_n\in\mathcal F\cup\overline{\mathcal F}$. If $i\in\mathcal F$, we let $E^1(W_i)$ be the vector space of all quasi-primary vectors  $w^{(i)}\in W_i$ satisfying the condition that for any $j,k\in\mathcal F^\boxtimes$ and any $\mathcal Y_\alpha\in\mathcal V{k\choose i~j}$, $\mathcal Y_\alpha(w^{(i)},x)$ satisfies linear energy bounds. $E^1(V)$ is defined in a similar way to be the set of all quasi-primary vectors $v\in V$, such that for any $k\in\mathcal F^\boxtimes$, $Y_k(v,x)$ satisfies linear energy bounds.

In this section, we always assume, unless otherwise stated, that $\mathcal F$ satisfies one of the following two conditions.
\begin{cond}\label{CondA}
{~}\\
(a) Every irreducible submodule of a tensor product of $V$-modules in  $\mathcal F\cup\overline{\mathcal F}$ is unitarizable.\\
(b) $V$ is generated by $E^1(V)$.\\
(c) If $i\in\mathcal F,j,k\in\mathcal F^\boxtimes$, and $\mathcal Y_\alpha\in\mathcal V{k\choose i~j}$, then $\mathcal Y_\alpha$ is energy-bounded.
\end{cond}
\begin{cond}\label{CondB}
{~}\\
(a) Every irreducible submodule of a tensor product of $V$-modules in  $\mathcal F\cup\overline{\mathcal F}$ is unitarizable and energy-bounded.\\
(b) For any $i\in\mathcal F$, $E^1(W_i)$ contains at least one non-zero vector.
\end{cond}

Note that if $V$ is unitary and $\mathcal F$ satisfies condition \ref{CondA}-(b), then by corollary 3.7 and theorem \ref{lb61}, $V$ is energy bounded and strongly local. By corollary 3.7,  Conditions \ref{CondA}-(a),(b) $\Rightarrow$ condition \ref{CondB}-(a), and condition \ref{CondB}-(b) $\Rightarrow$ \ref{CondA}-(c).

\begin{rem}\label{lb122}
If $\mathcal F$ satisfies condition \ref{CondB}, then by theorem \ref{lb120}, any unitary $V$-module $W_i$ in $\mathcal F^\boxtimes$ is strongly integrable. Now we suppose that $\mathcal F$ satisfies condition \ref{CondA}. Then, using the same argument as in the proof of theorem \ref{lb120}, one can show that any $W_i$ in $\mathcal F^\boxtimes$ is \textbf{almost strongly integrable}, which means the following:  Define a real vector subspace $\ER=\{v+\theta v, i(v-\theta v):v\in \E \}$ of $\E$. Then there exists a representation $\pi_i$ of the conformal net $\mathcal M_V$ on the $\mathcal H_i$, such that for any $I\in\mathcal J,v\in\ER$, and $f\in C^\infty_c(I)$ satisfying that
\begin{align}
e^{i\pi\Delta_v/2}e_{1-\Delta_v}f=\overline {e^{i\pi\Delta_v/2}e_{1-\Delta_v}f},\label{eq359}
\end{align}
we have
\begin{align}
\pi_{i,I}\big(\overline{Y(v,f)}\big)=\overline{Y_i(v,f)}.
\end{align}

Note that  by theorem \ref{lb61}, the von Neumann algebra $\mathcal M_V(I)$ is generated by these $\overline{Y(v,f)}$'s. Therefore, such representation $\pi_i$, if exists, must be unique. In this way, we have a  functor $\mathfrak F:\Rep^\uni_{\mathcal F^\boxtimes}(V)\rightarrow \Rep_{\mathcal F^\boxtimes}(\mathcal M_V)$ sending the object $(W_i,Y_i)$ to $(\mathcal H_i,\pi)$.  By proposition 3.6, the conformal vector $\nu$ is inside $\ER$. Therefore, from their proof we see that theorem \ref{lb63} and  corollary \ref{lb103} still hold, with $\mathcal S$ replaced by $\mathcal F^\boxtimes$.

We define $\mathcal M_V(I)_\infty$ to be the set of all $x\in\mathcal M_V(I)$ satisfying relation \eqref{eq301} for any $i\in\mathcal F^\boxtimes$. We can conclude that $\mathcal M_V(I)_\infty$ is a strongly dense self-adjoint subalgebra of $\mathcal M_V(I)$, either by using the same argument as in the proof of proposition \ref{lb104}, or by observing that every $e^{it\overline{Y(v,f)}}$ is inside $\mathcal M_V(I)_\infty$ (by Lemma B.8-(1)), where $t\in\mathbb R,v\in \ER$, and $f\in C^\infty_c(I)$ satisfies equation \eqref{eq359}.
\end{rem}

We now define generalized smeared intertwining operators. First, for any $I\in\mathcal J,n=1,2,\dots$, we choose an arbitrary continuous argument function $\arg_I$ on $I$, and define $\mathfrak O_n(I)$ to be the set of all $(I_1,\dots,I_n)$, where $I_1,\dots,I_n\in\mathcal J(I)$ are mutually disjoint, and for any $2\leq l<m\leq n$,  either $\arg_I(z_lz_1^{-1})\arg_I(z_mz_1^{-1})<0$ for all $z_m\in I_m,z_l\in I_l$, or $|\arg_I(z_lz_1^{-1})|<|\arg_I(z_mz_1^{-1})|$ for all $z_m\in I_m,z_l\in I_l$.

Let $\mathcal Y_{\sigma_n\cdots\sigma_2,\alpha}$ be a  generalized intertwining operator in $\mathcal V{k\choose i_n~\dots~i_1~j}$. We say that $\mathcal Y_{\sigma_n\cdots\sigma_2,\alpha}$ is \textbf{controlled by $\mathcal F$} if $i_1,\dots,i_n\in\mathcal F\cup\overline{\mathcal F}$, and $j,k\in\mathcal F^\boxtimes$. Choose $I\in\mathcal J(S^1\setminus\{-1\})$, $(I_1,\dots,I_n)\in\mathfrak O_n(I)$ and $f_1\in C^\infty_c(I_1),\dots,f_n\in C^\infty_c(I_n)$. For any  $w^{(i_1)}\in W_{i_1},\dots,w^{(i_n)}\in W_{i_n}$, we define a sesquilinear form
\begin{gather*}
\mathcal Y_{\sigma_n\cdots\sigma_2,\alpha}(w^{(i_n)},f_n;\dots;w^{(i_1)},f_1): W_j\times W_k\rightarrow\mathbb C,\\
(w^{(j)}, w^{(k)})\mapsto\langle \mathcal Y_{\sigma_n\cdots\sigma_2,\alpha}(w^{(i_n)},f_n;\dots;w^{(i_1)},f_1)w^{(j)}|w^{(k)} \rangle
\end{gather*}
using the equation
\begin{align}
&\langle \mathcal Y_{\sigma_n\cdots\sigma_2,\alpha}(w^{(i_n)},f_n;\dots;w^{(i_1)},f_1)w^{(j)}|w^{(k)} \rangle\nonumber\\
=&\int_{-\pi}^{\pi}\cdots \int_{-\pi}^{\pi}\langle \mathcal Y_{\sigma_n\cdots\sigma_2,\alpha}(w^{(i_n)},e^{i\theta_n};\dots;w^{(i_1)},e^{i\theta_1})w^{(j)}|w^{(k)} \rangle\cdot f_1(e^{i\theta_1})\cdots f_n(e^{i\theta_n})\di\theta_1\cdots\di\theta_n,
\end{align}
where, for each $l=2,3,\dots,n$, $\arg(e^{i\theta_l}-e^{i\theta_1})$ is close to $\theta_l=\arg e^{i\theta_l}$ as $e^{i\theta_1}\rightarrow 0$.
\begin{pp}\label{lbb2}
Assume that $\mathcal Y_{\sigma_n\cdots\sigma_2,\alpha}$ is controlled by $\mathcal F$. Then the linear operator $ \mathcal Y_{\sigma_n\cdots\sigma_2,\alpha}(w^{(i_n)},f_n;\dots;w^{(i_1)},f_1):W_j\rightarrow \widehat W_k$ maps $W_j$ into $\mathcal H^\infty_k$. If we regard it as an unbounded operator  $\mathcal H_j\rightarrow\mathcal H_k$ with domain $W_j$, then it is preclosed. The closure $\overline{ \mathcal Y_{\sigma_n\cdots\sigma_2,\alpha}(w^{(i_n)},f_n;\dots;w^{(i_1)},f_1)}$ maps $\mathcal H^\infty_j$ into $\mathcal H^\infty_k$, and its adjoint maps $\mathcal H^\infty_k$ into $\mathcal H^\infty_j$. Moreover, there exists  $p\in\mathbb Z_{\geq0}$, such that for any $l\in\mathbb Z_{\geq0}$, we can find  $C_{l+p}>0$, such that the inequality 
\begin{align}
\big\lVert\overline{ \mathcal Y_{\sigma_n\cdots\sigma_2,\alpha}(w^{(i_n)},f_n;\dots;w^{(i_1)},f_1)}\xi^{(j)}\big\lVert_l\leq C_{l+p}\lVert\xi^{(j)}\lVert_{l+p}\label{eq306}
\end{align}
holds for any $\xi^{(j)}\in\mathcal H^\infty_j$.
\end{pp}
\begin{proof}
Choose any $z_1\in I_1,\dots,z_n\in I_n$. Choose arguments $\arg z_1,\dots,\arg z_n\in(-\pi,\pi)$. For each $l=2,3,\dots,n$, we let $\arg(z_l-z_1)$ be close to $\arg z_l$ as $z_1\rightarrow 0$. Suppose that for any $w^{(i_1)}\in W_{i_1},\dots,w^{(i_n)}\in W_{i_n}$ we have the fusion relation
\begin{align}
\mathcal Y_{\sigma_n\cdots\sigma_2,\alpha}(w^{(i_n)},z_n;\dots;w^{(i_1)},z_1)=\mathcal Y_{\alpha_n}(w^{(i_n)},z_n)\cdots\mathcal Y_{\alpha_1}(w^{(i_1)},z_1)\label{eqb3}
\end{align}
for a chain of intertwining operators $\mathcal Y_{\alpha_1},\dots,\mathcal Y_{\alpha_n}$.  Then the  source spaces and the charge spaces of these intertwining operators are unitary $V$-modules in $\mathcal F^\boxtimes$. By condition \ref{CondA}-(c) and proposition 3.3, these intertwining operators are energy-bounded. It follows from proposition 3.12 that
\begin{align}
\mathcal Y_{\sigma_n\cdots\sigma_2,\alpha}(w^{(i_n)},f_n;\dots;w^{(i_1)},f_1)=\mathcal Y_{\alpha_n}(w^{(i_n)},f_n)\cdots\mathcal Y_{\alpha_1}(w^{(i_1)},f_1)\label{eqb4}
\end{align}
when both sides act on $W_j$. Therefore, by equation (3.25), the adjoint of $\mathcal Y_{\sigma_n\cdots\sigma_2,\alpha}(w^{(i_n)},f_n;\dots;w^{(i_1)},f_1)$ has a dense domain containing $\mathcal H^\infty_k$, which proves that $\mathcal Y_{\sigma_n\cdots\sigma_2,\alpha}(w^{(i_n)},f_n;\dots;w^{(i_1)},f_1)$ is preclosed. By proposition 3.9, there exists $p\in\mathbb Z_{\geq0}$, such that for any $l\in\mathbb Z_{\geq0}$, there exists $C_{l+p}>0$,  such that inequality \eqref{eq306} holds for any $\xi^{(j)}\in W_j$. From this we know that $\mathcal H^\infty_j$ is inside the domain of $\overline{\mathcal Y_{\sigma_n\cdots\sigma_2,\alpha}(w^{(i_n)},f_n;\dots;w^{(i_1)},f_1)}$, that this closed operator maps $\mathcal H^\infty_j$ into $\mathcal H^\infty_k$, and that inequality \eqref{eq306} holds for any $\xi^{(j)}\in\mathcal H^\infty_j$. Clearly we have
\begin{align*}
\mathcal Y_{\sigma_n\cdots\sigma_2,\alpha}(w^{(i_n)},f_n;\dots;w^{(i_1)},f_1)^*\supset\mathcal Y_{\alpha_1}(w^{(i_1)},f_1)^\dagger\cdots\mathcal Y_{\alpha_n}(w^{(i_n)},f_n)^\dagger.
\end{align*}
So $\overline{\mathcal Y_{\sigma_n\cdots\sigma_2,\alpha}(w^{(i_n)},f_n;\dots;w^{(i_1)},f_1)}^*$ maps $\mathcal H^\infty_k$ into $\mathcal H^\infty_j$.
\end{proof}

We  regard the linear operator $ \mathcal Y_{\sigma_n\cdots\sigma_2,\alpha}(w^{(i_n)},f_n;\dots;w^{(i_1)},f_1):\mathcal H^\infty_j\rightarrow\mathcal H^\infty_k$ as the restriction of $\overline{ \mathcal Y_{\sigma_n\cdots\sigma_2,\alpha}(w^{(i_n)},f_n;\dots;w^{(i_1)},f_1)}$ to $\mathcal H^\infty_j$, and call it a \textbf{generalized smeared intertwining operator}. Then, if the fusion relation \eqref{eqb3} holds,  relation \eqref{eqb4} holds when both sides act on $\mathcal H^\infty_j$. The \textbf{formal adjoint} $\mathcal Y_{\sigma_n\cdots\sigma_2,\alpha}(w^{(i_n)},f_n;\dots;w^{(i_1)},f_1)^\dagger:\mathcal H^\infty_k\rightarrow\mathcal H^\infty_j$ of $\mathcal Y_{\sigma_n\cdots\sigma_2,\alpha}(w^{(i_n)},f_n;\dots;w^{(i_1)},f_1)$ is defined to be the restriction of the closed operator $\mathcal Y_{\sigma_n\cdots\sigma_2,\alpha}(w^{(i_n)},f_n;\dots;w^{(i_1)},f_1)^*$ to $\mathcal H^\infty_k$.

\begin{pp}[Strong intertwining property]\label{lb106}
Let $\mathcal Y_{\sigma_n\cdots\sigma_2,\alpha}\in \mathcal V{k\choose i_n~\dots~i_1~j}$ be controlled by $\mathcal F$, $w^{(i_1)}\in W_{i_1},\dots,w^{(i_n)}\in W_{i_n}$, $I\in\mathcal J,J\in\mathcal J(S^1\setminus\{-1\})$ be disjoint, and $(J_1,\dots,J_n)\in\mathfrak O_n(J)$. If $\mathcal F$ satisfies condition \ref{CondA},  then for any $x\in\mathcal M_V(I),w^{(i_1)}\in W_{i_1},\dots,w^{(i_n)}\in W_{i_n},g_1\in C^\infty_c(J_1),\dots,g_n\in C^\infty_c(J_n)$, we have
\begin{align}
\pi_k(x)\cdot\overline{\mathcal Y_{\sigma_n\cdots\sigma_2,\alpha}(w^{(i_n)},g_n;\dots;w^{(i_1)},g_1)} \subset\overline{\mathcal Y_{\sigma_n\cdots\sigma_2,\alpha}(w^{(i_n)},g_n;\dots;w^{(i_1)},g_1)}\cdot \pi_j(x).\label{eq317}
\end{align}
Relation \eqref{eq317} still holds if we assume that $\mathcal F$ satisfies condition \ref{CondB}, and that $x\in\mathcal M_V(I),w^{(i_1)}\in E^1(W_{i_1}),\dots,w^{(i_n)}\in E^1(W_{i_n}),g_1\in C^\infty_c(J_1),\dots,g_n\in C^\infty_c(J_n)$.
\end{pp}
\begin{proof}
We assume that the fusion relation \eqref{eqb3} holds when $z_1\in J_1,\dots,z_n\in J_n$ and the arguments are chosen as in the proof of proposition \ref{lbb2}.

First, suppose that $\mathcal F$ satisfies condition \ref{CondA}.  By theorem \ref{lb61}, the von Neumann algebra $\mathcal M_V(I)$ is generated by the bounded operators $e^{it\overline{Y(v,f)}}$, where $t\in\mathbb R,v\in \ER, f\in C^\infty_c(I)$, and $e^{i\pi\Delta_v/2}e_{1-\Delta_v}f=\overline {e^{i\pi\Delta_v/2}e_{1-\Delta_v}f}$. Now for $m=1,2,\dots,n$ we let $W_{j_{m-1}}$ and $W_{j_m}$ be the source space and the target space of $\mathcal Y_{\alpha_m}$ respectively. Then by proposition 3.16 (and proposition B.1), for any $x\in\mathcal M_V(I),w^{(i_m)}\in W_{i_m},g_m\in C^\infty_c(J_m)$, we have
\begin{align}
\pi_{j_m}(x)\overline{\mathcal Y_{\alpha_m}(w^{(i_m)},g_m)}\subset \overline{\mathcal Y_{\alpha_m}(w^{(i_m)},g_m)}\pi_{j_{m-1}}(x). \label{eq364}
\end{align}
Therefore, if $x\in\mathcal M_V(I)_\infty$, then equation
\begin{align}
\pi_{j_m}(x)\mathcal Y_{\alpha_m}(w^{(i_m)},g_m)=\mathcal Y_{\alpha_m}(w^{(i_m)},g_m)\pi_{j_{m-1}}(x).
\end{align}
holds when both sides act on $\mathcal H^\infty_{j_{m-1}}$. Thus, by \eqref{eqb4}, for any $x\in\mathcal M_V(I)_\infty$, equation
\begin{align}
\pi_k(x)\cdot\mathcal Y_{\sigma_n\cdots\sigma_2,\alpha}(w^{(i_n)},g_n;\dots;w^{(i_1)},g_1)=\mathcal Y_{\sigma_n\cdots\sigma_2,\alpha}(w^{(i_n)},g_n;\dots;w^{(i_1)},g_1)\cdot \pi_j(x)
\end{align}
also holds when both sides act on $\mathcal H^\infty_j$. This proves relation \eqref{eq317} for any $x\in\mathcal M_V(I)_\infty$, and hence for any $x\in\mathcal M_V(I)$.

Now we assume that $\mathcal F$ satisfies condition \ref{CondB}. Then from step 2 of the proof of theorem \ref{lb120},  relation \eqref{eq364} holds for any $x\in\mathcal M_V(I)$. This again implies relation \eqref{eq317}. Thus we are done with the proofs for both cases.
\end{proof}

\begin{pp}[Rotation covariance]\label{lb101}
Let $\mathcal Y_{\sigma_n\cdots\sigma_2,\alpha}\in \mathcal V{k\choose i_n~\dots~i_1~j}$ be controlled by $\mathcal F$, $w^{(i_1)}\in W_{i_1},\dots,w^{(i_n)}\in W_{i_n}$ be homogeneous, $J\in S^1\setminus\{-1 \}$, and $(J_1,\dots,J_n)\in\mathfrak O_n(J)$. Choose $\varepsilon>0$ such that $\mathfrak r(t)J\subset S^1\setminus\{-1\}$. Then for any $g_1\in C^\infty_c(J_1),\dots,g_n\in C^\infty_c(J_n)$, and $t\in(-\varepsilon,\varepsilon)$, we have
\begin{align}
&e^{it{\overline {L_0}}}\cdot\overline{\mathcal Y_{\sigma_n\cdots\sigma_2,\alpha}(w^{(i_n)},g_n;\dots;w^{(i_1)},g_1)}\cdot e^{-it\overline {L_0}}\nonumber\\
=&\overline{\mathcal Y_{\sigma_n\cdots\sigma_2,\alpha}\big(w^{(i_n)},e^{i(\Delta_{w^{(i_n)}}-1)t}\mathfrak r(t)g_n;\dots;w^{(i_1)},e^{i(\Delta_{w^{(i_1)}}-1)t}\mathfrak r(t)g_1\big)}.
\end{align}
\end{pp}
\begin{proof}
This follows from relations \eqref{eqb4} and (3.39).
\end{proof}

\begin{thm}[Braiding]\label{lb109}
Let $I,J\in\mathcal J(S^1\setminus\{-1\})$ be disjoint. Choose  $(I_1,\dots,I_m)\in\mathfrak O_m(I),(J_1,\dots,J_n)\in\mathfrak O_n(J)$. Choose $z\in I,\zeta\in J$, and let $-\pi<\arg z,\arg\zeta<\pi$. Let $\mathcal Y_{\sigma_m\cdots\sigma_2,\alpha}\in\mathcal V{k'\choose i_m~\dots~i_1~k_1},\mathcal Y_{\rho_n\cdots\rho_2,\beta}\in\mathcal V{k_1\choose j_n~\dots~j_1~k},\mathcal Y_{\sigma_m\cdots\sigma_2,\alpha'}\in\mathcal V{k_2\choose i_m~\dots~i_1~k},\mathcal Y_{\rho_n\cdots\rho_2,\beta'}\in\mathcal V{k'\choose j_n~\dots~j_1~k_2}$ be generalized intertwining operators of $V$ controlled by $\mathcal F$. Suppose that $W_i$ is the charge spaces of $\mathcal Y_\alpha$ and $\mathcal Y_{\alpha'}$, $W_j$ is the charge space of $\mathcal Y_\beta$ and $\mathcal Y_{\beta'}$, and for  any $w^{(i)}\in W_i,w^{(j)}\in W_j$, we have the braid relation
\begin{align}
	\mathcal Y_\alpha(w^{(i)},z)\mathcal Y_\beta(w^{(j)},\zeta)=\mathcal Y_{\beta'}(w^{(j)},\zeta)\mathcal Y_{\alpha'}(w^{(i)},z).
	\end{align}
Then for any $w^{(i_1)}\in W_{i_1},\dots,w^{(i_m)}\in W_{i_m},w^{(j_1)}\in W_{j_1},\dots,w^{(j_n)}\in W_{j_n},f_1\in C^\infty_c(I_1),\dots f_m\in C^\infty_c(I_m),g_1\in C^\infty_c(J_1),\dots, g_n\in C^\infty_c(J_n)$, we have the braid relation
\begin{align}
&\mathcal Y_{\sigma_m\cdots\sigma_2,\alpha}(w^{(i_m)},f_m;\dots;w^{(i_1)},f_1)	\mathcal Y_{\rho_n\cdots\rho_2,\beta}(w^{(j_n)},g_n;\dots;w^{(j_1)},g_1)\nonumber\\
=&\mathcal Y_{\rho_n\cdots\rho_2,\beta'}(w^{(j_n)},g_n;\dots;w^{(j_1)},g_1)\mathcal Y_{\sigma_m\cdots\sigma_2,\alpha'}(w^{(i_m)},f_m;\dots;w^{(i_1)},f_1).\label{eq311}
\end{align}	
\end{thm}

\begin{proof}
Choose $z_1\in I_1,\dots,z_m\in I_m,\zeta_1\in J_1,\dots,\zeta_n\in J_n$. Let $-\pi<\arg z_1,\dots,\arg z_m,\arg\zeta_1,\dots,\arg \zeta_n<\pi$, and let $\arg(z_2-z_1),\dots,\arg(z_m-z_1),\arg(\zeta_2-\zeta_1),\dots,\arg(\zeta_n-\zeta_1)$ be close to $\arg z_2,\dots,\arg z_m,\arg\zeta_2,\dots,\arg\zeta_n$ as $z_1,\dots,z_1,\zeta_1,\dots,\zeta_1\rightarrow 0$ respectively.
Suppose that for any $w^{(i_1)}\in W_{i_1},\dots,w^{(i_m)}\in W_{i_m},w^{(j_1)}\in W_{j_1},\dots,w^{(j_n)}\in W_{j_n}$, we have the fusion relations
\begin{gather}
\mathcal Y_{\sigma_m\cdots\sigma_2,\alpha}(w^{(i_m)},z_m;\dots;w^{(i_1)},z_1)=\mathcal Y_{\alpha_m}(w^{(i_m)},z_m)\cdots\mathcal Y_{\alpha_1}(w^{(i_1)},z_1),\\
\mathcal Y_{\rho_n\cdots\rho_2,\beta}(w^{(j_n)},\zeta_n;\dots;w^{(j_1)},\zeta_1)=\mathcal Y_{\beta_n}(w^{(j_n)},\zeta_n)\cdots\mathcal Y_{\beta_1}(w^{(j_1)},\zeta_1).
\end{gather}
Then the source spaces and the target spaces of $\mathcal Y_{\alpha_1},\dots\mathcal Y_{\alpha_m},\mathcal Y_{\beta_1},\dots,\mathcal Y_{\beta_n}$ are unitary $V$-modules inside $\mathcal F^\boxtimes$. So these intertwining operators are energy-bounded. By relation \eqref{eqb4}, we have
\begin{gather}
\mathcal Y_{\sigma_m\cdots\sigma_2,\alpha}(w^{(i_m)},f_m;\dots;w^{(i_1)},f_1)=\mathcal Y_{\alpha_m}(w^{(i_m)},f_m)\cdots\mathcal Y_{\alpha_1}(w^{(i_1)},f_1),\\
\mathcal Y_{\rho_n\cdots\rho_2,\beta}(w^{(j_n)},g_n;\dots;w^{(j_1)},g_1)=\mathcal Y_{\beta_n}(w^{(j_n)},g_n)\cdots\mathcal Y_{\beta_1}(w^{(j_1)},g_1).
\end{gather}
Therefore, by proposition 3.12,
\begin{align}
&\mathcal Y_{\sigma_m\cdots\sigma_2,\alpha}(w^{(i_m)},f_m;\dots;w^{(i_1)},f_1)	\mathcal Y_{\rho_n\cdots\rho_2,\beta}(w^{(j_n)},g_n;\dots;w^{(j_1)},g_1)\nonumber\\
=&\mathcal Y_{\alpha_m}(w^{(i_m)},f_m)\cdots\mathcal Y_{\alpha_1}(w^{(i_1)},f_1)\mathcal Y_{\beta_n}(w^{(j_n)},g_n)\cdots\mathcal Y_{\beta_1}(w^{(j_1)},g_1)\nonumber\\
=&\int_{-\pi}^{\pi}\cdots\int_{-\pi}^{\pi}\cdot\int_{-\pi}^{\pi}\cdots\int_{-\pi}^{\pi}\mathcal Y_{\alpha_m}(w^{(i_m)},e^{i\theta_m})\cdots\mathcal Y_{\alpha_1}(w^{(i_1)},e^{i\theta_1})\nonumber\\
&\cdot\mathcal Y_{\beta_n}(w^{(j_n)},e^{i\vartheta_n})\cdots\mathcal Y_{\beta_1}(w^{(j_1)},e^{i\vartheta_1})f_1(e^{i\theta_1})\cdots f_m(e^{i\theta_m})\nonumber\\
&\cdot g_1(e^{i\vartheta_1})\cdots g_n(e^{i\vartheta_n})\di\theta_1\cdots\di\theta_m\di\vartheta_1\cdots\di\vartheta_n\nonumber\\
=&\int_{-\pi}^{\pi}\cdots\int_{-\pi}^{\pi}\cdot\int_{-\pi}^{\pi}\cdots\int_{-\pi}^{\pi}\mathcal Y_{\sigma_m\cdots\sigma_2,\alpha}(w^{(i_m)},e^{i\theta_m};\cdots;w^{(i_1)},e^{i\theta_1})\nonumber\\
&\cdot\mathcal Y_{\rho_n\cdots\rho_2,\beta}(w^{(j_n)},e^{i\vartheta_n};\cdots;w^{(j_1)},e^{i\vartheta_1})f_1(e^{i\theta_1})\cdots f_m(e^{i\theta_m})\nonumber\\
&\cdot g_1(e^{i\vartheta_1})\cdots g_n(e^{i\vartheta_n})\di\theta_1\cdots\di\theta_m\di\vartheta_1\cdots\di\vartheta_n.\label{eq309}
\end{align}
The same argument shows that
\begin{align}
&\mathcal Y_{\rho_n\cdots\rho_2,\beta'}(w^{(j_n)},g_n;\dots;w^{(j_1)},g_1)\mathcal Y_{\sigma_m\cdots\sigma_2,\alpha'}(w^{(i_m)},f_m;\dots;w^{(i_1)},f_1)\nonumber\\
=&\int_{-\pi}^{\pi}\cdots\int_{-\pi}^{\pi}\cdot\int_{-\pi}^{\pi}\cdots\int_{-\pi}^{\pi}\mathcal Y_{\rho_n\cdots\rho_2,\beta'}(w^{(j_n)},e^{i\vartheta_n};\cdots;w^{(j_1)},e^{i\vartheta_1})\nonumber\\
&\cdot\mathcal Y_{\sigma_m\cdots\sigma_2,\alpha'}(w^{(i_m)},e^{i\theta_m};\cdots;w^{(i_1)},e^{i\theta_1})f_1(e^{i\theta_1})\cdots f_m(e^{i\theta_m})\nonumber\\
&\cdot g_1(e^{i\vartheta_1})\cdots g_n(e^{i\vartheta_n})\di\theta_1\cdots\di\theta_m\di\vartheta_1\cdots\di\vartheta_n.\label{eq310}
\end{align}
By theorem \ref{lb16}, the right hand sides of equations \eqref{eq309} and \eqref{eq310} are equal, which proves equation \eqref{eq311}.
\end{proof}

\begin{thm}[Adjoint relation]\label{lb108}
Choose $I\in\mathcal J(S^1\setminus\{-1\})$ and $(I_1,\dots,I_n)\in\mathfrak O_n(I)$. Let $W_{i_1},W_{i_2},\dots,W_{i_n}$ be unitary $V$-modules in $\mathcal F\cup\overline{\mathcal F}$, and let $\mathcal Y_{\sigma_2},\dots,\mathcal Y_{\sigma_n}$ be a chain of unitary intertwining operators of $V$ with charge spaces $W_{i_2},\dots,W_{i_n}$ respectively, such that  the source space of $\mathcal Y_{\sigma_2}$ is $W_{i_1}$. We let $W_i\in\mathcal F^\boxtimes$ be  the target space of $\mathcal Y_{\sigma_n}$. Then for each $2\leq m\leq n$, there exists a unitary intertwining operator $\mathcal Y_{\widetilde\sigma_m}$ whose type is the same as that of $\mathcal Y_{\overline{\sigma_m}}$, such that for   any unitary $V$-modules $W_j,W_k$ in $\mathcal F^\boxtimes,\mathcal Y_\alpha\in\mathcal V{k\choose i~j},w^{(i_1)}\in W_{i_1},\dots,w^{(i_n)}\in W_{i_n}$ being quasi-primary, and  $f_1\in C^\infty_c(I_1),\dots,f_n\in C^\infty_c(I_n)$, we have
	\begin{align}
	&\mathcal Y_{\sigma_n\cdots\sigma_2,\alpha}\big(w^{(i_n)},f_n;\dots;w^{(i_1)},f_1\big)^\dagger\nonumber\\
	=&e^{-i\pi\big(\Delta_{w^{(i_1)}}+\cdots+\Delta_{w^{(i_n)}}\big)}\cdot\mathcal Y_{\widetilde\sigma_n\cdots\widetilde \sigma_2,\alpha^*}\big(\overline{w^{(i_n)}},\overline{e_{(2-2\Delta_{w^{(i_n)}})}f_n};\dots;\overline{w^{(i_1)}},\overline{e_{(2-2\Delta_{w^{(i_1)}})}f_1}\big).\label{eq328}
	\end{align}
\end{thm}
\begin{proof}
This is obtained by multiplying both sides of equation \eqref{eq314} by the expression\begin{align*}
\overline{f_1(e^{i\theta_1})}\cdots\overline{f_n(e^{i\theta_n})}e^{-2i(\theta_1+\cdots+\theta_n)}\di\theta_1\cdots \di\theta_n,
\end{align*}
and then taking the integral. We leave the details to the reader.
\end{proof}

\section{Defining an inner product $\Lambda$ on $W_i\boxtimes W_j$}

In this chapter, we define (in section 2) a sesquilinear form $\Lambda$ on $W_{ij}=W_i\boxtimes W_j$ using transport matrices, and prove (in section 3) that these forms are inner products. As discussed in the  introduction of part I, our strategy for proving the positivity of $\Lambda$ is to identify the form $\Lambda$ on a dense subspace of $\mathcal H_{ij}$ with the inner product on a subspace of the  Connes fusion product $\mathcal H_i\boxtimes\mathcal H_j$ of the conformal net modules $\mathcal H_i$ and $\mathcal H_j$. In section 1, we prove a density property for constructing such a dense subspace. 

Note that the  Connes fusion product (Connes relative tensor product) is a motivation rather than a logistic background of our theory. So we don't assume the reader has any previous knowledge on this topic, nor shall we give a formal definition on Connes fusion in this paper. Those who are interested in this topic can read \cite{Wassermann} section 30 for a brief introduction, or read \cite{Con80} or \cite{Takesaki II} section IX.3 for more details.

\subsection{Density of the range of fusion product}\label{lb19}

Recall from section 3.2 that $W_{ij}=W_i\boxtimes W_j=\bigoplus_{k\in\mathcal E}\mathcal V{k \choose i~j}^*\otimes W_k$ is the tensor product module of $W_i,W_j$. We now define a type $ij\choose{i~j}$ intertwining operator $\mathcal Y_{i\boxtimes j}:W_i\otimes W_j\rightarrow W_{ij}\{x\}$ in the following way: If $\mathcal Y_\alpha\in\mathcal V {k\choose i~j}, w^{(i)}\in W_i,w^{(j)}\in W_j$ and $w^{(\overline k)}\in W_{\overline k}$, then
\begin{align}
\langle\mathcal Y_\alpha\otimes w^{(\overline k)}, \mathcal Y_{i\boxtimes j}(w^{(i)},x)w^{(j)}\rangle=\langle w^{(\overline k)},\mathcal Y_\alpha(w^{(i)},x)w^{(j)}\rangle.
\end{align}
For any $k\in\mathcal E$, we choose a basis $\{\mathcal Y_\alpha:\alpha\in\Theta^k_{ij} \}$ of $\mathcal V{k\choose i~j}$, and let $\{\widecheck{ \mathcal Y}^\alpha:\alpha\in\Theta^k_{ij} \}\subset\mathcal V{k\choose i~j}^*$ be the  dual basis of $\Theta^k_{ij}$. (i.e., if $\alpha,\beta\in\Theta^k_{ij}$, then $\langle \mathcal Y_\alpha,\widecheck{\mathcal Y}^\beta \rangle=\delta_{\alpha,\beta}$.) Then for any $w^{(i)}\in W_i$ and $w^{(j)}\in W_j$ we have
\begin{align}
\mathcal Y_{i\boxtimes j}(w^{(i)},x)w^{(j)}=\sum_{k\in\mathcal E}\sum_{\alpha\in\Theta^k_{ij}}\widecheck{\mathcal Y}^{\alpha}\otimes\mathcal Y_{\alpha}(w^{(i)},x)w^{(j)}=\sum_{\alpha\in\Theta^*_{ij}}\widecheck{\mathcal Y}^{\alpha}\otimes\mathcal Y_{\alpha}(w^{(i)},x)w^{(j)}.
\end{align}
(See the beginning of section 2 for notations.)

The following density property generalizes proposition A.3.
\begin{pp}\label{lb100}
Let $\mathcal Y_{\sigma_2},\dots,\mathcal Y_{\sigma_n}$ be a chain of non-zero irreducible intertwining operators of $V$ with charge spaces $W_{i_2},\dots,W_{i_n}$ respectively. Let $W_{i_1}$ be the source space of $\mathcal Y_{\sigma_2}$, and  let $W_i$ be the target space of $\mathcal Y_{\sigma_n}$. Choose a $V$-module $W_j$, non-zero vectors $w_0^{(i_1)}\in W_{i_1},\dots,w_0^{(i_n)}\in W_{i_n}$, $I\in\mathcal J,(z_1,\dots,z_n)\in\mathscr O_n(I)$, and choose arguments $\arg z_1,\arg (z_2-z_1),\dots,\arg(z_n-z_1)$. Fix $w^{(\overline {ij})}\in W_{\overline {ij}}$. Suppose that for any $w^{(j)}\in W_j$, 
\begin{align}
\langle w^{(\overline{ij})},\mathcal Y_{\sigma_n\cdots\sigma_2,i\boxtimes j}(w^{(i_n)}_0,z_n;\dots;w^{(i_1)}_0,z_1)w^{(j)} \rangle=0,\label{eq315}
\end{align}
then $w^{(\overline{ij})}=0$.
\end{pp}
\begin{proof}
Suppose that equation  \eqref{eq315} holds. From the proof of corollary 2.15, we see that
\begin{align}
\langle w^{(\overline{ij})},\mathcal Y_{\sigma_n\cdots\sigma_2,i\boxtimes j}(w^{(i_n)},z_n;\dots;w^{(i_1)},z_1)w^{(j)} \rangle=0\label{eq316}
\end{align}
for all $w^{(i_1)}\in W_{i_1},\dots,w^{(i_n)}\in W_{i_n},w^{(i)}\in W_i$. By theorem 2.4 and the discussion below, equation \eqref{eq316} holds for all $(z_1,\dots,z_n)\in\mathscr O_n(I)$ (the arguments $\arg z_1,\arg (z_2-z_1),\dots,\arg(z_n-z_1)$ are changed continuously). In particular, for any $(z_1,\dots,z_n)\in\mathscr O_n(I)$ satisfying $0<|z_2-z_1|<|z_3-z_1|<\cdots<|z_n-z_1|<|z_1|$, equation \eqref{eq316} reads
\begin{align}
\big\langle w^{(\overline{ij})},\mathcal Y_{i\boxtimes j}\big(\mathcal Y_{\sigma_n}(w^{(i_n)},z_n-z_1)\cdots\mathcal Y_{\sigma_2}(w^{(i_2)},z_2-z_1)w^{(i_1)},z_1 \big) w^{(j)} \big\rangle=0.
\end{align}
If we let $z_2$ be close to $z_1$, then by proposition A.1, for any $s_2\in\mathbb R$, we have
\begin{align}
\big\langle w^{(\overline{ij})},\mathcal Y_{i\boxtimes j}\big(\mathcal Y_{\sigma_n}(w^{(i_n)},z_n-z_1)\cdots\mathcal Y_{\sigma_3}(w^{(i_3)},z_3-z_1)\mathcal Y_{\sigma_2}(w^{(i_2)},s_2)w^{(i_1)},z_1 \big) w^{(j)} \big\rangle=0,
\end{align}
where $\mathcal Y_{\sigma_2}(w^{(i_2)},s_2)$ is a mode of the intertwining operator $\mathcal Y_{\sigma_2}(w^{(i_2)},x)$. Let $W_{j_2}$ be the target space of $\mathcal Y_{\sigma_2}$ (which is also the source space of $\mathcal Y_{\sigma_3}$). Then by corollary A.4, vectors of the form $\mathcal Y_{\sigma_2}(w^{(i_2)},s_2)w^{(i_1)}$ span the vector space $W_{j_2}$. Therefore, for any $w^{(j_2)}\in W_{j_2}$, we have
\begin{align}
\big\langle w^{(\overline{ij})},\mathcal Y_{i\boxtimes j}\big(\mathcal Y_{\sigma_n}(w^{(i_n)},z_n-z_1)\cdots\mathcal Y_{\sigma_3}(w^{(i_3)},z_3-z_1)w^{(j_2)},z_1 \big) w^{(j)} \big\rangle=0.
\end{align}
If we apply the same argument several times, then for any $w^{(i)}\in W_i,w^{(j)}\in W_j$,
\begin{align}
\langle w^{(\overline{ij})},\mathcal Y_{i\boxtimes j}(w^{(i)},z_1 ) w^{(j)} \rangle=0.
\end{align}
So by proposition A.3, $w^{(\overline{ij})}$ must be zero.
\end{proof}

A smeared version of the above proposition can be stated as follows.

\begin{pp}\label{lb105}
Let $V$ be unitary, energy-bounded, and strongly local. Let $\mathcal F$ be a non-empty set of non-zero irreducible unitary $V$-modules satisfying condition \ref{CondA} or \ref{CondB} in section \ref{Condition ABC}.  Let $W_i,W_j$ be unitary $V$-modules in $\mathcal F^\boxtimes$, and assume that $W_i$ is irreducible. Fix an arbitrary unitary structure on $W_{ij}$.

Let $W_{i_1},\dots,W_{i_n}$ be irreducible unitary $V$-modules in $\mathcal F\cup\overline{\mathcal F}$. Let $\mathcal Y_{\sigma_2},\dots,\mathcal Y_{\sigma_n}$ be a chain of non-zero irreducible unitary intertwining operators of $V$ with charge spaces $W_{i_2},\dots,W_{i_n}$ respectively, such that $W_{i_1}$ is the source space of $\mathcal Y_{\sigma_2}$, and   $W_i$ is the target space of $\mathcal Y_{\sigma_n}$. Choose $I\in\mathcal J(S^1\setminus\{-1\}),(I_1,\dots,I_n)\in\mathfrak O_n(I)$. Fix non-zero homogeneous vectors $w^{(i_1)}_0\in W_{i_1},\dots,w^{(i_n)}_0\in W_{i_n}$. Then for any $l\in\mathbb Z_{\geq0}$, vectors of the form
\begin{align}
\pi_{ij}(x)\mathcal Y_{\sigma_n\cdots\sigma_2,i\boxtimes j}(w^{(i_n)}_0,f_n;\dots;w^{(i_1)}_0,f_1)w^{(j)},\label{eq318}
\end{align}
span a core for $\overline{L_0}^l$, where $x\in\mathcal M_V(I)_\infty,f_1\in C^\infty_c(I_1),\dots,f_n\in C^\infty_c(I_n),w^{(j)}\in W_j$.
\end{pp}
\begin{proof}
Let $\mathcal W_1$ be the subspace of $\mathcal H^\infty_{ij}$ spanned by vectors of the form \eqref{eq318}. We first show that $\mathcal W_1$ is a dense subspace of $\mathcal H_{ij}$.

The first step is to show that $\mathcal W_1^\perp$ is invariant under the action of the conformal net $\mathcal M_V$.  Choose an open interval $J\subset\joinrel\subset I$, and choose $\delta>0$ such that $\mathfrak r(t)J\subset I$ for any $t\in(-\delta,\delta)$. Fix $\xi^{(ij)}\in\mathcal W_1^\perp$.  Then for any $w^{(j)}\in W_j,m\in\mathbb Z_{>0},x_1,\dots,x_m\in\mathcal M_V(J)_\infty,f_1\in C^\infty_c(I_1),\dots,f_n\in C^\infty_c(I_n)$, we have
\begin{align}
\langle x_m\cdots x_2x_1\mathcal Y_{\sigma_n\cdots\sigma_2,i\boxtimes j}(w^{(i_n)}_0,f_n;\dots;w^{(i_1)}_0,f_1)w^{(j)}|\xi^{(ij)} \rangle=0.\label{eq321}
\end{align}
Choose $\varepsilon>0$ such that   the support of
$$f_a^t=\exp(it(\Delta_{w^{(i_a)}_0}-1))\mathfrak r(t)f_a$$ 
is inside $I_a$ for any $t\in(-\varepsilon,\varepsilon)$ and any  $a=1,2,\dots,n$. Then, by proposition \ref{lb101}, for any $t\in\mathbb R$ we have
\begin{align}
&\langle x_m\cdots x_1\cdot e^{it\overline{L_0}}\mathcal Y_{\sigma_n\cdots\sigma_2,i\boxtimes j}(w^{(i_n)}_0,f_n;\dots;w^{(i_1)}_0,f_1)w^{(j)}|\xi^{(ij)}\rangle\nonumber\\
=&\langle x_m\cdots x_1\mathcal Y_{\sigma_n\cdots\sigma_2,i\boxtimes j}(w^{(i_n)}_0,f^t_n;\dots;w^{(i_1)}_0,f^t_1)e^{it\overline{L_0}}w^{(j)}|\xi^{(ij)}\rangle,\label{eq352}
\end{align}
which must be zero when $t\in(-\varepsilon,\varepsilon)$. Therefore, as in step 1 of the proof of theorem \ref{lb120},  the Schwarz reflection principle implies that \eqref{eq352} equals zero for any $t\in\mathbb R$. (Note that when we define generalized smeared intertwining operators, the arguments are restricted to $(-\pi,\pi)$. Here we allow the arguments  to exceed $(-\pi,\pi)$ and change continuously according to the action of $\mathfrak r(t)$.) Hence we conclude that equation \eqref{eq321} holds for any $t\in\mathbb R,w^{(j)}\in W_j,x_1,\dots,x_m\in\mathcal M_V(J)_\infty,f_1\in C^\infty_c(\mathfrak r(t)I_1),\dots,f_n\in C^\infty_c(\mathfrak r(t)I_n)$.

We use similar argument once more.  Choose any  $w^{(j)}\in W_j,t_0,t\in\mathbb R,x_1,\dots,x_m\in\mathcal M_V(J)_\infty,f_1\in C^\infty_c(\mathfrak r(t_0)I_1),\dots,f_n\in C^\infty_c(\mathfrak r(t_0)I_n)$. Then by proposition \ref{lb101} and equation \eqref{eq319}, we have
\begin{align}
&\langle x_m\cdots x_2\cdot e^{it\overline{L_0}}\pi_{ij}(x_1)\mathcal Y_{\sigma_n\cdots\sigma_2,i\boxtimes j}(w^{(i_n)}_0,f_n;\dots;w^{(i_1)}_0,f_1)w^{(j)}|\xi^{(ij)}\rangle\nonumber\\
=&\langle x_m\cdots x_2\cdot e^{it\overline{L_0}}\pi_{ij}(x_1)e^{-it\overline{L_0}}\mathcal Y_{\sigma_n\cdots\sigma_2,i\boxtimes j}(w^{(i_n)}_0,f^t_n;\dots;w^{(i_1)}_0,f^t_1)e^{it\overline{L_0}}w^{(j)}|\xi^{(ij)}\rangle\nonumber\\
=&\langle x_m\cdots x_2\cdot \pi_{ij}(e^{it\overline{L_0}}x_1e^{-it\overline{L_0}})\mathcal Y_{\sigma_n\cdots\sigma_2,i\boxtimes j}(w^{(i_n)}_0,f^t_n;\dots;w^{(i_1)}_0,f^t_1)e^{it\overline{L_0}}w^{(j)}|\xi^{(ij)}\rangle.\label{eq322}
\end{align}
If $t\in(-\delta,\delta)$, then $e^{it\overline{L_0}}x_1e^{-it\overline{L_0}}\in\mathcal M_V(\mathfrak r(t)J)_\infty\subset\mathcal M_V(I)_\infty$, and hence \eqref{eq322} must be zero. So the value of \eqref{eq322} equals zero when $t\in(-\delta,\delta)$. By Schwarz reflection principle, \eqref{eq322} equals zero for any $t\in\mathbb R$. Since the choice of $t_0$ is arbitrary, we  conclude that equation \eqref{eq321} holds for any $t_0,t_1\in\mathbb R,x_1\in\mathcal M_V(\mathfrak r(t_1)J)_\infty,x_2\in\mathcal M_V(J)_\infty,\dots,x_m\in\mathcal M_V(J)_\infty,w^{(j)}\in W_j,f_1\in C^\infty_c(\mathfrak r(t_0)I_1),\dots,f_n\in C^\infty_c(\mathfrak r(t_0)I_n)$. The same argument shows that  equation \ref{eq321} holds for any $t_0,t_1,t_2,\dots,t_m\in\mathbb R,w^{(j)}\in W_j,x_1\in\mathcal M_V(\mathfrak r(t_1)J)_\infty,x_2\in\mathcal M_V(\mathfrak r(t_2)J)_\infty,\dots,x_m\in\mathcal M_V(\mathfrak r(t_m)J)_\infty,f_1\in C^\infty_c(\mathfrak r(t_0)I_1),\dots,f_n\in C^\infty_c(\mathfrak r(t_0)I_n)$. Hence, by proposition \ref{lb104} and the additivity of $\mathcal M_V$,  the equation
\begin{align}
\langle x_m\cdots x_1\mathcal Y_{\sigma_n\cdots\sigma_2,i\boxtimes j}(w^{(i_n)}_0,f_n;\dots;w^{(i_1)}_0,f_1)w^{(j)}|\xi^{(ij)}\rangle=0
\end{align}
holds for any $m\in\mathbb Z_{\geq0},J_1,\dots,J_m\in\mathcal J,x_1\in\mathcal M_V(J_1),\dots,x_m\in\mathcal M_V(J_m),f_1\in C^\infty_c(I_1),\dots,f_n\in C^\infty_c(I_n),w^{(j)}\in W_j,\xi^{(ij)}\in\mathcal W_1^\perp$. This proves that $\mathcal W_1^\perp$ is $\mathcal M_V$-invariant. 

Now suppose that $\mathcal W_1^\perp$ is non-trivial. By corollary \ref{lb103} and remark \ref{lb122}, $\mathcal W_1^\perp$ is the closure of a non-trivial $V$-submodule of $W_{ij}$. Thus there exists a non-zero vector $w^{(ij)}\in W_{ij}\cap\mathcal W_1^\perp$. For any $f_1\in C^\infty_c(I_1),\dots,f_n\in C^\infty_c(I_n),w^{(j)}\in W_j$, we have
\begin{align}
\langle \mathcal Y_{\sigma_n\cdots\sigma_2,i\boxtimes j}(w^{(i_n)}_0,f_n;\dots;w^{(i_1)}_0,f_1)w^{(j)}|w^{(ij)} \rangle=0.
\end{align}
Fix $z_1\in I_1,\dots,z_n\in I_n$. For each $1\leq m\leq n$ we let $f_m$ converge to the $\delta$-function at $z_m$. Then we have
\begin{align}
\langle \mathcal Y_{\sigma_n\cdots\sigma_2,i\boxtimes j}(w^{(i_n)}_0,z_n;\dots;w^{(i_1)}_0,z_1)w^{(j)}|w^{(ij)} \rangle=0
\end{align}
for any $w^{(j)}\in W_j$. By proposition \ref{lb100}, $w^{(ij)}$ equals zero, which is impossible. So $\mathcal W_1$ must be dense.

Now we show that $\mathcal W_1$ is a core for $\overline{L_0}^l$.  Choose an open interval $K\subset\joinrel\subset I$, and  $(K_1,\dots,K_n)\in\mathfrak O_n(K)$, such that  $K_1\subset\joinrel\subset I_1,\dots,K_n\subset\joinrel\subset I_n$. Let $\mathcal W_2$ be the subspace of $\mathcal H_{ij}^\infty$ spanned by vectors of the form
\begin{align*}
\pi_{ij}(x)\mathcal Y_{\sigma_n\cdots\sigma_2,i\boxtimes j}(w^{(i_n)}_0,f_n;\dots;w^{(i_1)}_0,f_1)w^{(j)},
\end{align*}
where $x\in\mathcal M_V(K)_\infty,f_1\in C^\infty_c(K_1),\dots,f_n\in C^\infty_c(K_n),w^{(j)}\in W_j$. Then clearly $\mathcal W_2$ is also dense in $\mathcal H_{ij}$. Choose $\epsilon>0$ such that for any $t\in(-\epsilon,\epsilon)$, $\mathfrak r(t)K\subset I,\mathfrak r(t)K_1\subset I_1,\dots,\mathfrak r(t)K_n\subset I_n$. Then by proposition \ref{lb101}, $e^{it\overline{L_0}}\mathcal W_2\subset\mathcal W_1$. Hence, by the next lemma, $\mathcal W_1$ is a core for $\overline{L_0}^l$.
\end{proof}

\begin{lm}[cf. \cite{CKLW} lemma 7.2.]\label{lb15}
	Let $A$ be a self-adjoint operator on a Hilbert space $\mathcal H$, and let $U(t)=e^{itA}$, $t\in\mathbb R$ be the corresponding strongly-continuous one-parameter group of unitary operators on $\mathcal H$. For any $k\in\mathbb Z_{\geq0}$, let $\mathcal H^k$ denote the domain of $A^k$, and let $\mathcal H^\infty=\bigcap_{k\in\mathbb Z_{\geq0}}\mathcal H^k$. Assume that there exists a real number $\epsilon>0$ and two dense linear subspaces $\mathscr D_\epsilon$ and $\mathscr D$ of $\mathcal H^\infty$ such that $U(t)\mathscr D_\epsilon\subset\mathscr D$ for any $t\in(-\epsilon,\epsilon)$. Then, for every positive integer $k$, $\mathscr D$ is a core for $A^k$.
\end{lm}

\subsection{The sesquilinear form $\Lambda$ on $W_i\boxtimes W_j$}

Beginning with this section, we assume that $V$ is unitary, energy bounded, and strongly local, and that there exists a non-empty set $\mathcal F$ of non-zero irreducible unitary $V$-modules  satisfying condition \ref{CondA} or \ref{CondB} in section \ref{Condition ABC}. 

Choose unitary $V$-modules $W_i,W_j$ in $\mathcal F^\boxtimes$.  We now define, for any $k\in\mathcal E$, a sesquilinear form $\Lambda=\Lambda(\cdot|\cdot)$ on $\mathcal V{k\choose i~j}^*$ (antilinear on the second variable).  Choose a basis $\{\mathcal Y_{\alpha}:\alpha\in\Theta^k_{ij} \}$ of $\mathcal V{k\choose i~j}$. Choose  $z_1,z_2\in\mathbb C^\times$ satisfying $0<|z_2-z_1|<|z_1|<|z_2|$. Choose $\arg z_2$, let $\arg z_1$ be close to $\arg z_2$ as $z_2-z_1\rightarrow 0$, and let $\arg (z_2-z_1)$ be close to $\arg z_2$ as $z_1\rightarrow 0$.
By fusion of intertwining operators, there exists a complex $N^k_{ij}\times N^k_{ij}$ matrix $\Lambda=\{\Lambda^{\alpha\beta}\}_{\alpha,\beta\in\Theta^k_{ij}}$, such that  for any $w^{(i)}_1,w^{(i)}_2\in W_i$ we have the following \textbf{transport formula} (version 1):
\begin{align}
&Y_j\big(\mathcal Y^0_{\overline ii}(\overline {w^{(i)}_2},z_2-z_1)w^{(i)}_1,z_1\big)\nonumber\\
=&\sum_{k\in\mathcal E}\sum_{\alpha,\beta\in\Theta^k_{ij}}\Lambda^{\alpha\beta}\mathcal Y_{\beta^*}(\overline{w^{(i)}_2},z_2)\mathcal Y_{\alpha}(w^{(i)}_1,z_1)\nonumber\\
=&\sum_{\alpha,\beta\in\Theta^*_{ij}}\Lambda^{\alpha\beta}\mathcal Y_{\beta^*}(\overline{w^{(i)}_2},z_2)\mathcal Y_{\alpha}(w^{(i)}_1,z_1).\label{eq62}
\end{align}
The  matrix $\Lambda$ is called a \textbf{transport matrix} of $V$.  Let $\{\widecheck{\mathcal Y}^{\alpha}:\alpha\in\Theta^k_{ij} \}$ be the dual basis of $\Theta^k_{ij}$. We then define a sesquilinear form $\Lambda(\cdot|\cdot )$ on $\mathcal V{k\choose i~j}^*$ by setting
\begin{align}
\Lambda( \widecheck{\mathcal Y}^{\alpha}|\widecheck{\mathcal Y}^{\beta}  )=\Lambda^{\alpha\beta}.
\end{align}
It is easy to see that this definition does not depend on the basis chosen. These sesquilinear forms induce one on the vector space $W_i\boxtimes W_j=\bigoplus_{k\in\mathcal E}\mathcal V{k\choose i~j}^*\otimes W_k$: if $k_1,k_2\in\mathcal E\cap\mathcal F^\boxtimes,\widecheck{\mathcal Y}_1\in\mathcal V{k_1\choose i~j}^*,\widecheck{\mathcal Y}_2\in\mathcal V{k_2\choose i~j}^*,w^{(k_1)}\in W_{k_1},w^{(k_2)}\in W_{k_2}$, then
\begin{align}
\Lambda\big(\widecheck{\mathcal Y}_1\otimes w^{(k_1)}\big|\widecheck{\mathcal Y}_2\otimes w^{(k_2)}\big)= \left\{ \begin{array}{cl}
\Lambda(\widecheck{\mathcal Y}_1|\widecheck{\mathcal Y}_2)\langle w^{(k_1)}|w^{(k_2)} \rangle & \textrm{if $k_1=k_2$,}\\
0 & \textrm{if $k_1\neq k_2$.}
\end{array} \right.\label{eqb5}
\end{align}
In the next section, we will prove that $\Lambda$ is an inner product.

\begin{rem}\label{lb30}
Our definition of transport formulas is motivated by the landmark paper  \cite{Wassermann} of A.Wassermann.  In that paper, Wassermann used smeared intertwining operators to define transport formulas  when the fusion rules are at most $1$ (see section 31). In those cases  transport matrices become transport coefficients. Proving the strict positivity of these coefficients  is one of the key steps to compute the \emph{Connes} fusion rules for type $A_n$ unitary WZW models in \cite{Wassermann}. However, to prove similar results for other examples, one has to combine Wassermann's methods with Huang-Lepowsky's tensor product theory on VOA modules, as we now briefly explain.

The non-zeroness of the transport coefficients in \cite{Wassermann} was proved  by computing explicitly the monodromy coefficients of the solutions of differential equations (2.5) (in the case of WZW models, the Knizhnik-Zamolodchikov (KZ) equations). In the case of \cite{Wassermann}, these equations reduce to certain generalized hypergeometric equations, the manipulation of which is still possible. But for other examples, say type $G_2$ WZW models, these differential equations are more complicated and  the precise values of transport coefficients are therefore much harder to calculate. 

Fortunately, Y.Huang's remarkable works on the rigidity and the modularity of VOA tensor categories (cf. \cite{H MI,H Verlinde,H Rigidity}) provide a solution to this issue. Indeed, the non-zeroness of transport coefficients (or the non-degeneracy of transport formulas) is closely related to the (weak) \emph{rigidity} of the braided tensor category $\Rep(V)$ (see step 3 of the proof of theorem \ref{lb110}). To prove this rigidity, Huang made the following contributions:\\
(a) In \cite{H Verlinde} and \cite{H Rigidity}, he showed that the rigidity   of $\Rep(V)$, which is a genus-$0$ property of CFT, follows from the genus-$1$ property of modular invariance. This important observation complements the Verlinde formula, which says that the fusion rules (which are also genus-$0$ data) can be computed using $S$-matrices (genus-$1$ data).\\
(b) The result of Y.Zhu \cite{Zhu96} on modular invariance of  VOA characters (as a special case of the modular invariance of genus-$1$ conformal blocks) is insufficient to prove the rigidity property. In \cite{H MI}, Huang proved the most general form of modular invariance by deriving certain differential equations on the moduli spaces of punctured complex tori, hence proving the rigidity of $\Rep(V)$. We refer the reader to \cite{HL13} for a detailed discussion of this issue.

Generalizing Wassermann's positivity result is also not easy. Indeed, Wassermann's key idea for proving the positivity of $\Lambda$ is to construct enough bounded intertwiners of conformal nets. Smeared intertwining operators (or rather their phases) provide such examples of bounded intertwiners, but they are not enough even in the case of type $A$ WZW models. Therefore, in \cite{Wassermann} Wassermann also considered products of smeared intertwining operators, which are essentially the same as the generalized smeared intertwining operators considered in the section \ref{Condition ABC}. Moreover, to prove the positivity of all $\Lambda$, one  needs to calculate certain braid relations between generalized smeared intertwining operators and their adjoints. However, Wassermann's calculation in \cite{Wassermann} (especially section 32) is model dependent and hence not easy to generalize. This is perhaps due to the fact that our understanding of the corresponding VOA structures was not mature enough by the time \cite{Wassermann} was written. But now, with the help of Huang-Lepowsky's important works on vertex tensor categories \cite{H 1,H 2,H 3,H 4,H ODE}, we are able to calculate the braid and adjoint relations for generalized (smeared) intertwining operators in a general setting, as we have already seen in the previous chapter.

The sesquilinear form $\Lambda$ is also closely related to the non-degenerate bilinear form constructed in \cite{HK07}. This will be explained in section \ref{lbb5}.
\end{rem}

For any $k\in\mathcal E\cap\mathcal F^\boxtimes$, since $W_k$ is irreducible, we have $N^k_{0k}=N^k_{k0}=1$. That the sesquilinear forms $\Lambda$ on $\mathcal V{k\choose 0~k}^*$ and on $\mathcal V{k\choose k~0}^*$ are  positive definite can be seen from the following two fusion relations:
\begin{gather}
Y_k\big(Y(u,z_2-z_1)v,z_1\big)=Y_k(u,z_2)Y_k(v,z_1),\label{eq323}\\
Y\big(\mathcal Y^0_{\overline kk}(\overline {w^{(k)}_2},z_2-z_1)w^{(k)}_1,z_1\big)=\mathcal Y^0_{\overline kk}(\overline{w^{(k)}_2},z_2)\mathcal Y^k_{k0}(w^{(k)}_1,z_1),\label{eq324}
\end{gather}
where $u,v\in V$, and $w^{(k)}_1,w^{(k)}_2\in W_k$. The first equation follows from proposition 2.13, and the second one follows from proposition 2.17.  (Note that these two fusion relations hold for any $V$-module $W_k$.) Moreover, the dual element of $Y_k$ is an orthonormal basis of $\mathcal V{k\choose 0~k}^*$, and the dual element of $\mathcal Y^k_{k0}$ is an orthonormal basis of $\mathcal V{k\choose k~0}^*$.

We  derive now some variants of transport formulas.

\begin{pp}\label{lb111}
	Let $I\in\mathcal J$. Choose distinct complex numbers $z_1,z_2\in I$. Choose $z_0\in I^c$ with argument $\arg z_0$. Define a continuous argument function $\arg_I$ on $I$, and let $\arg z_1=\arg_I(z_1),\arg z_2=\arg_I(z_2)$. Let $W_i,W_j$ be  unitary $V$-modules in $\mathcal F^\boxtimes$. \\
	(1) Let $W_s,W_r$ be unitary $V$-modules in $\mathcal F^\boxtimes$, and choose  $\mathcal Y_\gamma\in\mathcal V{r\choose j~s}$. Then for any $w^{(i)}_1,w^{(i)}_2\in W_i, w^{(j)}\in W_j$, we have the braid relation
	\begin{align}
	&\mathcal Y_\gamma(w^{(j)},z_0)\bigg(\sum_{\alpha,\beta\in\Theta^*_{is}}\Lambda^{\alpha\beta}\mathcal Y_{\beta^*}(\overline{w^{(i)}_2},z_2)\mathcal Y_{\alpha}(w^{(i)}_1,z_1)\bigg)\nonumber\\
	=&\bigg(\sum_{\alpha,\beta\in\Theta^*_{ir}}\Lambda^{\alpha\beta}\mathcal Y_{\beta^*}(\overline{w^{(i)}_2},z_2)\mathcal Y_{\alpha}(w^{(i)}_1,z_1)\bigg)\mathcal Y_\gamma(w^{(j)},z_0).\label{eq63}
	\end{align}
	(2) For any $w^{(i)}_1,w^{(i)}_2\in W_i$ and $ w^{(j)}\in W_j$, we have the transport formula (version 2)
	\begin{align}
	&\mathcal Y^j_{j0}(w^{(j)},z_0)\mathcal Y^0_{\overline ii}(\overline{w^{(i)}_2},z_2)\mathcal Y^i_{i0}(w^{(i)}_1,z_1)\nonumber\\
	=&\bigg(\sum_{\alpha,\beta\in\Theta^*_{ij}}\Lambda^{\alpha\beta}\mathcal Y_{\beta^*}(\overline{w^{(i)}_2},z_2)\mathcal Y_{\alpha}(w^{(i)}_1,z_1)\bigg)\mathcal Y^j_{j0}(w^{(j)},z_0).\label{eq64}
	\end{align}
(3) If $\arg z_0<\arg z_2<\arg z_0+2\pi$, then for any $w^{(i)}_2\in W_i,w^{(j)}\in W_j$, we have the transport formula (version 3)
\begin{align}
\mathcal Y^j_{j0}(w^{(j)},z_0)\mathcal Y^0_{\overline ii}(\overline{w^{(i)}_2},z_2)=\sum_{\alpha,\beta\in\Theta^*_{ij}}\Lambda^{\alpha\beta}\mathcal Y_{\beta^*}(\overline{w^{(i)}_2},z_2)\mathcal Y_{B_+\alpha}(w^{(j)},z_0).\label{eq330}
\end{align}
If $\arg z_2<\arg z_0<\arg z_2+2\pi$,  then equation \eqref{eq330} still holds, with $B_+\alpha$ replaced by $B_-\alpha$.
\end{pp}

\begin{proof}
(1) By rotating $z_1,z_2$ along $I$ and changing their arguments continuously, we can assume that $0<|z_1-z_2|<1$. Then clearly $\arg z_1$ is close to $\arg z_2$ as $z_2-z_1\rightarrow 0$. We also let $\arg(z_2-z_1)$ be close to $\arg z_2$ as $z_1\rightarrow 0$. Then by	equation \eqref{eq62}, proposition 2.13, and theorem \ref{lb16}, we have
	\begin{align}
	&\mathcal Y_\gamma(w^{(j)},z_0)\bigg(\sum_{\alpha,\beta\in\Theta^*_{is}}\Lambda^{\alpha\beta}\mathcal Y_{\beta^*}(\overline{w^{(i)}_2},z_2)\mathcal Y_{\alpha}(w^{(i)}_1,z_1)\bigg)\nonumber\\
	=&\mathcal Y_\gamma(w^{(j)},z_0)Y_s\big(\mathcal Y^0_{\overline ii}(\overline {w^{(i)}_2},z_2-z_1)w^{(i)}_1,z_1\big)\label{eqb13}\\
	=&Y_r\big(\mathcal Y^0_{\overline ii}(\overline {w^{(i)}_2},z_2-z_1)w^{(i)}_1,z_1\big)\mathcal Y_\gamma(w^{(j)},z_0)\label{eqb14}\\
	=&\bigg(\sum_{\alpha,\beta\in\Theta^*_{ir}}\Lambda^{\alpha\beta}\mathcal Y_{\beta^*}(\overline{w^{(i)}_2},z_2)\mathcal Y_{\alpha}(w^{(i)}_1,z_1)\bigg)\mathcal Y_\gamma(w^{(j)},z_0),\nonumber
	\end{align}
where \eqref{eqb13} and \eqref{eqb14} are understood as products of two generalized intertwining operators (see the beginning of chapter \ref{lb112}). This proves equation \eqref{eq63}.

(2) Equation \eqref{eq64} is a special case of equation \eqref{eq63}. 

(3) If $\arg z_0<\arg z_2<\arg z_0+2\pi$, we choose $z_1\in S^1\setminus\{-1\}$ close to $z_2$ and let $\arg z_1$ be close to $\arg z_2$ as $z_1\rightarrow z_2$. Then by equation \eqref{eq64}, corollary 2.18, and proposition 2.11, we have
\begin{align}
&\mathcal Y^j_{j0}(w^{(j)},z_0)\mathcal Y^0_{\overline ii}(\overline{w^{(i)}_2},z_2)\mathcal Y^i_{i0}(w^{(i)}_1,z_1)\nonumber\\
=&\sum_{\alpha,\beta\in\Theta^*_{ij}}\Lambda^{\alpha\beta}\mathcal Y_{\beta^*}(\overline{w^{(i)}_2},z_2)\mathcal Y_{B_+\alpha}(w^{(j)},z_0)\mathcal Y^i_{i0}(w^{(i)}_1,z_1).\nonumber
\end{align}
By proposition 2.3, we obtain equation \eqref{eq330}. The other case is proved in a similar way.
\end{proof}

\subsection{Positive definiteness of $\Lambda$}

Let $W_i,W_j$ be unitary $V$-modules in $\mathcal F^\boxtimes$, and let $W_k$ be in $\mathcal E\cap\mathcal F^\boxtimes$ as before. We prove in this section that the sesquilinear form $\Lambda$ on $\mathcal V{k\choose i~j}^*$ is positive definite. One suffices to prove this when $W_i,W_j$ are irreducible. Indeed, if $W_i,W_j$ are not necessarily irreducible, and have orthogonal decompositions $W_i=W_{i_1}\oplus W_{i_2}\oplus\cdots\oplus W_{i_m},W_j=W_{j_1}\oplus W_{j_2}\oplus\cdots\oplus W_{j_n}$, then clearly the unitary $V$-modules $W_{i_1},\dots,W_{i_m},W_{j_1},\dots,W_{j_n}$ are in $\mathcal F^\boxtimes$. It is easy to see that the transport matrix for $\mathcal V{k\choose i~j}^*$ can be diagonalized into the  $mn$ blocks of the transport matrices for $\mathcal V{k\choose i_a~j_b}^*$ ($1\leq a\leq m,1\leq b\leq n$). Therefore, if we choose $W_{i_1},\dots,W_{i_m},W_{j_1},\dots,W_{j_n}$  to be irreducible, and if we can prove that the transport matrix for every $\mathcal V{k\choose i_a~j_b}^*$ is positive definite, then the one for $\mathcal V{k\choose i~j}^*$ is also positive definite.

So let us assume that $W_i,W_j$ are irreducible. We let $\mathcal Y_{\kappa(i)}=\mathcal Y^i_{i0}$ and $\mathcal Y_{\kappa(j)}=\mathcal Y^j_{j0}$. Then $\mathcal Y_{\kappa(i)^*}=\mathcal Y^0_{\overline ii},\mathcal Y_{\kappa(j)^*}=\mathcal Y^0_{\overline jj}$.
Since $W_i$ (resp. $W_j$) is in $\mathcal F^\boxtimes$, there exits unitary $V$-modules $W_{i_1},\dots,W_{i_m}$ (resp. $W_{j_1},\dots,W_{j_n}$) in $\mathcal F\cup\overline{\mathcal F}$, such that $W_i$ (resp. $W_j$) is equivalent to a submodule of $W_{i_m\cdots i_1}=W_{i_m}\boxtimes\cdots\boxtimes W_{i_1}$ (resp. $W_{j_n\cdots j_1}$). Therefore, we can  choose a chain of non-zero irreducible unitary intertwining operators $\mathcal Y_{\sigma_2},\dots,\mathcal Y_{\sigma_m}$ (resp. $\mathcal Y_{\rho_2},\dots,\mathcal Y_{\rho_n}$) with charge spaces $W_{i_2},\dots,W_{i_m}$ (resp. $W_{j_2},\dots,W_{j_n}$) respectively, such that $W_{i_1}$ (resp. $W_{j_1}$) is the source space of $\mathcal Y_{\sigma_2}$ (resp. $\mathcal Y_{\rho_2}$), and that $W_i$ (resp. $W_j$) is the target space of $\mathcal Y_{\sigma_m}$ (resp. $\mathcal Y_{\rho_n}$).

Fix non-zero quasi-primary vectors $w^{(i_1)}\in W_{i_1},\dots,w^{(i_m)}\in W_{i_m},w^{(j_1)}\in W_{j_1},\dots,w^{(j_n)}\in W_{j_n}$. If $\mathcal F$ satisfies condition \ref{CondB} in section \ref{Condition ABC}, we assume moreover that $w^{(i_1)}\in E^1(W_{i_1}),\dots,w^{(i_m)}\in E^1(W_{i_m}),w^{(j_1)}\in E^1(W_{j_1}),\dots,w^{(j_n)}\in E^1(W_{j_n})$. Choose disjoint open intervals $I,J\in\mathcal J(S^1\setminus\{-1\})$, and choose  $(I_1,\dots,I_m)\in \mathfrak O_m(I),(J_1,\dots,J_n)\in\mathfrak O_n(J)$.  We define two sets $\mathcal A=\mathcal M_V(I)_\infty\times  C^\infty_c(I_1)\times\cdots\times  C^\infty_c(I_m)$ and $\mathcal B=\mathcal M_V(J)_\infty\times C^\infty_c(J_1)\times\cdots\times C^\infty_c(J_n)$. For any $ a=(x,f_1,\dots,f_m)\in\mathcal A$ and $ b=(y,g_1,\dots,g_n)\in\mathcal B$, we define two linear operators $A( a):\mathcal H^\infty_0\rightarrow \mathcal H^\infty_i$ and $B( b):\mathcal H^\infty_0\rightarrow \mathcal H^\infty_j$ as follows: if $\xi^{(0)}\in\mathcal H^\infty_0$ then
\begin{gather}
A( a)\xi^{(0)}=\pi_i(x)\mathcal Y_{\sigma_m\cdots\sigma_2,\kappa(i)}(w^{(i_m)},f_m;\dots,w^{(i_1)},f_1)\xi^{(0)},\\
B( b)\xi^{(0)}=\pi_j(y)\mathcal Y_{\rho_n\cdots\rho_2,\kappa(j)}(w^{(j_n)},g_n;\dots,w^{(j_1)},g_1)\xi^{(0)}.
\end{gather}
By proposition \ref{lbb2}, the formal adjoints of these two linear operators exist.

\begin{lm}\label{lb20}
	For any $N\in\mathbb Z_{>0}$, $ a_1,\dots, a_N\in\mathcal A$, $ b_1,\dots, b_N\in\mathcal B$ and $\xi^{(0)}_1,\dots,\xi^{(0)}_N\in\mathcal H^\infty_0$, we have
	\begin{equation}
	\sum_{s,t=1,\dots,N}\langle B( b_s)A( a_t)^\dagger A( a_s)\xi^{(0)}_s|B( b_t)\xi^{(0)}_t\rangle\geq0.\label{eq73}
	\end{equation}
\end{lm}

\begin{proof}
	Suppose that 
	\begin{equation}
	\sum_{s,t=1,\dots,N}\langle B( b_s)A( a_t)^\dagger A( a_s)\xi^{(0)}_s|B( b_t)\xi^{(0)}_t\rangle\notin[0,+\infty).
	\end{equation}
Then we can find $\varepsilon>0$, such that  for any $\tau\in[0,+\infty)$,
	\begin{gather}
	\bigg|\sum_{s,t=1,\dots,N}\langle B( b_s)A( a_t)^\dagger A( a_s)\xi^{(0)}_s|B( b_t)\xi^{(0)}_t\rangle-\tau \bigg|\geq\varepsilon.\label{eq72}
	\end{gather}
By proposition \ref{lb106}, for any $x\in\mathcal M_V(J^c)$ and $r=1,\dots,N$, we have $\pi_j(x)\overline{B( b_r)}\subset\overline{B( b_r)}\pi_0(x)$. We also  regard $\overline{B( b_r)}$ as an unbounded operator on $\mathcal H_0\oplus\mathcal H_j$, being the original operator when restrict to $\mathcal H_0$, and  the zero map when restricted to $\mathcal H_j$. We let  $x$ act on  $\mathcal H_0\bigoplus\mathcal H_j$ diagonally (i.e., $x=\diag(\pi_0(x),\pi_I(x))$). Then $x\overline{B( b_r)}\subset\overline{B( b_r)}x$. Since $x^*$ also satisfies this relation,	elements in $\mathcal M_V(J^c)$ commute strongly with $\overline{B( b_r)}$. Therefore, if we take the right polar decomposition $\overline{B( b_r)}=K_rV_r$ (where $K_r$ is self-adjoint and $V_r$ is a partial isometry), then $\mathcal M_V(J^c)$ commutes strongly with $V_r$ and $K_r$.  We let $K_r=\int_{-\infty}^{+\infty} \lambda dQ_r(\lambda)$ be the spectral decomposition of $K_r$. Then for each $\lambda\geq0$, $Q_r(\lambda)=\int_{-\infty}^\lambda dQ_r(\mu)$ commutes with $\mathcal M_V(J^c)$. Therefore, the \emph{bounded} operator $\overline{Q_r(\lambda)\overline{B( b_r)}}$ commutes with $\mathcal M_V(J^c)$, i.e., $\overline{Q_r(\lambda)\overline{B( b_r)}}\in\mathrm{Hom}_{\mathcal M_V(J^c)}(\mathcal H_0,\mathcal H_j)$.
	
Now we  choose a real number $M>0$, such that for any $s,t=1,\dots,N$,
\begin{gather}
\lVert  B( b_s)A( a_t)^\dagger A( a_s)\xi_s^{(0)}\lVert\leq M,~~\lVert  B( b_t)\xi_t^{(0)}\lVert\leq M.
\end{gather}
For each $r=1,\dots,N$, since the projection $Q_r(\lambda)$ converges strongly to $1$ as $\lambda\rightarrow+\infty$, there exists $\lambda_r>0$, such that for any $t=1,\dots,N$,
	\begin{gather}
	\lVert  B( b_r)A( a_t)^\dagger A( a_r)\xi_r^{(0)}-Q_r(\lambda_r)B( b_r)A( a_t)^\dagger A( a_r)\xi_r^{(0)}\lVert<\frac{\varepsilon}{4MN^2},\\
	\lVert  B( b_r)\xi_r^{(0)}-Q_r(\lambda_r)B( b_r)\xi_r^{(0)}\lVert<\frac{\varepsilon}{4MN^2}.
	\end{gather}
	We let $\mathfrak B(b_r)=\overline{Q_r(\lambda_r)\overline{B( b_r)}}\in\mathrm{Hom}_{\mathcal M_V(J^c)}(\mathcal H_0,\mathcal H_j)$, then the above inequalities imply that
	\begin{align}
	\bigg|\sum_{s,t}\langle\mathfrak B( b_s)A( a_t)^\dagger A( a_s)\xi^{(0)}_s|\mathfrak B( b_t)\xi^{(0)}_t\rangle-\sum_{s,t}\langle B( b_s)A( a_t)^\dagger A( a_s)\xi^{(0)}_s|B( b_t)\xi^{(0)}_t\rangle\bigg|<\frac \varepsilon 2.\label{eq70}
	\end{align}
	
Now, for any $1\leq r\leq N$, since $\mathfrak B(b_r)\in\mathrm{Hom}_{\mathcal M_V(J^c)}(\mathcal H_0,\mathcal H_j)$, we also have	$\mathfrak B(b_r)^*\in\mathrm{Hom}_{\mathcal M_V(J^c)}(\mathcal H_j,\mathcal H_0)$. Thus, for any $1\leq s,t\leq N$, we have $\mathfrak B(b_s)^* \mathfrak B(b_t)\in\End_{\mathcal M_V(J^c)}(\mathcal H_0)=\mathcal M_V(J^c)'$. By Haag duality, $\mathfrak B(b_s)^* \mathfrak B(b_t)\in\mathcal M_V(J)$.
By proposition \ref{lb106}, $\pi_i\big(\mathfrak B(b_s)^* \mathfrak B(b_t)\big)\overline{A( a_t)}\subset\overline{A( a_t)}\mathfrak B(b_s)^* \mathfrak B(b_t)$. In particular,  $\mathfrak B(b_s)^* \mathfrak B(b_t) \mathscr D(\overline{A(a_t)})\subset  \mathscr D(\overline{A(a_t)})$. Since $\xi^{(0)}_t\in\mathcal H^\infty_0\subset\mathscr D(\overline{A(a_t)})$, 
	\begin{equation}
	\mathfrak B(b_s)^* \mathfrak B(b_t)\xi_t^{(0)}\in  \mathfrak B(b_s)^* \mathfrak B(b_t) \mathscr D(\overline{A(a_t)})\subset  \mathscr D(\overline{A(a_t)}).
	\end{equation}
Therefore,
	\begin{align}
	&\langle\mathfrak B( b_s)A( a_t)^\dagger A( a_s)\xi^{(0)}_s|\mathfrak B( b_t)\xi^{(0)}_t\rangle\nonumber\\
	=&\langle A( a_t)^\dagger A( a_s)\xi^{(0)}_s|\mathfrak B( b_s)^*\mathfrak B( b_t)\xi^{(0)}_t\rangle\nonumber\\
	=&\langle\overline{ A( a_t)}^* \cdot\overline{A( a_s)}\xi^{(0)}_s|\mathfrak B( b_s)^*\mathfrak B( b_t)\xi^{(0)}_t\rangle\nonumber\\
	=&\langle \overline{A( a_s)}\xi^{(0)}_s|\overline{ A( a_t)}\mathfrak B( b_s)^*\mathfrak B( b_t)\xi^{(0)}_t\rangle.\label{eq325}
	\end{align}
Let $\overline{A(a_s)}=H_sU_s$ be the right polar decomposition of $\overline{A(a_s)}$, and take the spectral decomposition $H_s=\int_{-\infty}^{+\infty} \kappa dP_s(\kappa)$. Then for each $s$, we can find  $\kappa_s>0$ such that 
	\begin{equation}\label{eq69}
	\bigg|\sum_{s,t}\langle \overline{A( a_s)}\xi^{(0)}_s|\overline{ A( a_t)}\mathfrak B( b_s)^*\mathfrak B( b_t)\xi^{(0)}_t\rangle-\sum_{s,t}\langle \mathfrak A( a_s)\xi^{(0)}_s|\mathfrak A( a_t)\mathfrak B( b_s)^*\mathfrak B( b_t)\xi^{(0)}_t\rangle\bigg|<\frac \varepsilon 2,
	\end{equation}
	where $\mathfrak A( a_s)=\overline{P_s(\kappa_s)\overline{A(a_s)}}\in\mathrm{Hom}_{\mathcal M(I^c)}(\mathcal H_0,\mathcal H_i)$. Note that  $\mathfrak A( a_s)$ and $\mathfrak B(b_t)$ are  bounded operators. Set 
\begin{align}
\tau=\sum_{s,t}\langle \mathfrak A( a_s)\xi^{(0)}_s|\mathfrak A( a_t)\mathfrak B( b_s)^*\mathfrak B( b_t)\xi^{(0)}_t\rangle=\sum_{s,t}\langle \mathfrak B( b_s)\mathfrak A( a_t)^*\mathfrak A( a_s)\xi^{(0)}_s|\mathfrak B( b_t)\xi^{(0)}_t\rangle.
\end{align}	
Then by inequalities \eqref{eq70}, \eqref{eq69}, and equation \eqref{eq325},
	\begin{equation}\label{eq71}
	\bigg|\sum_{s,t}\langle B( b_s)A( a_t)^\dagger A( a_s)\xi^{(0)}_s|B( b_t)\xi^{(0)}_t\rangle-\tau\bigg|<\varepsilon.
	\end{equation}
	
We now show that $\tau\geq0$, which will contradict condition \eqref{eq72} and thus prove inequality \eqref{eq73}. Let $M(N,\mathbb C)$ be the complex valued $N\times N$ matrix algebra. By evaluating between vectors in $\mathcal H_0^{\oplus N}$, we find that the $\mathcal M_V(I)$-valued matrix $[\mathfrak A(a_t)^*\mathfrak A(a_s)]_{N\times N}$ is a positive element in the von Neumann algebra $\mathcal M_V(I)\otimes M(N,\mathbb C)$.  So $[\pi_{j,I}(\mathfrak A(a_t)^*\mathfrak A(a_s))]_{N\times N}\in \pi_{j,I}(\mathcal M_V(I))\otimes M(N,\mathbb C)$ is also positive. Therefore, if for each $s$ we define a vector $\eta_s=\mathfrak B(b_s)\xi_s^{(0)}$, then 
	\begin{equation}
	\sum_{s,t}(\pi_{j,I}(\mathfrak A(a_t)^*\mathfrak A(a_s))\eta_s|\eta_t)\geq0.
	\end{equation}
Since $\mathfrak B(b_s)\in\mathrm{Hom}_{\mathcal M_V(J^c)}(\mathcal H_0,\mathcal H_j)\subset\mathrm{Hom}_{\mathcal M_V(I)}(\mathcal H_0,\mathcal H_j)$, we have
	\begin{align}
\mathfrak B( b_s)\mathfrak A( a_t)^*\mathfrak A( a_s)\xi^{(0)}_s=\pi_{j,I}\big(\mathfrak A(a_t)^*\mathfrak A(a_s)\big)\mathfrak B( b_s)\xi_s^{(0)}=\pi_{j,I}(\mathfrak A(a_t)^*\mathfrak A(a_s))\eta_s.
	\end{align}
	Hence
	\begin{align}
\tau=	\sum_{s,t}\langle \mathfrak B( b_s)\mathfrak A( a_t)^*\mathfrak A( a_s)\xi^{(0)}_s|\mathfrak B( b_t)\xi^{(0)}_t\rangle=\sum_{s,t}(\pi_{j,I}(\mathfrak A(a_t)^*\mathfrak A(a_s))\eta_s|\eta_t)\geq0.
	\end{align}
\end{proof}

\begin{thm}\label{lb110}
Suppose that $V$ is unitary, energy bounded, and strongly local, and $\mathcal F$ is a non-empty set of non-zero irreducible unitary $V$-modules satisfying condition \ref{CondA} or \ref{CondB} in section \ref{Condition ABC}. Let $W_i,W_j$  be unitary $V$-modules in $\mathcal F^\boxtimes$. Then the sesquilinear form $\Lambda$ on $W_i\boxtimes W_j$ is an inner product. Equivalently, for any irreducible unitary $V$-module $W_k$ in $\mathcal E\cap\mathcal F^\boxtimes$, 	the sesquilinear form $\Lambda$ on $\mathcal V{k\choose i~j}^*$ is positive definite.
\end{thm}

\begin{proof}
As argued at the beginning of this section, we can assume, without loss of generality, that $W_i,W_j$ are irreducible.\\

Step 1. We first show that $\Lambda$ is positive. For each $k\in\mathcal E\cap\mathcal F^\boxtimes$, we   choose a basis $\{\mathcal Y_\alpha:\alpha\in \Theta^k_{ij}\}$ of $\mathcal V{k\choose i~j}$,  let $\{\widecheck{\mathcal Y}^\alpha:\alpha\in \Theta^k_{ij}\}$ be its dual basis in $\mathcal V{k\choose i~j}^*$, and define an inner product on $\mathcal V{k\choose i~j}^*$ under which $\{\widecheck{\mathcal Y}^\alpha:\alpha\in \Theta^k_{ij}\}$ becomes orthonormal. We extend these inner products to a unitary structure on $W_{ij}=\bigoplus_k \mathcal V{k\choose i~j}^*\otimes W_k$, just as we extend $\Lambda$ using \eqref{eqb5}. As usual, we let $\mathcal H_{ij}$ be the corresponding $\mathcal M_V$-module. The sesquilinear form $\Lambda$ on $W_{ij}$ defined by \eqref{eqb5} can be  extended uniquely to a \emph{continuous} sesquilinear form $\Lambda$ on the Hilbert space $\mathcal H_{ij}$.

Choose intertwining operators $\mathcal Y_{\sigma_2},\dots,\mathcal Y_{\sigma_m},\mathcal Y_{\rho_2},\dots,\mathcal Y_{\rho_n}$, disjoint open intervals $I,J,(I_1,\dots,I_m)\in\mathfrak O_m(I),(J_1,\dots,J_n)\in\mathfrak O_n(J)$, and non-zero quasi-primary vectors $w^{(i_1)},\dots,w^{(i_m)},w^{(j_1)},\dots,w^{(j_n)}$ as 
at the beginning of this section. By proposition \ref{lb105}, for each $l\in\mathbb Z_{\geq0}$, vectors of the form
\begin{align}
B(b)\xi^{(0)}=\pi_j(y)\mathcal Y_{\rho_n\cdots\rho_2,\kappa(j)}(w^{(j_n)},g_n;\dots;w^{(j_1)},g_1)\xi^{(0)}\label{eq326}
\end{align}
span a core for $\overline {L_0}^l$ in $\mathcal H^\infty_j$, where $b=(y,g_1,\dots,g_n)\in\mathcal B$, and $\xi^{(0)}\in\mathcal H^\infty_0$. For any $a=(x,f_1,\dots,f_m)\in\mathcal A$, we define an unbounded operator $\widetilde A(a):\mathcal H_j\rightarrow \mathcal H_{ij}$ with domain $\mathcal H^\infty_{j}$ to satisfy
\begin{align}
\widetilde A(a)=\pi_{ij}(x)\mathcal Y_{\sigma_m\cdots\sigma_2,i\boxtimes j}(w^{(i_m)},f_m;\dots;w^{(i_1)},f_1).
\end{align}
Then, by inequality \eqref{eq306}, vectors of the form \eqref{eq326} span a core for $\widetilde A(a)$.
Therefore, by proposition \ref{lb105}, vectors of the form
\begin{align}
\xi^{(ij)}=\sum_{s=1,\dots,N}\widetilde A(a_s)B(b_s)\xi^{(0)}_s\label{eq327}
\end{align}
form a dense subspace of $\mathcal H_{ij}$, where $N=1,2,\dots$, and for each $s$, $a_s=(x_s,f_{s,1},\dots,f_{s,m})\in\mathcal A,b_s=(y_s,g_{s,1},\dots,g_{s,n})\in\mathcal B$, and $\xi^{(0)}_s\in\mathcal H_0$. If we can prove, for any $\xi^{(ij)}\in \mathcal H_{ij}$ of the form \eqref{eq327}, that $\Lambda(\xi^{(ij)}|\xi^{(ij)})\geq0$, then $\Lambda$ is positive on $W_i\boxtimes W_j$.\\

Step 2. We show that $\Lambda(\xi^{(ij)}|\xi^{(ij)})\geq0$. Let us simplify the notations a little bit. Let $\vec w^{(\vec i)}=(w^{(i_1)},\dots,w^{(i_m)}),\vec \sigma=(\sigma_2,\dots,\sigma_m),\vec f_s=(f_{s,1},\dots,f_{s,m})$. If $\mathcal Y_\alpha$ is an intertwining operator whose charge space, source space, and target space are inside $\mathcal F^\boxtimes$, then we set
\begin{align}
\mathcal Y_{\vec \sigma,\alpha }(\vec w^{(\vec i)},\vec f_s)=\mathcal Y_{\sigma_m\cdots\sigma_2,\alpha}(w^{(i_m)},f_{s,m};\dots;w^{(i_1)},f_{s,1}).
\end{align}
Similarly, we let  $\vec w^{(\vec j)}=(w^{(j_1)},\dots,w^{(j_n)}),\vec \rho=(\rho_2,\dots,\rho_n),\vec g_s=(g_{s,1},\dots,g_{s,n})$. $\mathcal Y_{\vec \rho,\kappa(j)}(w^{(\vec j)},\vec g_s)$ is defined in a similar way.

 Assume, without loss of generality, that $I$ is anti-clockwise to $J$, i.e., for any $z\in I,\zeta\in J$, we have $-\pi<\arg\zeta<\arg z<\pi$. By  proposition \ref{lb106}, for any $s=1,\dots,N$,
\begin{align}
&\widetilde A(a_s)B(b_s)=x_s\mathcal Y_{\vec \sigma,i\boxtimes j }(\vec w^{(\vec i)},\vec f_s)y_s\mathcal Y_{\vec \rho,\kappa(j)}(\vec w^{(\vec j)},\vec g_s)\nonumber\\
=&\sum_{\alpha\in\Theta^*_{ij}}\widecheck{\mathcal Y}^\alpha\otimes x_sy_s\mathcal Y_{\vec \sigma,\alpha }(\vec w^{(\vec i)},\vec f_s)\mathcal Y_{\vec \rho,\kappa(j)}(\vec w^{(\vec j)},\vec g_s).
\end{align}
So for any $s,t=1,\dots,N$,
\begin{align}
&\Lambda\big(\widetilde A(a_s)B(b_s)\xi^{(0)}_s\big|\widetilde A(a_t)B(b_t)\xi^{(0)}_t\big)\nonumber\\
=&\sum_{\alpha,\beta\in\Theta^*_{ij}}\Lambda^{\alpha\beta}\big\langle x_sy_s\mathcal Y_{\vec \sigma,\alpha }(\vec w^{(\vec i)},\vec f_s)\mathcal Y_{\vec \rho,\kappa(j)}(\vec w^{(\vec j)},\vec g_s)\xi^{(0)}_s\big| x_ty_t\mathcal Y_{\vec \sigma,\beta }(\vec w^{(\vec i)},\vec f_t)\mathcal Y_{\vec \rho,\kappa(j)}(\vec w^{(\vec j)},\vec g_t)\xi^{(0)}_t \big\rangle\nonumber\\
=&\sum_{\alpha,\beta\in\Theta^*_{ij}}\Lambda^{\alpha\beta}\big\langle \mathcal Y_{\vec \sigma,\beta }(\vec w^{(\vec i)},\vec f_t)^\dagger x_t^*x_sy_s\mathcal Y_{\vec \sigma,\alpha }(\vec w^{(\vec i)},\vec f_s)\mathcal Y_{\vec \rho,\kappa(j)}(\vec w^{(\vec j)},\vec g_s)\xi^{(0)}_s\big|y_t\mathcal Y_{\vec \rho,\kappa(j)}(\vec w^{(\vec j)},\vec g_t)\xi^{(0)}_t \big\rangle\nonumber\\
=&\sum_{\alpha,\beta\in\Theta^*_{ij}}\Lambda^{\alpha\beta}\big\langle \mathcal Y_{\vec \sigma,\beta }(\vec w^{(\vec i)},\vec f_t)^\dagger x_t^*x_sy_s\mathcal Y_{\vec \sigma,\alpha }(\vec w^{(\vec i)},\vec f_s)\mathcal Y_{\vec \rho,\kappa(j)}(\vec w^{(\vec j)},\vec g_s)\xi^{(0)}_s\big|B(b_t)\xi^{(0)}_t \big\rangle.\label{eqb6}
\end{align}
By corollary 2.18 and theorem \ref{lb109},
\begin{align}
&\sum_{\alpha,\beta\in\Theta^*_{ij}} \Lambda^{\alpha\beta}\mathcal Y_{\vec \sigma,\beta }(\vec w^{(\vec i)},\vec f_t)^\dagger x_t^*x_sy_s\mathcal Y_{\vec \sigma,\alpha }(\vec w^{(\vec i)},\vec f_s)\mathcal Y_{\vec \rho,\kappa(j)}(\vec w^{(\vec j)},\vec g_s)\nonumber\\
=&\sum_{\alpha,\beta\in\Theta^*_{ij}} \Lambda^{\alpha\beta}\mathcal Y_{\vec \sigma,\beta }(\vec w^{(\vec i)},\vec f_t)^\dagger x_t^*x_sy_s\mathcal Y_{\vec \rho,B_+\alpha}(\vec w^{(\vec j)},\vec g_s)\mathcal Y_{\vec \sigma,\kappa(i) }(\vec w^{(\vec i)},\vec f_s)\nonumber\\
=&\sum_{\alpha,\beta\in\Theta^*_{ij}} \Lambda^{\alpha\beta}y_s\mathcal Y_{\vec \sigma,\beta }(\vec w^{(\vec i)},\vec f_t)^\dagger \mathcal Y_{\vec \rho,B_+\alpha}(\vec w^{(\vec j)},\vec g_s)x_t^*x_s\mathcal Y_{\vec \sigma,\kappa(i) }(\vec w^{(\vec i)},\vec f_s).\label{eq329}
\end{align}
By theorem \ref{lb108}, for each $l=2,\dots,m$, there exists an  intertwining operators $\widetilde{\sigma_l}$ having the same type as that of $\sigma_l$, such that \eqref{eq328} holds for all $\mathcal Y_\alpha$ whose charge space, source space, and target space are unitary $V$-modules in $\mathcal F^\boxtimes$. Let $h_{t,1}=e^{i\pi\Delta_{w^{(i_1)}}}(e_{2-2\Delta_{w^{(i_1)}}}f_{t,1}),\dots,h_{t,m}=e^{i\pi\Delta_{w^{(i_m)}}}(e_{2-2\Delta_{w^{(i_m)}}}f_{t,m})$. Set $\vec h_t=(h_{t,1},\dots,h_{t,m}),\overline{\vec h_t}=(\overline {h_{t,1}},\dots,\overline {h_{t,m}}),  \overline{\vec w^{(\vec i)}}=(\overline{w^{(i_1)}},\dots,\overline{w^{(i_m)}})$. Then \eqref{eq329} equals
\begin{align}
\sum_{\alpha,\beta\in\Theta^*_{ij}} \Lambda^{\alpha\beta}y_s\mathcal Y_{\vec {\widetilde\sigma},\beta^* }(\overline{\vec w^{(\vec i)}},\overline{\vec h_t})\mathcal Y_{\vec \rho,B_+\alpha}(\vec w^{(\vec j)},\vec g_s)x_t^*x_s\mathcal Y_{\vec \sigma,\kappa(i) }(\vec w^{(\vec i)},\vec f_s).\label{eq331}
\end{align}
By equation \eqref{eq330} and theorem \ref{lb109}, \eqref{eq331} equals
\begin{align}
y_s\mathcal Y_{\vec \rho,\kappa(j)}(\vec w^{(\vec j)},\vec g_s)\mathcal Y_{\vec {\widetilde\sigma},\kappa(i)^* }(\overline{\vec w^{(\vec i)}},\overline{\vec h_t})x_t^*x_s\mathcal Y_{\vec \sigma,\kappa(i) }(\vec w^{(\vec i)},\vec f_s),
\end{align}
which, due to equation \eqref{eq328}, also equals
\begin{align}
&y_s\mathcal Y_{\vec \rho,\kappa(j)}(\vec w^{(\vec j)},\vec g_s)\mathcal Y_{\vec \sigma,\kappa(i) }(\vec w^{(\vec i)},\vec f_t)^\dagger x_t^*x_s\mathcal Y_{\vec \sigma,\kappa(i) }(\vec w^{(\vec i)},\vec f_s)\nonumber\\
=&B(b_s)A(a_t)^\dagger A(a_s).
\end{align}
Substitute this expression into equation \eqref{eqb6}, we see that
\begin{align}
\Lambda\big(\widetilde A(a_s)B(b_s)\xi^{(0)}_s\big|\widetilde A(a_t)B(b_t)\xi^{(0)}_t\big)=\langle B(b_s)A(a_t)^\dagger A(a_s)\xi^{(0)}_s|B(b_t)\xi^{(0)}_t \rangle.
\end{align}
Therefore, by lemma \ref{lb20},
\begin{align}
\Lambda(\xi^{(ij)}|\xi^{(ij)})
=&\sum_{s,t=1,\dots,N}\Lambda\big(\widetilde A(a_s)B(b_s)\xi^{(0)}_s\big|\widetilde A(a_t)B(b_t)\xi^{(0)}_t\big)\nonumber\\
=&\sum_{s,t=1,\dots,N}\langle B(b_s)A(a_t)^\dagger A(a_s)\xi^{(0)}_s|B(b_t)\xi^{(0)}_t \rangle\geq0.
\end{align}

	Step 3 (See also \cite{HK07} theorem 3.4). We prove the non-degeneracy of $\Lambda$ using the rigidity of $\Rep(V)$. Since $\Lambda$ is positive, for each $k\in\mathcal E$, we can choose  a basis $\Theta^k_{ij}$, such that the transport matrix $\Lambda$ is a diagonal, and that the  entries are either $1$ or $0$. Thus, we have the transport formula 
	\begin{align}
	Y_j(\mathcal Y^0_{\overline ii}(\overline {w^{(i)}_2},z_2-z_1)w^{(i)}_1,z_1)=\sum_{\alpha\in\Theta^*_{ij}}\lambda_\alpha\mathcal Y_{\alpha^*}(\overline{w^{(i)}_2},z_2)\mathcal Y_{\alpha}(w^{(i)}_1,z_1),\label{eq224}
	\end{align}
where each $\lambda_\alpha$ is either $1$ or $0$. For each $k\in\mathcal E$, we let $n^k_{ij}$ be the number of $\alpha\in\Theta^k_{ij}$ satisfying $\lambda_\alpha=1$. Then clearly $n^k_{ij}\leq N^k_{ij}$. If we can show that $n^k_{ij}= N^k_{ij}$, then the non-degeneracy of $\Lambda$ follows.
	
	Since $W_i$ is irreducible, we have $N^0_{\overline i i}=N^i_{0i}=1$. So there exists a complex number $\mu_i\neq0$ such that $\mathcal Y^0_{\overline ii}$  represents  the morphism $\mu_i \ev_i:W_{\overline i}\boxtimes W_i\rightarrow V$. We also regard $\mathcal Y_\alpha$ as a morphism $W_i\boxtimes W_j\rightarrow W_k$, and $\mathcal Y_{\alpha^*}$ a morphism $W_{\overline i}\boxtimes W_k\rightarrow W_j$ (see section 2.4). Then equation \eqref{eq224} is equivalent to the following relation for morphisms:
	\begin{align}
	\mu_i (\ev_i\otimes \id_j)=\sum_{k\in\mathcal E}\sum_{\alpha\in\Theta^k_{ij}}\lambda_{\alpha}\mathcal Y_{\alpha^*}\circ(\id_{\overline i}\otimes\mathcal Y_{\alpha}).
	\end{align}
	By equation (2.64),
	\begin{align*}
\mu_i(\id_i\otimes\id_j)
	=&\mu_i[(\id_i\otimes\ev_i)\circ(\coev_i\otimes\id_i)]\otimes\id_j\\
	=&\mu_i(\id_i\otimes\ev_i\otimes\id_j)\circ(\coev_i\otimes\id_i\otimes\id_j)\\
	=&\sum_{k\in\mathcal E}\sum_{\alpha\in\Theta^k_{ij}}\lambda_\alpha\big(\id_i\otimes(\mathcal Y_{\alpha^*}\circ(\id_{\overline i}\otimes\mathcal Y_{\alpha}))\big)\circ(\coev_i\otimes\id_i\otimes\id_j)\\
	=&\sum_{k\in\mathcal E}\sum_{\alpha\in\Theta^k_{ij}}\lambda_\alpha(\id_i\otimes\mathcal Y_{\alpha^*})\circ(\id_i\otimes\id_{\overline i}\otimes\mathcal Y_\alpha)\circ(\coev_i\otimes\id_i\otimes\id_j)\\
	=&\sum_{k\in\mathcal E}\sum_{\alpha\in\Theta^k_{ij}}\lambda_\alpha(\id_i\otimes\mathcal Y_{\alpha^*})\circ(\coev_i\otimes\mathcal Y_\alpha)\\
		=&\sum_{k\in\mathcal E}\sum_{\alpha\in\Theta^k_{ij}}\lambda_\alpha(\id_i\otimes\mathcal Y_{\alpha^*})\circ(\coev_i\otimes\id_k)\circ(\id_0\otimes\mathcal Y_\alpha).
	\end{align*}
	This equation implies that the isomorphism $\mu_i(\id_i\otimes \id_j):W_i\boxtimes W_j\rightarrow W_i\boxtimes W_j$ factors through the homomorphism $$\Phi:\sum^\oplus_{k\in\mathcal E}\sum^\oplus_{\alpha\in\Theta^k_{ij},\lambda_{\alpha}\neq0}\id_0\otimes\mathcal Y_\alpha:W_i\boxtimes W_j\rightarrow W=\bigoplus_{k\in\mathcal E}\bigoplus_{\alpha\in\Theta^k_{ij},\lambda_{\alpha}\neq0} W_k.$$ So $\Phi$ must be injective, which implies that $W_i\boxtimes W_j$ can be embedded as a submodule of $W$. Note that $W_i\boxtimes W_j\simeq\bigoplus_{k\in\mathcal E}W_k^{\oplus N^k_{ij}}$ and $W\simeq\bigoplus_{k\in\mathcal E}W_k^{\oplus n^k_{ij}}$. So we must have $n^k_{ij}\geq N^k_{ij}$. 
\end{proof}

\section{Unitarity of the ribbon fusion categories}

In this chapter, we still assume  that $V$ is unitary, energy bounded, and strongly local, and that $\mathcal F$ is a non-empty set of non-zero irreducible unitary $V$-modules satisfying condition \ref{CondA} or \ref{CondB} in section \ref{Condition ABC}. If $W_i,W_j$ are unitary $V$-modules in $\mathcal F^\boxtimes$, then by theorem \ref{lb110}, for each $k\in\mathcal E$, the sesquilinear form $\Lambda$ on $\mathcal V{k\choose i~j}^*$ defined by the transport matrix is  an inner product.  Therefore, we have a unitary structure on $\mathcal F^\boxtimes$ defined by $\Lambda$ (see section 2.4). We fix this unitary structure, and show that the ribbon fusion category $\Rep^\uni_{\mathcal F^\boxtimes}(V)$ is unitary.

We first note that the inner product $\Lambda$ on $\mathcal V{k\choose i~j}^*$ induces naturally an antilinear isomorphism map $\mathcal V{k\choose i~j}\rightarrow \mathcal V{k\choose i~j}^*$. We then define the inner product $\Lambda$ on $\mathcal V{k\choose i~j}$ so that this map becomes anti-unitary. Then a basis $\Theta^k_{ij}\subset\mathcal V{k\choose i~j}$ is orthonormal if and only if its dual basis is an orthonormal basis of  $\mathcal V{k\choose i~j}^*$. Therefore, if for each $k\in\mathcal E\cap\mathcal F^\boxtimes$, $\Theta^k_{ij}$ is an orthonormal basis of $\mathcal V{k\choose i~j}$, then  the transport formulas \eqref{eq62}, \eqref{eq63} and \eqref{eq330} become
\begin{gather}
Y_j\big(\mathcal Y^0_{\overline ii}(\overline {w^{(i)}_2},z_2-z_1)w^{(i)}_1,z_1\big)=\sum_{\alpha\in\Theta^*_{ij}}\mathcal Y_{\alpha^*}(\overline{w^{(i)}_2},z_2)\mathcal Y_{\alpha}(w^{(i)}_1,z_1),\\
\mathcal Y_\gamma(w^{(j)},z_0)\bigg(\sum_{\alpha\in\Theta^*_{is}}\mathcal Y_{\alpha^*}(\overline{w^{(i)}_2},z_2)\mathcal Y_{\alpha}(w^{(i)}_1,z_1)\bigg)=\bigg(\sum_{\alpha\in\Theta^*_{ir}}\mathcal Y_{\alpha^*}(\overline{w^{(i)}_2},z_2)\mathcal Y_{\alpha}(w^{(i)}_1,z_1)\bigg)\mathcal Y_\gamma(w^{(j)},z_0),\label{eqb11}\\
\mathcal Y^j_{j0}(w^{(j)},z_0)\mathcal Y^0_{\overline ii}(\overline{w^{(i)}_2},z_2)=\sum_{\alpha\in\Theta^*_{ij}}\mathcal Y_{\alpha^*}(\overline{w^{(i)}_2},z_2)\mathcal Y_{B_+\alpha}(w^{(j)},z_0).\label{eqb12}
\end{gather}

\subsection{Unitarity of braid matrices}

For any   unitary $V$-modules $W_i,W_j$ in $\mathcal F^\boxtimes$, and any $s,t\in\mathcal E\cap\mathcal F^\boxtimes$, we choose bases $\Theta^t_{is},\Theta^t_{sj}$ of $\mathcal V{t\choose i~s},\mathcal V{t\choose s~j}$ respectively.  Now fix $i,j\in\mathcal F^\boxtimes$, we also define $$\Theta^*_{i*}=\coprod_{s,t\in\mathcal E\cap\mathcal F^\boxtimes}\Theta^t_{is},\Theta^*_{*j}=\coprod_{s,t\in\mathcal E\cap\mathcal F^\boxtimes}\Theta^t_{sj}.$$ 

Choose distinct $z_i,z_j\in S^1$, and let $\arg z_j<\arg z_i<\arg z_j+2\pi$. For any $\alpha,\alpha'\in\Theta^*_{i*},\beta,\beta'\in\Theta^*_{j*}$, if either the source space of $\mathcal Y_\alpha$ does not equal the target space of $\mathcal Y_\beta$, or the target space of $\mathcal Y_{\alpha'}$ does not equal the source space of $\mathcal Y_{\beta'}$, or the target space of $\mathcal Y_\alpha$ does not equal the target space of $\mathcal Y_{\beta'}$, or the source space of $\mathcal Y_\beta$ does not equal the source space of $\mathcal Y_{\alpha'}$, then we set $(B_+)^{\beta'\alpha'}_{\alpha\beta}=0$; otherwise the values $(B_+)^{\beta'\alpha'}_{\alpha\beta}$ are determined by the following braid relation: for any $w^{(i)}\in W_i,w^{(j)}\in W_j$, 
\begin{align}
\mathcal Y_{\alpha}(w^{(i)},z_i)\mathcal Y_{\beta}(w^{(j)},z_j)=\sum_{\alpha'\in\Theta^*_{i*},\beta'\in\Theta^*_{j*}}(B_+)^{\beta'\alpha'}_{\alpha\beta}\mathcal Y_{\beta'}(w^{(j)},z_j)\mathcal Y_{\alpha'}(w^{(i)},z_i).\label{eq83}
\end{align}
The matrix $(B_+)_{ij}=\{ (B_+)^{\beta'\alpha'}_{\alpha\beta}\}_{\alpha\in\Theta^*_{i*},\beta\in\Theta^*_{j*}}^{\alpha'\in\Theta^*_{i*},\beta'\in\Theta^*_{j*}}$ is called a \textbf{braid matrix}. The matrix $(B_-)_{ij}=\{ (B_-)^{\beta'\alpha'}_{\alpha\beta}\}_{\alpha\in\Theta^*_{i*},\beta\in\Theta^*_{j*}}^{\alpha'\in\Theta^*_{i*},\beta'\in\Theta^*_{j*}}$ is defined in a similar way by assuming $\arg z_i<\arg z_j<\arg z_i+2\pi$. Clearly $(B_\pm)_{ij}$ is the inverse matrix of $(B_\mp)_{ji}$.

\begin{pp}\label{lb25}
For any $\alpha,\alpha'\in\Theta^*_{i*},\beta,\beta'\in\Theta^*_{j*}$, we have 
\begin{align}
\overline{(B_\pm)^{\beta'\alpha'}_{\alpha\beta}}=(B_\mp)^{\alpha'^*\beta'^*}_{\beta^*\alpha^*}.
\end{align}
\end{pp}
\begin{proof}
Choose distinct $z_i,z_j\in S^1$, and let $\arg z_j<\arg z_i<\arg z_j+2\pi$. Then for any $w^{(i)}\in W_i,w^{(j)}\in W_j$, the braid relation \eqref{eq83} holds. Taking the formal adjoint of \eqref{eq83}, we have
\begin{align}
\mathcal Y_{\beta}(w^{(j)},z_j)^\dagger\mathcal Y_{\alpha}(w^{(i)},z_i)^\dagger=\sum_{\alpha',\beta'}\overline{(B_+)^{\beta'\alpha'}_{\alpha\beta}}\mathcal Y_{\alpha'}(w^{(i)},z_i)^\dagger\mathcal Y_{\beta'}(w^{(j)},z_j)^\dagger.
\end{align}
By equation (1.34), for any $w^{(i)}\in W_i,w^{(j)}\in W_j$ we have
\begin{align}
\mathcal Y_{\beta^*}(\overline{w^{(j)}},z_j)\mathcal Y_{\alpha^*}(\overline {w^{(i)}},z_i)=\sum_{\alpha',\beta'}\overline{(B_+)^{\beta'\alpha'}_{\alpha\beta}}\mathcal Y_{\alpha'^*}(\overline {w^{(i)}},z_i)\mathcal Y_{\beta'^*}(\overline{w^{(j)}},z_j).\label{eq332}
\end{align}
But $\{(B_-)^{\alpha'^*\beta'^*}_{\beta^*\alpha^*} \}$ is also the braid matrix for the braid relation \eqref{eq332}. So we must have $\overline{(B_+)^{\beta'\alpha'}_{\alpha\beta}}=(B_-)^{\alpha'^*\beta'^*}_{\beta^*\alpha^*}$. If we let $\arg z_i<\arg z_j<\arg z_i+2\pi$, then we obtain $\overline{(B_-)^{\beta'\alpha'}_{\alpha\beta}}=(B_+)^{\alpha'^*\beta'^*}_{\beta^*\alpha^*}$.
\end{proof}

\begin{pp}\label{lb26}
If the bases $\Theta^*_{i*},\Theta^*_{j*}$ are
orthonormal under the inner product $\Lambda$, then for any $\alpha,\alpha'\in\Theta^*_{i*},\beta,\beta'\in\Theta^*_{j*}$, we have
\begin{align}
(B_\pm)^{\beta'\alpha'}_{\alpha\beta}=(B_\mp)^{\alpha^*\beta'}_{\beta\alpha'^*}=(B_\pm)^{\beta^*\alpha^*}_{\alpha'^*\beta'^*}.\label{eq26}
\end{align}
\end{pp}

\begin{proof}
Choose distinct $z_1,z_2,z_3,z_4\in S^1$ with arguments $\arg z_1<\arg z_2<\arg z_3<\arg z_4<\arg z_1+2\pi$. By relation \eqref{eqb11}, for any $k\in\mathcal E\cap\mathcal F^\boxtimes,w_0,w_5\in W_k, w_1,w_2\in W_i,w_3,w_4\in W_j$, we have, following convention 2.19,
\begin{align}
&\sum_{\substack{\alpha'\in\Theta^*_{i*} \\ \beta\in\Theta^*_{j*}}}\Big\langle \mathcal Y_{\beta^*}(\overline{w_4},z_4)\mathcal Y_{\beta}(w_3,z_3)\mathcal Y_{\alpha'^*}(\overline{w_2},z_2)\mathcal Y_{\alpha'}(w_1,z_1)w_0\Big|w_5 \Big\rangle\nonumber\\
=&\sum_{\substack{\alpha\in\Theta^*_{i*} \\ \beta\in\Theta^*_{j*}}}\Big\langle \mathcal Y_{\beta^*}(\overline{w_4},z_4)\mathcal Y_{\alpha^*}(\overline{w_2},z_2)\mathcal Y_{\alpha}(w_1,z_1)\mathcal Y_{\beta}(w_3,z_3)w_0\Big|w_5 \Big\rangle.\label{eq333}
\end{align}
By exchanging $\mathcal Y_\alpha$ and $\mathcal Y_\beta$,  \eqref{eq333} equals

\begin{align}
\sum_{\substack{\alpha,\alpha'\in\Theta^*_{i*} \\ \beta,\beta'\in\Theta^*_{j*} }}(B_-)^{\beta'\alpha'}_{\alpha\beta}\Big\langle \mathcal Y_{\beta^*}(\overline{w_4},z_4)\mathcal Y_{\alpha^*}(\overline{w_2},z_2)\mathcal Y_{\beta'}(w_3,z_3)\mathcal Y_{\alpha'}(w_1,z_1)w_0\Big|w_5  \Big\rangle.
\end{align}
By proposition 2.3,  we have
\begin{align}
\mathcal Y_{\beta}(w_3,z_3)\mathcal Y_{\alpha'^*}(\overline{w_2},z_2)=\sum_{\substack{\alpha,\alpha'\in\Theta^*_{i*} \\ \beta,\beta'\in\Theta^*_{j*} }}(B_-)^{\beta'\alpha'}_{\alpha\beta}\mathcal Y_{\alpha^*}(\overline{w_2},z_2)\mathcal Y_{\beta'}(w_3,z_3).
\end{align}
This proves that $(B_+)^{\alpha^*\beta'}_{\beta\alpha'^*}=(B_-)^{\beta'\alpha'}_{\alpha\beta}$.

Similarly, we also have
\begin{align}
&\sum_{\substack{\alpha'\in\Theta^*_{i*} \\ \beta\in\Theta^*_{j*}}}\Big\langle \mathcal Y_{\beta^*}(\overline{w_4},z_4)\mathcal Y_{\beta}(w_3,z_3)\mathcal Y_{\alpha'^*}(\overline{w_2},z_2)\mathcal Y_{\alpha'}(w_1,z_1)w_0\Big|w_5 \Big\rangle\nonumber\\
=&\sum_{\substack{\alpha'\in\Theta^*_{i*} \\ \beta'\in\Theta^*_{j*}}}\Big\langle\mathcal Y_{\alpha'^*}(\overline{w_2},z_2) \mathcal Y_{\beta'^*}(\overline{w_4},z_4)\mathcal Y_{\beta'}(w_3,z_3)\mathcal Y_{\alpha'}(w_1,z_1)w_0\Big|w_5 \Big\rangle\\
=&\sum_{\substack{\alpha,\alpha' \in\Theta^*_{i*} \\ \beta,\beta'\in\Theta^*_{j*}}}(B_-)^{\beta^*\alpha^*}_{\alpha'^*\beta'^*}\Big\langle \mathcal Y_{\beta^*}(\overline{w_4},z_4)\mathcal Y_{\alpha^*}(\overline{w_2},z_2)\mathcal Y_{\beta'}(w_3,z_3)\mathcal Y_{\alpha'}(w_1,z_1)w_0\Big|w_5  \Big\rangle,
\end{align}
which implies that $(B_+)^{\alpha^*\beta'}_{\beta\alpha'^*}=(B_-)^{\beta^*\alpha^*}_{\alpha'^*\beta'^*}$.

If $z_1,z_2,z_3,z_4\in S^1$ and their arguments are chosen such that $\arg z_4<\arg z_3<\arg z_2<\arg z_1<\arg z_4+2\pi$, then the same argument implies that $(B_+)^{\beta'\alpha'}_{\alpha\beta}=(B_-)^{\alpha^*\beta'}_{\beta\alpha'^*}=(B_+)^{\beta^*\alpha^*}_{\alpha'^*\beta'^*}$. 
\end{proof}

\begin{co}\label{lbb3}
If the bases $\Theta^*_{i*},\Theta^*_{j*}$ are
orthonormal under the inner product $\Lambda$, then the braid matrix $(B_\pm)_{ij}$ is unitary.
\end{co}

\begin{proof}
If we apply propositions \ref{lb25} and \ref{lb26}, then for any $\alpha,\alpha'\in\Theta^*_{i*},\beta,\beta'\in\Theta^*_{j*}$, we have

\begin{align}
(B_\pm)^{\beta'\alpha'}_{\alpha\beta}=(B_\pm)^{\beta^*\alpha^*}_{\alpha'^*\beta'^*}=\overline{(B_\mp)^{\alpha\beta}_{\beta'\alpha'}},
\end{align}
which shows that $(B_\pm)_{ij}$ is the adjoint  of $(B_\mp)_{ji}$. But we  know that $(B_\pm)_{ij}$ is also the inverse matrix of $(B_\mp)_{ji}$. So $(B_\pm)_{ij}$ is unitary.
\end{proof}

\subsection{Unitarity of fusion matrices}

Recall from section 2.4 that for any $W_i,W_j,W_k,W_t$ in $\mathcal F^\boxtimes$, we have a fusion matrix $\{F^{\beta'\alpha'}_{\alpha\beta} \}_{\alpha\in\Theta^t_{i*},\beta\in\Theta^*_{jk}}^{\alpha'\in\Theta^*_{ij},\beta'\in\Theta^t_{*k}}$ defined by the fusion relation
\begin{align}
\mathcal Y_\alpha(w^{(i)},z_i)\mathcal Y_\beta(w^{(j)},z_j)=\sum_{\alpha'\in\Theta^*_{ij},\beta'\in\Theta^t_{*k}}F^{\beta'\alpha'}_{\alpha\beta}\mathcal Y_{\beta'}(\mathcal Y_{\alpha'}(w^{(i)},z_i-z_j)w^{(j)},z_j),
\end{align}
where $z_i,z_j\in\mathbb C^\times,0<|z_i-z_j|<|z_j|<|z_i|$,  $\arg z_j$ is close to $\arg z_i$ as $z_j\rightarrow z_i$, and $\arg(z_i-z_j)$ is close to $\arg z_i$ as $z_j\rightarrow 0$. We let $F^{\beta'\alpha'}_{\alpha\beta}=0$ if the source space of $\mathcal Y_\alpha$ does not equal the target space of $\mathcal Y_\beta$, or if the target space of $\mathcal Y_{\alpha'}$ does not equal the charge space of $\mathcal Y_{\beta'}$. In this section, we show that fusion matrices are unitary.

\begin{pp}\label{lb113}
Choose unitary $V$-modules $W_i,W_k$ in $\mathcal F^\boxtimes$, $W_j,W_t$ in $\mathcal E\cap\mathcal F^\boxtimes$. Then for any for any $\alpha\in\Theta^t_{i*},\beta\in\Theta^*_{jk},\alpha'\in\Theta^*_{ij},\beta'\in\Theta^t_{*k}$, we have
\begin{align}
F^{\beta'\alpha'}_{\alpha\beta}=(B_+)^{B_+\beta',\alpha'}_{\alpha,B_+\beta}=(B_-)^{B_-\beta',\alpha'}_{\alpha,B_-\beta}.
\end{align}
\end{pp}

\begin{proof}
Choose distinct $z_i,z_j,z_k\in S^1$ with arguments $\arg z_k<\arg z_j<\arg z_i<\arg z_k+2\pi$, and assume that $0<|z_i-z_j|<1$. Choose $w^{(i)}\in W_i,w^{(j)}\in W_j,w^{(k)}\in W_k$. By corollary 2.18,  we have
\begin{align}
&\mathcal Y_\alpha(w^{(i)},z_i)\mathcal Y_\beta(w^{(j)},z_j)\mathcal Y^k_{k0}(w^{(k)},z_k)\nonumber\\
=&\mathcal Y_\alpha(w^{(i)},z_i)\mathcal Y_{B_+\beta}(w^{(k)},z_k)\mathcal Y^j_{j0}(w^{(j)},z_j)\nonumber\\
=&\sum_{\substack{\alpha'\in\Theta^*_{i*}\\\beta'\in\Theta^*_{*k}}}(B_+)^{B_+\beta',\alpha'}_{\alpha,B_+\beta}\mathcal Y_{B_+\beta'}(w^{(k)},z_k)\mathcal Y_{\alpha'}(w^{(i)},z_i)\mathcal Y^j_{j0}(w^{(j)},z_j).\label{eq335}
\end{align}
On the other hand, by corollary 2.18 and theorem \ref{lb16},
\begin{align}
&\mathcal Y_\alpha(w^{(i)},z_i)\mathcal Y_\beta(w^{(j)},z_j)\mathcal Y^k_{k0}(w^{(k)},z_k)\nonumber\\
=&\sum_{s\in\mathcal E}\sum_{\substack{\alpha'\in\Theta^s_{ij}\\\beta'\in\Theta^t_{sk}}}F_{\alpha\beta}^{\beta'\alpha'}\mathcal Y_{\beta'}\big(\mathcal Y_{\alpha'}(w^{(i)},z_i-z_j)w^{(j)},z_j \big)\mathcal Y^k_{k0}(w^{(k)},z_k)\label{eqb7}\\
=&\sum_{s\in\mathcal E}\sum_{\substack{\alpha'\in\Theta^s_{ij}\\\beta'\in\Theta^t_{sk}}}F_{\alpha\beta}^{\beta'\alpha'}\mathcal Y_{B_+\beta'}(w^{(k)},z_k)\mathcal Y^s_{s0}\big(\mathcal Y_{\alpha'}(w^{(i)},z_i-z_j)w^{(j)},z_j \big),\label{eq334}
\end{align}
where \eqref{eqb7} and \eqref{eq334} are understood as products of two generalized intertwining operators (see the beginning of chapter \ref{lb112}).  By proposition 2.17,  \eqref{eq334} equals
\begin{align}
\sum_{s\in\mathcal E}\sum_{\substack{\alpha'\in\Theta^s_{ij}\\\beta'\in\Theta^t_{sk}}}F_{\alpha\beta}^{\beta'\alpha'}\mathcal Y_{B_+\beta'}(w^{(k)},z_k)\mathcal Y_{\alpha'}(w^{(i)},z_i)\mathcal Y^j_{j0}(w^{(j)},z_j).
\end{align}
Comparing this result with \eqref{eq335}, we see immediately that $F_{\alpha\beta}^{\beta'\alpha'}=(B_+)^{B_+\beta',\alpha'}_{\alpha,B_+\beta}$. If we assume at the beginning that $\arg z_i<\arg z_j<\arg z_k<\arg z_i+2\pi$, then we obtain $F_{\alpha\beta}^{\beta'\alpha'}=(B_-)^{B_-\beta',\alpha'}_{\alpha,B_-\beta}$.
\end{proof}

\begin{pp}\label{lb34}
Let $W_i,W_j$ be unitary $V$-modules in $\mathcal F^\boxtimes$. For each $k\in\mathcal E\cap\mathcal F^\boxtimes$, we let $\{\mathcal Y_{\alpha}:\alpha\in\Theta^k_{ij}\}$ be a set of orthonormal basis of $\mathcal V{k\choose i~j}$ under the inner product $\Lambda$. Then $B_+\Theta^k_{ij}=\{\mathcal Y_{B_+\alpha}:\alpha\in\Theta^k_{ij}\}$ and $B_-\Theta^k_{ij}=\{\mathcal Y_{B_-\alpha}:\alpha\in\Theta^k_{ij}\}$ are orthonormal bases of $\mathcal V{k\choose j~i}$.
\end{pp}

\begin{proof}
Choose distinct $z_i,z_j\in S^1$ with arguments satisfying $\arg z_i<\arg z_j<\arg z_i+2\pi$. By proposition \ref{lb111}-(3), for any $w^{(i)}\in W_i,w^{(j)}\in W_j$, we have
\begin{align}
\mathcal Y^j_{j0}(w^{(j)},z_j)\mathcal Y^0_{\overline ii}(\overline{w^{(i)}},z_i)=\sum_{\alpha\in\Theta^*_{ij}}\mathcal Y_{\alpha^*}(\overline{w^{(i)}},z_i)\mathcal Y_{B_-\alpha}(w^{(j)},z_j).
\end{align}
Take the formal adjoint of both sides, we obtain
\begin{align}
\mathcal Y^0_{\overline ii}(\overline{w^{(i)}},z_i)^\dagger\mathcal Y^j_{j0}(w^{(j)},z_j)^\dagger=\sum_{\alpha\in\Theta^*_{ij}}\mathcal Y_{B_-\alpha}(w^{(j)},z_j)^\dagger\mathcal Y_{\alpha^*}(\overline{w^{(i)}},z_i)^\dagger.\label{eq336}
\end{align}
Recall that $(\mathcal Y^j_{j0})^\dagger=\mathcal Y^0_{\overline jj}$ and $(\mathcal Y^0_{\overline ii})^\dagger=\mathcal Y^i_{i0}$. Thus, by equation (1.34),  equation \eqref{eq336} shows that
\begin{align}
&\mathcal Y^i_{i0}(w^{(i)},z_i)\mathcal Y^0_{\overline jj}(\overline{w^{(j)}},z_j)=\sum_{\alpha\in\Theta^*_{ij}}\mathcal Y_{(B_-\alpha)^*}(\overline{w^{(j)}},z_j)\mathcal Y_{\alpha}(w^{(i)},z_i)\nonumber\\
=&\sum_{\beta\in B_-\Theta^*_{ij}}\mathcal Y_{\beta^*}(\overline{w^{(j)}},z_j)\mathcal Y_{B_+\beta}(w^{(i)},z_i),
\end{align}
which, by proposition \ref{lb111}-(3), shows that $B_-\Theta^k_{ij}$ is an orthonormal basis of $\mathcal V{k\choose i~j}$ for any $k\in\mathcal E$. The other case is treated in a similar way.
\end{proof}

\begin{co}\label{lb37}
For any $W_i,W_j,W_k$ in $\mathcal F^\boxtimes$ and $W_t$ in $\mathcal E$, the fusion matrix $\{F^{\beta'\alpha'}_{\alpha\beta} \}_{\alpha\in\Theta^t_{i*},\beta\in\Theta^*_{jk}}^{\alpha'\in\Theta^*_{ij},\beta'\in\Theta^t_{*k}}$ is unitary.
\end{co}
\begin{proof}
If $W_j$ is irreducible, then  $W_j$ is unitarily equivalent to a unitary $V$-module in $\mathcal E\cap\mathcal F^\boxtimes$. The unitarity of the fusion matrix follows then from propositions \ref{lb113}, \ref{lb34}, and the unitarity of braid matrices proved in the last section. In general,  the fusion matrix is diagonalized according to the orthogonal decomposition of $W_j$ into irreducible submodules. Thus the unitarity can be proved easily.
\end{proof}

\subsection{Unitarity of the ribbon fusion categories}

In this section, we prove that $\Rep^\uni_{\mathcal F^\boxtimes}(V)$ is unitary when the unitary structure on $\mathcal F^\boxtimes$ is defined by $\Lambda$. By corollary \ref{lb37}, the associators are unitary. By proposition \ref{lb34}, the braid operators are unitary. That $\lambda_i:V\boxtimes W_i\rightarrow W_i$ and $\rho_i:W_i\boxtimes V\rightarrow W_i$ are unitary follows from equations \eqref{eq323} and \eqref{eq324}.

Choose $W_{i_1},W_{i_2},W_{j_1},W_{j_2}$ in $\mathcal F^\boxtimes$. We show, for any $F\in\Hom_V(W_{i_1},W_{i_2}),G\in\Hom_V(W_{j_1},W_{j_2})$, that
\begin{align}
(F\otimes G)^*=F^*\otimes G^*.\label{eq337}
\end{align}
Consider  direct sum modules $W_i=W_{i_1}\oplus^\perp W_{i_2},W_j=W_{j_1}\oplus^\perp W_{j_2}$. For each $k\in\mathcal E$, it is easy to see that $\mathcal V{k\choose i~j}$ has the natural orthogonal decomposition
\begin{align}
\mathcal V{k\choose i~j}=\bigoplus^\perp_{a,b=1,2}\mathcal V{k\choose i_a~j_b},
\end{align}
which induces the natural decomposition
\begin{align}
W_i\boxtimes W_j=\bigoplus^\perp_{a,b=1,2}W_{i_a}\boxtimes W_{j_b}.
\end{align}
Therefore, if we regard $F,G$  as endomorphisms of the modules $W_i,W_j$ respectively, then  $F\otimes G$ and $F^*\otimes G^*$ can be regarded as endomorphisms of $W_i\boxtimes W_j$.  Thus, it suffices to prove equation \eqref{eq337} for any $F\in\End_V(W_i),G\in\End_V(W_j)$.

Since $\End_V(W_i)$ and $\End_V(W_j)$ are $C^*$-algebras (see theorem 2.21), they are spanned by unitary elements inside them. Therefore, by linearity, it suffices to prove \eqref{eq337} when  $F\in\End_V(W_i),G\in\End_V(W_j)$ are unitary operators. By equation (2.56), it is easy to see that  $F\otimes G$ is unitary. Hence we have
\begin{align}
(F^*\otimes G^*)(F\otimes G)=F^*F\otimes G^*G=\id_i\otimes\id_j=\id_{ij},
\end{align}
which implies that $F^*\otimes G^*=(F\otimes G)^{-1}=(F\otimes G)^*$. This proves relation \eqref{eq337}.\\

For each $W_i$ in $\mathcal F^\boxtimes$, the twist  $\vartheta_i=e^{2i\pi L_0}$ is clearly unitary. Hence, in order to prove the unitarity of $\Rep^\uni_{\mathcal F^\boxtimes}(V)$, it remains to find $\ev_i,\coev_i$, such that equations (2.69) and (2.70) hold. 

To prove this, we let $\ev_{i,\overline i}\in\Hom_V(W_i\boxtimes W_{\overline i},V)$ be the homomorphism represented by the intertwining operator $\mathcal Y^0_{i\overline i}$, and let $\coev_{i,\overline i}=\ev_{i\overline i}^*$. Since $i$ and $\overline{\overline i}$ are identified, we can define $\ev_{\overline i,i}$ and $\coev_{\overline i,i}$  in a similar way. Set $\ev_i=\ev_{\overline i, i},\coev_i=\coev_{i,\overline i}$. If we can verify, for all $W_i$ in $\mathcal F^\boxtimes$, the following relations:
\begin{gather}
(\id_i\otimes\ev_{\overline i,i})\circ(\coev_{i,\overline i}\otimes \id_i)=\id_i,\label{eq147}\\
(\ev_{i,\overline i}\otimes \id_i)\circ(\id_i\otimes\coev_{\overline i, i})=\id_i,\label{eq148}\\
\ev_{i,\overline i}=\ev_{\overline i, i}\circ \sigma_{i,\overline i}\circ(\vartheta_i\otimes \id_{\overline i}),\label{eq149}\\
\coev_{i,\overline i}=(\id_i\otimes\vartheta_{\overline i}^{-1})\circ \sigma_{i,\overline i}^{-1}\circ\coev_{\overline i,i},\label{eq150}
\end{gather}
then equations (2.64), (2.65), (2.69), and (2.70) are true for all $W_i$, and our modular tensor category is unitary.

To begin with, we define the positive number $d_i$ to be the norm square of the vector $\mathcal Y^0_{i\overline i}$ inside $\mathcal V{0\choose i~\overline i}$, i.e.,
\begin{gather}
d_i=\|\mathcal Y^0_{i\overline i}\|^2.
\end{gather}
By propositions 1.14 and \ref{lb34},  $d_{\overline i}=d_i$. The following property will indicate that $d_i$ is the quantum dimension of $W_i$.
\begin{pp}\label{lb115}
\begin{equation}
\ev_{i,\overline i}\circ\coev_{i,\overline i}=d_i.
\end{equation}
\end{pp}

\begin{proof}
First we assume that $W_i$ is irreducible. Then  $\{\mathcal Y^0_{i\overline i}\}$ is a basis of $\mathcal V{0\choose i~\overline i}$. Let $\{\widecheck{\mathcal Y}^0_{i\overline i}\}$ be its dual basis. Then $\widecheck{\mathcal Y}^{\alpha}=d_i^{\frac 1 2}\widecheck{\mathcal Y}^0_{i\overline i}$ has unit length. Now, for any $v\in V$, $\ev_{i,\overline i}$ maps $\widecheck{\mathcal Y}^{\alpha}\otimes v\in W_i\boxtimes W_{\overline i}$ to $\langle\widecheck{\mathcal Y}^{\alpha},\mathcal Y^0_{i\overline i}\rangle v=d_i^{\frac 1 2}\langle\widecheck{\mathcal Y}^0_{i\overline i},\mathcal Y^0_{i\overline i}\rangle v=d_i^{\frac 1 2}v$. It follows that its adjoint $\coev_{i,\overline i}$ maps each $v\in V$ to $d_i^{\frac 1 2}\widecheck{\mathcal Y}^{\alpha}\otimes v$. Hence $\ev_{i,\overline i}\circ\coev_{i,\overline i}(v)=d_iv$.

In general, $W_i$ has decomposition $W_i=\bigoplus^\perp_{a}W_{i_a}$, where each $W_{i_a}$ is irreducible.  Let $p_a$ be the projection of $W_i$ on $W_{i_a}$. Then the projection $\overline {p_a} $ of $W_{\overline i}$ on $W_{\overline{i_a}}$ satisfies $\overline {p_a}\overline{w^{(i)}}=\overline {p_aw^{(i)}}$ ($w^{(i)}\in W_i$). It is easy to check that 
\begin{gather}
\ev_{i,\overline i}=\sum_a\ev_{i,\overline i}\circ(p_a\otimes \overline{p_a})=\sum_a\ev_{i_a,\overline{i_a}},\label{eq338}\\
\coev_{i,\overline i}=\sum_a (p_a\otimes \overline{p_a})\circ\coev_{i,\overline i}=\sum_a\coev_{i_a,\overline{i_a}},\label{eq339}
\end{gather}
and $d_i=\sum_a d_{i_a}$. The general case can be proved using these relations.
\end{proof}

Now we are ready to prove equations \eqref{eq147}-\eqref{eq150}.
\begin{proof}[Proof of equation \eqref{eq148}]
By equations \eqref{eq338} and \eqref{eq339}, it suffices to prove \eqref{eq148} when $W_i$ is irreducible. Choose $w^{(i)}_1,w^{(i)}_2\in W_i$. Choose $z_1,z_2\in\mathbb C^\times$ satisfying $0<|z_2-z_1|<|z_1|<|z_2|$. Choose $\arg z_2$, let $\arg z_1$ be close to $\arg z_2$ as $z_2-z_1\rightarrow 0$, and let $\arg (z_2-z_1)$ be close to $\arg z_2$ as $z_1\rightarrow 0$. Since  $\lVert \widecheck{\mathcal Y}^0_{\overline ii} \lVert^2=d_i^{-1}$, by transport formula we have
\begin{align}
&Y_i\big(\mathcal Y^0_{i\overline i}(w^{(i)}_2,z_2-z_1)\overline {w^{(i)}_1},z_1\big)\nonumber\\
=&d_i^{-1}(\mathcal Y^0_{\overline i i})^\dagger(w^{(i)}_2,z_2)\mathcal Y^0_{\overline i i}(\overline {w^{(i)}_1},z_1)+\mathcal Y_\gamma(w^{(i)}_2,z_2)\mathcal Y_\beta(\overline {w^{(i)}_1},z_1)\nonumber\\
=&d_i^{-1}\mathcal Y^i_{i0}(w^{(i)}_2,z_2)\mathcal Y^0_{\overline ii}(\overline {w^{(i)}_1},z_1)+\mathcal Y_\gamma(w^{(i)}_2,z_2)\mathcal Y_\beta(\overline {w^{(i)}_1},z_1)\label{eq340}
\end{align}
where $\mathcal Y_\beta,\mathcal Y_\gamma$ are a chain of intertwining operators, and the target space of $\mathcal Y_\beta$ does not contain any submodule equivalent to the vacuum module $V$. Equation \eqref{eq340} is equivalent to the relation
\begin{gather}
(\ev_{i,\overline i}\otimes \id_i)=d_i^{-1}(\id_i\otimes \ev_{\overline i,i})+\mathcal Y_\gamma\circ(\id_i\otimes\mathcal Y_\beta),
\end{gather}
where $\mathcal Y_{\gamma}$ and $\mathcal Y_{\beta}$ denote the corresponding morphisms. By proposition \ref{lb115},
\begin{align}
&(\ev_{i,\overline i}\otimes \id_i)\circ(\id_i\otimes\coev_{\overline i, i})\nonumber\\
=&d_i^{-1}(\id_i\otimes \ev_{\overline i,i})\circ(\id_i\otimes\coev_{\overline i,i})+\mathcal Y_\gamma\circ(\id_i\otimes\mathcal Y_\beta)\circ(\id_i\otimes\coev_{\overline i,i})\nonumber\\
=&\id_i+\mathcal Y_\gamma\circ\big(\id_i\otimes(\mathcal Y_\beta\circ\coev_{\overline i,i})\big).\label{eq154}
\end{align}
Since $\mathcal Y_\beta\circ\coev_{\overline i, i}$ is a morphism from the vacuum module $V$ to a $V$-module with no irreducible submodule equivalent to  $V$, $\mathcal Y_\beta\circ\coev_{\overline i, i}$ must be zero. So \eqref{eq154} equals $\id_i$, and equation \eqref{eq148} is proved.
\end{proof}

\begin{proof}[Proof of equations \eqref{eq147}, \eqref{eq149}, and \eqref{eq150}]
Take the adjoint of equation \eqref{eq148}, we immediately obtain equation \eqref{eq147}. Equation \eqref{eq149} follows from equation (1.41). Equation (1.42) indicates that
\begin{gather}
\ev_{i,\overline i}=\ev_{\overline i, i}\circ \sigma_{i,\overline i}\circ(\id_i\otimes \vartheta_{\overline i}),
\end{gather}
the adjoint of which is \eqref{eq150}. 
\end{proof}
Thus we've proved the unitarity of our ribbon fusion category. 

\begin{thm}\label{lb117}
Let $V$ be unitary, energy bounded, and strongly local, and let $\mathcal F$ be a non-empty set of non-zero irreducible unitary $V$-modules satisfying condition \ref{CondA} or \ref{CondB} in section \ref{Condition ABC}. If we define a unitary structure on $\mathcal F^\boxtimes$ using $\Lambda$, then the ribbon fusion category $\Rep^\uni_{\mathcal F^\boxtimes}(V)$ is unitary.
\end{thm}

Note that our proof of the unitarity of the tensor categories  uses only the positive definiteness of  $\Lambda$. Thus our results in this chapter can also be summarized in the following way.

\begin{thm}
Let $V$ be unitary. If all irreducible $V$-modules are unitarizable, and the sesquilinear form $\Lambda$ defined on any tensor product of unitary $V$-modules is positive (definite), then $\Rep^\uni(V)$ is a unitary modular tensor category.
\end{thm}

\section{Epilogue}

\subsection{Application to unitary Virasoro VOAs $(c<1)$}
Let $\Vir=\Span_{\mathbb C}\{C,L_n:n\in\mathbb Z\}$ be the Virasoro Lie algebra satisfying the relation
\begin{gather*}
[L_m,L_n]=(m-n)L_{m+n}+\frac 1 {12}(m^3-m)\delta_{m,-n}C\qquad(m,n\in\mathbb Z),\\
[C,L_n]=0\qquad(n\in\mathbb Z).
\end{gather*}
If $W$ is a $\Vir$-module, and the vector space $W$ is equipped with an inner product $\langle\cdot|\cdot\rangle$, we say that $W$ is a unitary $\Vir$-module, if $L_n^\dagger=L_{-n}$ holds for any $n\in\mathbb Z$. More precisely, this means that for any $w_1,w_2\in W$, we have
\begin{align}
\langle L_nw_1|w_2 \rangle=\langle w_1|L_{-n}w_2\rangle.
\end{align}

Choose Lie subalgebras $\Vir_+=\Span_{\mathbb C}\{L_n:n\in\mathbb Z_{>0}\}$ and $\Vir_-=\Span_{\mathbb C}\{L_n:n\in\mathbb Z_{<0}\}$ of $\Vir$, and let  $U(\Vir)$ be the universal enveloping algebra of $\Vir$. For each $c,h\in\mathbb C$, the Verma module $M(c,h)$ for $\Vir$ is  the free $U(\Vir_-)$-module generated by a distinguished vector (the highest weight vector) $v_{c,h}$, subject to the relation
\begin{gather}
U(\Vir_+)v_{c,h}=0,\qquad
Cv_{c,h}=cv_{c,h},\qquad
L_0v_{c,h}=hv_{c,h}.
\end{gather}
Then there exists a unique maximal proper submodule $J(c,h)$ of $M(c,h)$. We let $L(c,h)=M(c,h)/J(c,h)$. It was proved in \cite{FQS} and \cite{GKO} that when $0\leq c<1$, the $\Vir$-module $L(c,h)$ is unitarizable if and only if there exist $m,r,s\in\mathbb Z$ satisfying $2\leq m,1\leq r\leq m-1,1\leq s\leq m$, such that
\begin{gather}
c=1-\frac 6{m(m+1)},\label{eq343}\\ h=h_{r,s}=\frac{\big((m+1)r-ms\big)^2-1}{4m(m+1)}.\label{eq344}
\end{gather}
For such a module $L(c,h)$, we fix a unitary structure such that $\langle v_{c,h}|v_{c,h}\rangle$=1.

Let $\Omega=v_{c,0},\nu=L_{-2}\Omega$. Then there exists a unique VOA structure on $L(c,0)$, such that $\Omega$ is the vacuum vector, and $Y(\nu,x)=\sum_{n\in\mathbb Z}L_nx^{-n-2}$ (cf.\cite{Frenkel Zhu}).   Let $E=\{\Omega, \nu\}$, then $E$ is a set of quasi-primary vectors generating $L(c,0)$. 

We now assume that $c$ satisfies relation \eqref{eq343}. Then by \cite{DL} theorem 4.2 or \cite{CKLW} proposition 5.17,  $L(c,0)$ is a unitary VOA. The PCT operator $\theta$ is determined by the fact that $\theta$ fixes vectors in $E$.  $L(c,0)$ satisfies conditions $(\alpha),(\beta)$, and $(\gamma)$ in the introduction. (See the introduction of \cite{H Rigidity}, and the reference therein.)

Since $Y(\nu,n)=L_{n-1}$, representations of $L(c,0)$ are determined by their restrictions to $\Vir$. By \cite{Virasoro} theorem 4.2,  irreducible representations of $L(c,0)$ are precisely those that can be restricted to  irreducible $\Vir$-modules of the form $L
(c,h_{r,s})$, where the highest weight $h_{r,s}$ satisfies relation \eqref{eq344}. By proposition 1.10, $L(c,h_{r,s})$ is a unitary $L(c,0)$-module. It follows that any $L(c,0)$-module is unitarizable. Clearly the conformal dimension of $L(c,h_{r,s})$ is $h_{r,s}$. 

Let $\mathcal F=\{L(c,h_{1,2}),L(c,h_{2,2}) \}$. The fusion rules of $L(c,0)$ (see \cite{Virasoro} theorem 4.3) indicate that $\mathcal F$ is \textbf{generating}, i.e., any unitary $L(c,0)$-module is in $\mathcal F^\boxtimes$. We check that  $\mathcal F$ satisfies condition \ref{CondA} in section \ref{Condition ABC}:

Condition \ref{CondA}-(a): Since we know that any $L(c,0)$-module is unitarizable, condition \ref{CondA}-(a) is obvious.

Condition \ref{CondA}-(b): Since $E\subset E^1(L(c,0))$, $E^1(L(c,0))$ is generating.

Condition \ref{CondA}-(c): If $\mathcal Y_\alpha\in\mathcal V{k\choose i~j}$ is unitary and irreducible (hence $W_i,W_j,W_k$ restrict to irreducible highest weight $\Vir$-modules), we choose a non-zero highest weight vector $v^{(i)}\in W_i$. We then define a linear map
\begin{gather*}
\phi_\alpha:W_j\rightarrow W_k\{x\},\\
w^{(j)}\mapsto \phi_\alpha(x)w^{(j)}=\mathcal Y_\alpha(v^{(i)},x)w^{(j)}.
\end{gather*}
Then $\phi_\alpha$ is a primary field in the sense of \cite{Loke} chapter II. By \cite{Loke} proposition IV.1.3, if $W_i\in\mathcal F$, then $\phi_\alpha$ satisfies $0$-th order energy bounds. This proves condition \ref{CondA}-(c). Theorem \ref{lb117} now implies the following:

\begin{thm}
Let $c=1-\frac 6{m(m+1)}$ where $m=2,3,4,\dots$, and let $L(c,0)$ be the unitary Virasoro VOA with central charge $c$. Then any $L(c,0)$-module is unitarizable,  $\Lambda$ is positive definite on the tensor product of any two $L(c,0)$-modules, and the modular tensor category $\Rep^\uni(L(c,0))$ of the unitary representations of $L(c,0)$ is unitary. 
\end{thm}

\subsection{Application to unitary affine VOAs}

Let $\mathfrak g$ be a complex simple Lie algebra. Let $\mathfrak h$ be a Cartan subalgebra of $\mathfrak g$, $\lambda\in\mathfrak h^*$ be a highest weight, and let $L(\lambda)$ be the irreducible highest weight module of $\mathfrak g$ with highest weight $\lambda$ and a distinguished highest (non-zero) vector $v_\lambda\in L(\lambda)$. 

Choose the normalized invariant bilinear form $(\cdot,\cdot)$ satisfying $(\theta,\theta)=2$, where $\theta$ is the highest root of $\mathfrak g$. Let $\widehat{\mathfrak g}=\Span_{\mathbb C}\{K,X(n):X\in\mathfrak g,n\in\mathbb Z \}$ be the affine Lie algebra satisfying
\begin{gather*}
[X(m),Y(n)]=[X,Y](m+n)+m(X,Y)\delta_{m,-n}K\qquad (X,Y\in\mathfrak g,m,n\in\mathbb Z),\\
[K,X(n)]=0\qquad (X\in\mathfrak g,n\in\mathbb Z).
\end{gather*}
Let $\mathfrak g_{\mathbb R}$ be a compact real form of $\mathfrak g$. Then $\mathfrak g=\mathfrak g_{\mathbb R}\oplus_{\mathbb R}i\mathfrak g_{\mathbb R}$. If $W$ is a $\widehat{\mathfrak g}$-module, and the vector space $W$ is equipped with an inner product $\langle \cdot |\cdot\rangle$, we say that $W$ is a unitary $\widehat{\mathfrak g}$-module, if for any $X\in\mathfrak g_{\mathbb R}$ and $n\in\mathbb Z$, we have
\begin{align}
X(n)^\dagger=-X(-n),\qquad K^\dagger=K.
\end{align}

Let $U(\widehat {\mathfrak g})$ be the universal enveloping algebra of $\widehat{\mathfrak g}$. Choose Lie subalgebras $\widehat{\mathfrak g}_+=\Span_{\mathbb C}\{X(n):X\in\mathfrak g,n>0 \},\widehat{\mathfrak g}_-=\Span_{\mathbb C}\{X(n):X\in\mathfrak g,n<0 \}$ of $\widehat{\mathfrak g}$. We regard $\mathfrak g$ as a Lie subalgebra of $\widehat{\mathfrak g}$ by identifying $X\in\mathfrak g$ with $X(0)\in\widehat{\mathfrak g}$. For any $k\in\mathbb C$, highest weight $\lambda\in\mathfrak h^*$, the Verma module $M(k,\lambda)$ for $\widehat{\mathfrak g}$ is the free $U(\widehat{\mathfrak g}_-)$-module generated by $L(\lambda)$  and  subject to the conditions
\begin{gather}
U(\widehat{\mathfrak g}_+)L(\lambda)=0,\qquad K|_{L(\lambda)}=k\cdot\id|_{L(\lambda)}.
\end{gather}
We let $M(k,\lambda)$ be graded by $\mathbb Z_{\geq0}$: For any $X_1,\dots,X_m\in\mathfrak g,n_1,\dots,n_m>0,v\in L(\lambda)$, 
the weight of $X_1(-n_1)\cdots X_m(-n_m)v$ equals $n_1+\cdots+n_m$.
There exits a unique maximal proper graded submodule $J(k,\lambda)$ of $M(k,\lambda)$. We let $L(k,\lambda)=M(k,\lambda)/J(k,\lambda)$. Then by \cite{Kac infinite} theorem 11.7, the $\widehat{\mathfrak g}$-module $L(k,\lambda)$ is unitarizable if and only if 
\begin{gather}
k=0,1,2,\dots,\\
\lambda\text{ is a dominant integral weight of $\mathfrak g$, and }(\lambda,\theta)\leq k.\label{eq345}
\end{gather}
For such a $\widehat{\mathfrak g}$-module $L(k,\lambda)$, we fix a unitary structure.

Let $h^\vee$ be the dual Coxeter number of $\mathfrak g$. Let $\Omega$ be a highest weight vector of $L(k,0)$. It was proved in \cite{Frenkel Zhu} that when $k\neq-h^\vee$, there exists a unique VOA structure on $L(k,0)$, such that $\Omega$ is the vacuum vector, that for any $X\in\mathfrak g$ we have
\begin{align}
Y\big(X(-1)\Omega,x \big)=\sum_{n\in\mathbb Z}X(n)x^{-n-1},
\end{align}
and that the conformal vector $\nu$ is defined by
\begin{align}
\nu=\frac 1{2(k+h^\vee)}\sum_{i=1}^{\dim\mathfrak g}X_i(-1)^2\Omega,
\end{align}
where $\{X_i \}$ is an orthonormal basis of $i\mathfrak g_{\mathbb R}$ under the inner product $(\cdot,\cdot)$. The set $E=\{\Omega, X(-1)\Omega: X\in\mathfrak g_{\mathbb R}\}$ generates $L(k,0)$. Writing  the operator $L_1=Y(\nu,2)$ in terms of $X(n)$'s using Jacobi identity, one can show that the vectors in $E$ are quasi-primary. 

We now assume that $k\in\mathbb Z_{\geq0}$.  Then $L(k,0)$ satisfies conditions $(\alpha),(\beta)$, and $(\gamma)$ in the introduction. (See the introduction of \cite{H Rigidity}, and the reference therein.) By  \cite{DL} theorem 4.7 or \cite{CKLW} proposition 5.17, $L(k,0)$ is a unitary VOA, and the PCT operator $\theta$ is determined by the fact that it fixes the vectors in $E$.  

Representations of $L(k,0)$ are determined by their restrictions to $\widehat{\mathfrak g}$. By \cite{Frenkel Zhu} theorem 3.1.3, irreducible $L(k,0)$-modules are precisely those which can be restricted to the $\widehat{\mathfrak g}$-modules of the form $L(k,\lambda)$, where $\lambda\in\mathfrak h^*$ satisfies condition \eqref{eq345}. By proposition 1.10, these $L(k,0)$-modules are unitary. Hence all $L(k,0)$-modules are unitarizable, and any set $\mathcal F$ of irreducible unitary $L(k,0)$-module satisfies condition \ref{CondA}-(a).

By proposition 3.6, $E\subset E^1(L(k,0))$. Since $E$ generates $L(k,0)$,  any  $\mathcal F$ also satisfies condition \ref{CondA}-(b).  Checking condition \ref{CondA}-(c) is much harder, and requires  case by case studies. Note that given the set $\mathcal F$, finding out which irreducible modules are inside $\mathcal F^\boxtimes$ requires the knowledge of fusion rules. A very practical way of calculating fusion rules for a unitary affine VOA is to calculate the dimensions of the spaces of primary fields.

\subsubsection*{Primary fields}
Fix $k\in\mathbb Z_{>0}$. For each $\lambda\in\mathfrak h^*$ satisfying condition \eqref{eq345}, we write $U_\lambda=L(\lambda),W_{\lambda}=L(k,\lambda)$. Let $\Delta_\lambda$ be the conformal dimension of the $L(k,0)$-module $W_\lambda$. We define the normalized energy operator  on $W_\lambda$ to be $D=L_0-\Delta_\lambda$.

Assume that $\lambda,\mu,\nu\in\mathfrak h^*$ satisfy condition \eqref{eq345}. We let $\Delta^\nu_{\lambda\mu}=\Delta_\lambda+\Delta_\mu-\Delta_\nu$. A \textbf{type $\nu\choose\lambda~\mu$ primary field} $\phi_\alpha$ is a linear map
\begin{gather*}
\phi_\alpha:U_\lambda\otimes W_\mu\rightarrow W_\nu[[x^{\pm1}]]x^{-\Delta^\nu_{\lambda\mu}},\\
u^{(\lambda)}\otimes w^{(\mu)}\mapsto \phi_\alpha(u^{(\lambda)},x)w^{(\mu)}=\sum_{n\in\mathbb Z}\phi_\alpha(u^{(\lambda)},n)w^{(\mu)}x^{-\Delta^\nu_{\lambda\mu}-n}\\
(\text{where }\phi_\alpha(u^{\lambda},n)\in\Hom(W_\mu,W_\nu)),
\end{gather*}
such that for any $u^{(\lambda)}\in U_\lambda,X\in\mathfrak g,m\in\mathbb Z$, we have
\begin{gather}
[X(m),\phi_\alpha(u^{(\lambda)},x)]=\phi_\alpha(Xu^{(\lambda)},x)x^m,\label{eq346}\\
[L_0,\phi_\alpha(u^{(\lambda)},x)]=\Big(x\frac d{dx}+\Delta_\lambda\Big)\phi_\alpha(u^{(\lambda)},x).\label{eq347}
\end{gather}
We say that $U_\lambda$ is the \textbf{charge space} of $\phi_\alpha$.

Note that the above two conditions are equivalent to that for any $m,n\in\mathbb Z,u^{(\lambda)}\in U_\lambda,X\in\mathfrak g$,
\begin{gather}
[X(m),\phi_\alpha(u^{(\lambda)},n)]=\phi_\alpha(Xu^{(\lambda)},n+m),\\
[D,\phi_\alpha(u^{(\lambda)},n)]=-n\phi_\alpha(u^{(\lambda)},n).
\end{gather}

Primary fields and intertwining operators are related in the following way: Let $\mathcal V_\pri{\nu\choose\lambda~\mu}$ be the vector space of type $\nu\choose\lambda~\mu$ primary fields. If $\mathcal Y_\alpha\in\mathcal V{\nu\choose\lambda~\mu}$ is a type $\nu\choose\lambda~\mu$ intertwining operator of $L(k,0)$, then by relation (1.26), for any $w^{(\lambda)}\in W_\lambda$ we have,
\begin{align}
\mathcal Y_\alpha(w^{(\lambda)},x)=x^{L_0}\mathcal Y_\alpha(x^{-L_0}w^{(\lambda)},1)x^{-L_0}\in\End(W_\mu,W_\nu)[[x^{\pm1}]]x^{-\Delta^\nu_{\lambda\mu}}
\end{align}
where $\mathcal Y_\alpha(\cdot,1)=\mathcal Y_\alpha(\cdot,x)\big|_{x=1}$. We define a linear map $\phi_\alpha:U_\lambda\otimes W_\mu\rightarrow W_\nu[[x^{\pm1}]]x^{-\Delta^\nu_{\lambda\mu}}$ to be the restriction of $\mathcal Y_\alpha$ to $U_\lambda\otimes W_\mu$. Then the Jacobi identity and the translation property for $\mathcal Y_\alpha$ implies that $\mathcal Y_\alpha$ satisfies equations \eqref{eq346} and \eqref{eq347}. Therefore, we have a linear map
\begin{gather}
\Phi:\mathcal V{\nu\choose\lambda~\mu}\rightarrow \mathcal V_\pri{\nu\choose\lambda~\mu},\qquad\mathcal Y_\alpha\mapsto\phi_\alpha.\label{eq348}
\end{gather}

The injectivity of $\Phi$ follows immediately from relation (1.22) or from corollary 2.15.  $\Phi$ is also surjective. Indeed, if we fix any $z\in\mathbb C^\times$ and define another linear map
\begin{gather*}
\Psi_z:\mathcal V_\pri{\nu\choose\lambda~\mu}\rightarrow (W_\lambda\otimes W_\mu\otimes W_{\overline{\nu}})^*,\nonumber\\
\phi_\alpha\mapsto\phi_\alpha(\cdot,z)=\phi_\alpha(\cdot,x)\big|_{x=z},
\end{gather*}
then by equation \eqref{eq347}, $\Psi_z$ is injective. By equation \eqref{eq346} and \cite{Ueno} theorem 3.18, the dimension  of the image  of $\Psi_z$ is no greater than that of ``the space of vacua" $\mathcal V^\dagger_{\mu\lambda\overline{\nu}}(\mathbb P^1;0,z,\infty)$ defined in \cite{TUY} and \cite{Ueno}. The later can be calculated using the Verlinde formula proved in \cite{Beauville Verlinde}, \cite{Faltings Verlinde}, and \cite{Teleman Verlinde}. The same Verlinde formula for $N^\nu_{\lambda\mu}$ proved in \cite{H Verlinde} shows that the dimension  of the vector space $\mathcal V{\nu\choose\lambda~\mu}$ (which is the fusion rule $N^\nu_{\lambda\mu}$) equals that of $\mathcal V^\dagger_{\mu\lambda\overline{\nu}}(\mathbb P^1;0,z,\infty)$. So $\dim\mathcal V_{\pri}{\nu\choose\lambda~\mu}\leq N^\nu_{\lambda\mu}$, and hence $\Phi$ must be surjective. We conclude the following:
\begin{pp}
The  linear map $\Phi$ defined in \eqref{eq348} is an isomorphism. In particular, the fusion rule $N^\nu_{\lambda\mu}$ of $L(k,0)$ equals the dimension of the vector space of type $\nu\choose\lambda~\mu$ primary fields of $L(k,0)$.
\end{pp}
\begin{thm}\label{lbb4}
Let $k=0,1,2,\dots$, and let $L(k,0)$ be the level $k$ unitary affine VOA associated to $\mathfrak g$. Then any $L(k,0)$-module is unitarizable. Suppose that $\mathcal F$ is a \emph{generating} set of irreducible unitary $L(k,0)$-modules (i.e., $\mathcal F^\boxtimes$ contains any unitary $L(k,0)$-module), and that for any $\lambda\in\mathcal F$, all primary fields of $L(k,0)$ with charge spaces $U_\lambda$ are energy-bounded.  Then $\Lambda$ is positive definite on the tensor product of any two $L(k,0)$-modules, and the modular tensor category $\Rep^\uni(L(k,0))$  is unitary.
\end{thm}

We now show that theorem \ref{lbb4} can be applied to the unitary affine VOAs of type $A_n$ and $D_n$.

\subsubsection*{The case $\mathfrak g=\mathfrak{sl}_n~(n\geq2)$}

Let $L(\Box)$ be the ($n$-dimensional) vector representation of $\mathfrak{sl}_n$, and let 
$$\mathcal F=\{L(k,\Box) \}.$$
In \cite{Wassermann}, especially in section 25, it was proved that if $\lambda=\Box$ and the weights $\mu,\nu$ of $\mathfrak {sl}_n$ satisfy condition \eqref{eq345}, then
\begin{align}
\dim\mathcal V_\pri{\nu\choose\lambda~\mu}=\dim\Big(\Hom_{\mathfrak g}\big(L(\lambda)\otimes L(\mu),L(\nu)\big)\Big).\label{eq349}
\end{align}
(Note that this relation is not true for general $L(\lambda)$.) Using this relation, one can show that $\mathcal F$ is generating. In the same section, it was  proved that any $\phi_\alpha\in\mathcal V_\pri{\nu\choose\Box~\mu}$ satisfies $0$-th order energy bounds.

\subsubsection*{The case $\mathfrak g=\mathfrak {so}_{2n}~(n\geq3)$}

Let $L(\Box)$ be the vector representation of $\mathfrak{so}_{2n}$, and let $L(s_+)$ and $L(s_-)$ be the two half-spin representations of $\mathfrak{so}_{2n}$. In \cite{Toledano} chapter IV, it was proved that if $\lambda$ equals $\Box$ or $s_{\pm}$, and the weights $\mu,\nu$ of $\mathfrak {so}_{2n}$ satisfy condition \eqref{eq345}, then relation \eqref{eq349} holds. This shows that the set 
$$\mathcal F=\{L(k,\Box),L(k,s_+),L(k,s_-) \}$$
is generating. By \cite{Toledano} theorem VI.3.1, any primary field whose charge space is $L(k,\Box),L(k,s_+)$, or $L(k,s_-)$ is energy-bounded.

We conclude the following.

\begin{thm}
Let $\mathfrak g$ be $\mathfrak {sl}_n~(n\geq2)$ or $\mathfrak{so}_{2n}~(n\geq3)$, let $k=0,1,2,\dots$, and let $L(k,0)$ be the unitary affine VOA associated to $\mathfrak g$. Then $\Lambda$ is positive definite on the tensor product of any two $L(k,0)$-modules, and the modular tensor category $\Rep^\uni(L(k,0))$ of the unitary representations of $L(k,0)$ is unitary. 
\end{thm}

\subsubsection*{Other examples}

As we see in theorem \ref{lbb4}, to finish proving the unitarity of the modular tensor categories associated to unitary affine VOAs, one has to show, for the remaining types, that a ``generating" set of primary fields are energy-bounded. The success in type $A_n$ and $D_n$ unitary WZW models, as well as in unitary minimal models, shows that achieving this goal is promising. Indeed, the main idea of proving the energy-boundedness of a primary field in \cite{Wassermann}, \cite{Loke}, and \cite{Toledano} is to embedd the original VOA $V$ in a larger (super) VOA $\widetilde V$, the energy-boundedness of the field operators of which is easy to show, and realize the primary field as the compression of the vertex operator or an energy-bounded intertwining operator of $\widetilde V$. This strategy can in fact be successfully carried out for the other classical Lie types ($B$ and $C$) and for type $G_2$ (cf. \cite{Gui}).  We hope that it will also work for the remaining exceptional types $E$ and $F_4$.

\subsection{Full conformal field theory with reflection positivity}\label{lbb5}

In this section, we give an interpretation of our  unitarity results from the perspective of full conformal field theory. In \cite{HK07}, Y.Z.Huang and L.Kong constructed a (genus 0)  full conformal field theory for $V$ called ``diagonal model". This construction relies on the non-degeneracy of a bilinear form  on each pair $\mathcal V{k\choose i~j}\otimes \mathcal V{\overline k\choose \overline i~\overline j}$, which follows from the rigidity of $\Rep(V)$. These bilinear forms $(\cdot,\cdot)$ are directly related to our sesquilinear forms $\Lambda(\cdot|\cdot)$ on each $\mathcal V{k\choose i~j}$:
\begin{align}
(\mathcal Y_\alpha,\mathcal Y_{\overline\beta})=\Lambda(\mathcal Y_\alpha|\mathcal Y_\beta)\qquad(\alpha,\beta\in\mathcal V{k\choose i~j}).
\end{align}
In light of this relation, we  sketch the construction of diagonal model in \cite{HK07} from a unitary point of view.

Let us assume that $V$ is unitary, all $V$-modules are unitarizable, and all transport matrices are positive definite. (This last condition holds for $V$ if there exists a generating set $\mathcal F$ of irreducible unitary $V$-modules satisfying condition \ref{CondA} of \ref{CondB} in section \ref{Condition ABC}.) We define a vector space 
\begin{align}
F=\bigoplus_{i\in\mathcal E}W_i\otimes W_{\overline i}.\label{eqb8}
\end{align}
Its algebraic completion is $\widehat F=\bigoplus_{i\in\mathcal E}\widehat W_i\otimes \widehat W_{\overline i}$.

For each $i,j,k\in\mathcal E$, we choose an orthonormal basis $\Theta^k_{ij}$ of $\mathcal V{k\choose i~j}$ under the inner product $\Lambda$. The full field operator $\mathbb Y$ is defined to be an $\End(F\otimes F,\widehat F)$-valued continuous function on $\mathbb C^\times$, such that for any $w^{(i)}_{L}\otimes \overline{w^{(i)}_R}\in W_i\otimes W_{\overline i}\subset F,w^{(j)}_{L}\otimes \overline{ w^{(j)}_R}\in W_j\otimes W_{\overline j}\subset F$,
\begin{align}
\mathbb Y(w^{(i)}_{L}\otimes \overline{w^{(i)}_R};z,\overline z)(w^{(j)}_{L}\otimes \overline{ w^{(j)}_R})=\sum_{k\in\mathcal E}\sum_{\alpha\in\Theta^k_{ij}}\mathcal Y_{\alpha}(w^{(i)}_L,z)w^{(j)}_L\otimes\mathcal Y_{\overline \alpha}(\overline {w^{(i)}_R},\overline z)\overline{w^{(j)}_R}.
\end{align}
Then $(F,\mathbb Y)$ is a full field algebra of $V$ satisfying certain important properties, including commutativity (\cite{HK07} proposition 1.5) and associativity (\cite{HK07} proposition 1.4). In fact, in our unitarity context, it is not hard for  the reader to check that these two properties are equivalent to the unitarity of braid matrices and fusion matrices respectively. $(F,\mathbb Y)$ also satisfies modular invariance (\cite{HK10} proposition 5.1), which is indeed equivalent (\cite{HK10} theorem 3.8) to the unitarity of the projective representation of $SL_2(\mathbb Z)$ on the vector space of the traces of the intertwining operators. This in turn is equivalent (\cite{HK10} theorem 4.11) to the unitarity of projective representation of $SL_2(\mathbb Z)$ in the unitary modular tensor category $\Rep^\uni(V)$ proved by \cite{Kir96} theorem 2.5.

Let us equip the vector space $F$ with an \emph{inner product} $\langle \cdot|\cdot\rangle$, such that the decomposition  \eqref{eqb8} is orthogonal, and for any $i\in\mathcal E,w^{(i)}_{L,1},w^{(i)}_{R,1},w^{(i)}_{L,2},w^{(i)}_{R,2}\in W_i$,
\begin{align}
\langle w^{(i)}_{L,1}\otimes \overline{w^{(i)}_{R,1}}|w^{(i)}_{L,2}\otimes \overline{w^{(i)}_{R,2}} \rangle=d_i^{-1}\langle w^{(i)}_{L,1}|w^{(i)}_{L,2} \rangle\langle w^{(i)}_{R,2}|w^{(i)}_{R,1} \rangle.\label{eqb9}
\end{align}
We also define an antilinear operator $\theta:F\rightarrow F$ sending each $w^{(i)}_L\otimes \overline {w^{(i)}_R}$ to $\overline {w^{(i)}_L}\otimes  {w^{(i)}_R}$, which is easily checked to be an anti-automorphism:
\begin{align}
\theta Y(w;z,\overline z)=Y(\theta w;\overline z, z)\theta\qquad(w\in F).
\end{align}
We call $\theta$ the \textbf{PCT} operator of $(F,\mathbb Y)$.

Note that when $V$ is non-unitary, we can only define a non-degenerate bilinear form on $F$, and show that under this bilinear form, the full field algebra $(F,\mathbb Y)$ satisfies the \emph{invariance property} (\cite{HK07} definition 3.9). But in our case, this invariance property should be  replaced by  the \textbf{reflection positivity}:
\begin{align}
\mathbb Y(w;z,\overline z)^\dagger=\mathbb Y\big(\theta \cdot e^{zL_1^L+\overline zL_1^R}(e^{-i\pi}z^{-2})^{L_0^L}(\overline {e^{-i\pi}z^{-2}})^{L_0^R} w;\overline {z^{-1}},z^{-1}\big)\qquad(w\in F),\label{eqb10}
\end{align}
where for each $n\in\mathbb Z,$, the linear operators $L_n^L=L_n\otimes 1,L_n^R=1\otimes L_n$ are defined on $F$. The factor $e^{-i\pi}$ in equation \eqref{eqb10} can be replaced by any $e^{i(2n+1)\pi}$, where $n\in\mathbb Z$. The reflection positivity is equivalent to the fact that for any $i,j,k\in\mathcal E,\mathcal Y_\alpha,\mathcal Y_\beta\in\mathcal V{k\choose i~j}$,
\begin{align}
\langle \mathcal Y_{C\alpha}|\mathcal Y_{C\beta} \rangle=\frac{d_k}{d_j}\langle \mathcal Y_\alpha|\mathcal Y_\beta \rangle.
\end{align}
This relation is essentially proved in \cite{HK07} using properties of the fusion matrices of intertwining operators. We remark that it can also be proved using graphical calculations for ribbon fusion categories.

A final remark: The positivity of $\Lambda$ is not used in full power to prove the reflection positivity of $F$. One only uses the positivity of quantum dimensions $d_i$'s and the fact that $\Lambda$ is Hermitian (i.e., $\Lambda(\mathcal Y_\alpha|\mathcal Y_\beta)=\overline{\Lambda(\mathcal Y_\beta|\mathcal Y_\alpha)}$), which can be checked more directly without doing long and tedious analysis as in our papers. So unlike the non-degeneracy of $\Lambda$, which is of significant importance in the  construction of the full field algebras of diagonal models, the positivity of $\Lambda$ only plays a marginal role. However, we expect that a systematic treatment of all full rational CFTs (but not just diagonal models) with reflection positivity will rely heavily on the unitarity of the MTC of $V$, and hence on the positivity of $\Lambda$.

\end{document}